\documentclass[preprint,3p]{amsart}
\usepackage[margin=1.2in]{geometry}
\usepackage{wrapfig}
\usepackage{amssymb}

\usepackage{amsmath}
\usepackage{amsfonts}
\usepackage{amsthm}
\usepackage{amssymb}
\usepackage{url}
\usepackage{color}
\usepackage{ulem}
\usepackage[pdftex]{graphicx}

\DeclareMathAlphabet{\mathpzc}{OT1}{pzc}{m}{it}

\newtheorem{theorem}{Theorem}[section]
\newtheorem{lemma}[theorem]{Lemma}

\newtheorem{remark}[theorem]{Remark}

\newtheorem{proposition}[theorem]{Proposition}

\newcommand{\RR}{\mathbb{R}}

\newcommand{\CC}{\mathbb{C}}

\newcommand{\MM}{\mathcal{M}}

\newcommand{\II}{\mathcal{I}}

\newcommand{\SSS}{\mathcal{S}}

\newcommand{\X}{X^+}
\newcommand{\Y}{X^-}

\newcommand{\OO}{\mathcal{O}}

\newcommand{\VV}{\mathcal{V}}

\newcommand{\QQQ}{\mathcal{Q}}

\newcommand{\eps}{\varepsilon}

\newcommand{\de}{\delta}
\newcommand{\la}{\lambda}
\newcommand{\al}{\alpha}
\newcommand{\f}{\pi^e}
\newcommand{\Q}{g}

\newcommand{\be}{\begin{equation}}
\newcommand{\ee}{\end{equation}}

\newcommand{\rdelta}{\sqrt{\delta}}

\begin{document}

\title{Two regularizations of the grazing-sliding bifurcation giving non equivalent dynamics}
%\author{Carles Bonet Rev\'{e}s\thanks{\tt carles.bonet@upc.edu} \ and  Tere M. Seara\thanks{\tt tere.m-seara@upc.edu}}

\author[C. Bonet Rev\'{e}s]{Carles Bonet Rev\'{e}s}
\address[CBR]{Departament de Matem\`{a}tiques, Universitat Polit\`{e}cnica de Catalunya, Diagonal 647, 08028 Barcelona, Spain }
\email{carles.bonet@upc.edu}

\author[T. M- Seara]{Tere M- Seara}
\address[TS]{Departament de Matem\`{a}tiques, Universitat Polit\`{e}cnica de Catalunya, Diagonal 647, 08028 Barcelona, Spain }
\email{tere.m-seara@upc.edu }

\begin{abstract}
We present two ways of regularizing a one  parameter family of piece-wise smooth dynamical systems undergoing a codimension one grazing-sliding global bifurcation of periodic orbits. First we use the Sotomayor-Teixeira regularization and prove that the regularized family has a saddle-node bifurcation of periodic orbits. Then we perform a hysteretic regularization and show that the regularized family has chaotic dynamics.  Our result shows that, in spite that the two regularizations will give the same dynamics in the sliding modes, when a tangency appears the hysteretic process generates chaotic dynamics.

\end{abstract}

\maketitle

\bigskip

\begin{enumerate}
\item[] \textbf{Keywords:} Regularization of Filippov Systems; grazing-sliding bifurcation; saddle-node bifurcation; hysteresis; chaotic behavior
\end{enumerate}

\section{Introduction}

Discontinuous dynamical systems model many phenomena in control theory, in mechanical friction and impacts, in hysteresis in electrical circuits and plasticity, etc. In these systems the phase space is divided into several regions where the system takes different forms. Vector fields with jump discontinuities at the edges of these regions -the switching manifolds- are usually named Filippov Systems. See \cite{bc08} for a deep overview.

One major example of Filippov systems is the so called sliding mode control (SMC) (see \cite{u92}). Roughly speaking a SMC is an application of a discontinuous control signal u that forces the solutions to reach the switching surface in finite time, and ``slide'' on it with a prescribed convenient flow. Obviously this procedure cannot be continuous as the switching manifold won’t be, in general, an invariant manifold of any differentiable system. For instance, this designed control can produce ``chattering'' around the switching manifold.

Then two main questions arise: How to define a solution on the switching manifold and how to regularize the discontinuous system. That is, how to unfold the Filippov system  in a parametric family of smooth  vector fields, in such a manner that their (singular) limit be consistent with the prescribed switching dynamics.

It is well known and largely discussed that there is not a "canonical" way of defining the dynamics on the switching manifolds \cite{u92, j18book, Bonet2021AgeingOA}, but the most commonly used formalism to define a flow on the switch derived from the fields outside the edges is due to Filippov \cite{Filippov88} and its application to control by Utkin \cite{u92}.

They essentially approximate chattering to-and-fro across a discontinuity by a steady flow precisely along the discontinuity. Whereas Filippov sliding dynamics convention describes a linear combination of vector fields at the edges, Utkin equivalent control describes a function depending of the control. The two methods are derived from different regularizations of the piece-wise systems in a neighborhood of the sliding regions of the switching manifold. While Filippov procedure can be seen as a limiting process of oscillations created by hysteresis or delay \cite{uber1956, uber1956II} Utkin justifies his definition of equivalent control by filtering and averaging the oscillations around the sliding modes \cite{u92}. The two  approaches coincide in case of linear dependence on the control, but not otherwise. See \cite{Kristiansen2020, j18book, u92,BonetSFG2017,Bonet2021AgeingOA}.

One of the most used differentiable regularization of a piecewise smooth dynamical systems is the so called regularization of Sotomayor-Teixeira \cite{SotoTei}. The piecewise smooth system is approximated in a thin boundary layer around the switch by a one parameter family of differentiable flows. It is well known that near any compact sliding region of the switching manifold there exists a differentiable normally hyperbolic invariant manifold of the regularized family and the flow inside this manifold is close to the sliding Filippov flow in the switch \cite{tls09}.
But the Sotomayor-Teixeira is not the only possible regularization of a Filippov system. In fact,  the justification of Filippov convention also is based on hysteresis. In \cite{BonetSFG2017} it is proved that the regularization by hysteresis in sliding compact regions of the switch also gives the Filippov's solutions in the limit.

In conclusion, both regularizations, the Sotomayor-Teixeira and the hysteretic one, give the Filippov flow as a limit in  compact sliding regions of the switching manifold.  In this paper we will prove that this is not the case when the hyperbolicity is lost, as happens, for instance, at  grazing bifurcations.

In \cite{BonetS16}  the Sotomayor-Teixeira regularization of a general visible fold singularity (also called visible tangency point) of a planar Filippov system was studied.
Extending Geometric Fenichel Theory the deviation of the orbits of the regularized system from the orbits of the Filippov one were determined.
 This result was used to understand the global dynamics of a regularized family of Filippov vector field having some global bifurcations, like  the grazing-sliding of periodic orbits or the Sliding Homoclinic to a Saddle.
Both bifurcations involve a tangency between the periodic (or homoclinic) orbit of one of the adjacent vector fields with the discontinuity manifold. 
Therefore, although we are studying a global phenomenon, its behavior relies on the local behavior of the regularized Filippov System near a so-called visible tangency point.

In case of the grazing-sliding bifurcation, if the periodic orbit is repelling, it was shown that the regularized family also has a bifurcation of periodic orbits. 
As the parameter crosses the bifurcation value, the system passes from having two periodic orbits to none. We presented numerical and heuristic evidences of the bifurcation value and that it was of saddle node type. We also indicated how to prove it rigorously through the convexity of a certain Poincar\'{e} map, but we leaved the detailed proof to a future work (See Remark 2 in \cite{BonetS16}).

The goal of this paper is twofold: On the one hand we give a rigorous proof of the saddle-node character of the bifurcation appearing in the Sotomayor-Teixeira regularization of the grazing-sliding bifurcation.
On the other hand, and  completing the results in \cite{BonetSFG2017}, we explore which is the effect of a hysteretic regularization in the grazing-sliding bifurcation
and we show that it gives rise to very different behavior.

Despite that in the sliding regions, the regularization by hysteresis also tends to the Filippov flow, we will see that, near the fold point, this regularization produces chaotic behavior of spiral type. As a consequence, in the grazing-sliding bifurcation thus regularized, does not appear one attracting periodic orbit (and the corresponding unstable), but an annulus with chaotic behavior instead.
In particular, this set contains infinitely unstable periodic orbits and also dense orbits. This kind of "noisy" behavior is also present in many chaotic circuits, like Chua and Alpazur circuits \cite{SharkovskyCh1993, Brown1992, Takuji1999, Naka1996}. We believe that this work contributes to explain the cause of the appearance of vibrations in sliding mode control systems.\cite{Naka1996}

This paper follows the notation and results of \cite{BonetS16, BonetSFG2017}.
In Section \ref{sec:ST} we prove that the Sotomayor-Teixeira regularization of a family of Filippov systems undergoing a grazing-sliding bifurcation is a saddle-node bifurcation.
This is achieved by searching the bifurcation value near the intersection of the vault of the periodic orbit with the Fenichel solution of the regularized system.
Using normal forms developed in \cite{BonetS16} we can bound the parameters where the bifurcation would be.
For these parameters a Poincar\'{e} map can be defined, and finally proved a convexity property.
Moreover, we provide an asymptotic expression for the bifurcation value and the semi-stable perodic orbit at this value. As often occurs in fold singularities, at last, all relies on the study of a Riccati equation.

In Section \ref{sec:H}, we regularize the grazing-sliding bifurcation by hysteresis as is defined in \cite{BonetSFG2017}. Also a Poincar\'{e} map can be defined, but now this map has discontinuities. Actually the map looks like an overlapping Lorenz map on the interval, a class of maps which are widely studied its chaotic and stochastic features \cite{Keener1980, GlendinningJ2019}. Then the dynamics of this map is analyzed. For values of the parameter less than the bifurcation value, the size of the points that goes to zero as the number of the iterates of the map goes to infinite is the total size of the interval. That is, the attractor between the periodic unstable orbits, attracts all, except a measure zero set. But at the bifurcation value, it appears a chaotic map, exactly the Baker-like map of \cite{GardiniM2015, Nordmark1997}, which has infinite discontinuity branches. After this value, the discontinuities are already in finite number and are disappearing, but the chaotic character remains. Even with a single discontinuity, a case that we prove with the usual methods of finding a horseshoe pair of subintervals \cite{GlendinningJ2019}

Section \ref{sec:proofs} is devoted to the more technical proofs of the results needed in the two precedent parts.

In the course of writing the present paper, the  work \cite{Kristiansen2020} has appeared,
where the author considers a different regularization given by an analytic function and proves the existence of a saddle-node bifurcation. There is no way to use his results in the problem studied in this paper. The author excludes from his study the Sotomayor-Teixeira regularization using instead a regularization with flat behavior at infinity, which modifies the original vector fields throughout the whole domain, while the Sotomayor-Teixeira regularization only modifies them in a small environment of the switching manifold.

Nor can it be argued that one regularization is more natural than the other. While flat at infinity functions are widely used in numerical simulations, it is also a fact that in sliding mode control theory, the fields outside the switching surface remain unchanged, and the control variable is confined to a finite range. Anyway, one can combine both results to conclude that the saddle node character is maintained for a large family of regularizations, but not for the hysteretic one.

\section{The Sotomayor-Teixeira regularization of the
grazing-sliding bifurcation}\label{sec:ST}

To settle properly the problem we follow closely \cite{BonetS16} and its basic notation, that is:

We consider a  Filippov  system in $\RR^2$:
\begin{equation}\label{def:Filippov}
Z(x,y)=\left\{\begin{array}{l}
        \X(x,y),\, (x,y)\in\VV^+\\
        \Y(x,y),\, (x,y)\in\VV^-,
       \end{array}\right.
\end{equation}
where: $\VV^+=\{(x,y)\in \VV, \ y> 0\}$,
$\VV^-=\{(x,y)\in \VV, \ y< 0\}$, where $\VV$ is an open set containing  the origin,  with a  switching manifold  given by:
$$
\Sigma=\{(x,y)\in \VV,\ y= 0\}.
$$
We assume that the vector fields $\X$ and $\Y$ have an extension to a neighborhood of $\Sigma$, at least, $\CC^{2}$. We denote their flows by $\phi_{\X}$ and  $\phi_{\Y}$ respectively.

We assume that the vector field $\Y$ is transverse to $\Sigma$ and that $\X$ has a generic fold in $\Sigma$, that is:
\begin{equation}\label{generalform}
\begin{array}{rcl}
\X(0,0) &=& (\X_1(0,0),0), \quad \X_1(0,0) \ne 0 ,\quad \frac{\partial \X_2}{\partial x}(0,0)\ne 0\\
\Y(0,0)&=&(\Y_1(0,0),\Y_2(0,0)), \quad  \Y_2(0,0)\ne 0 .
\end{array}
\end{equation}
Without loss of generality we can assume that the fold point is at $(0,0)$.

We will consider the case where:
\begin{equation}\label{cond:visiblefold}
\Y_2(0,0)>0, \mbox{and} \ \X_2(x,0) <0 \ \mbox{for}\ x<0, \   \X_2(x,0) >0\ \mbox{for}\ x>0.
\end{equation}
These conditions ensure that $(0,0)$ is a generic visible fold-regular point.
As $\X _1(0,0)\ne 0$, we will deal with the case
\begin{equation}\label{cond:visiblefold1}
\X_1(0,0)>0,
\end{equation}
which implies that $\X$ goes ``to the right".\\

Moreover, by Prop. $14$ in \cite{BonetS16}, we know that, after a smooth change of variables, we can
assume that $Z =(\X,\Y)$ has the form:
\begin{equation}\label{def:Xg}
\X(x,y)=\left(\begin{array}{l}
        1 + f_1(x,y)\\
        2x + by +f_2(x,y)
       \end{array}\right)
\end{equation}
where $f_i(x,y)=O_i(x,y)$ and $f_2(x,0)=0$, and
\begin{equation}\label{def:Yg}
\Y(x,y)=\left(\begin{array}{l}
        0\\
        1
       \end{array}\right).
\end{equation}

As in this paper we will work with the Sotomayor-Teixeira regularization $Z_\eps$ of the vector field $Z$ in \eqref{def:Filippov}, let us recall here its definition:
\begin{equation}\label{regularizedvf}
Z_\eps (x,y)= \frac{\X(x,y)+\Y(x,y)}{2}+\varphi(\frac{y}{\eps}) \frac{\X(x,y)-\Y(x,y)}{2},
\end{equation}
where  $\varphi$ is any increasing smooth ${\mathcal{C}}^{p-1}$ function with:
$$
\varphi(v) =-1, \ \mbox{for} \ v\le -1, \quad \varphi(v) =1, \ \mbox{for} \ v\ge 1.
$$
During this paper we will consider the case $p=2$ and therefore we consider  ${\mathcal{C}}^{1}$ regularizing functions.
The general case can be done analogously.

We introduce
\begin{equation}\label{eq:nueps}
\mathcal{V}_\eps=\{(x,y)\in \mathcal{V},|y|\le \eps\},
\end{equation}
the regularizing strip.
Is clear that outside $\mathcal{V}_\eps$, $Z_\eps=Z$.

\subsection{Previous results}

\begin{figure}
\begin{center}
\includegraphics[width=7.5cm,height=4cm]{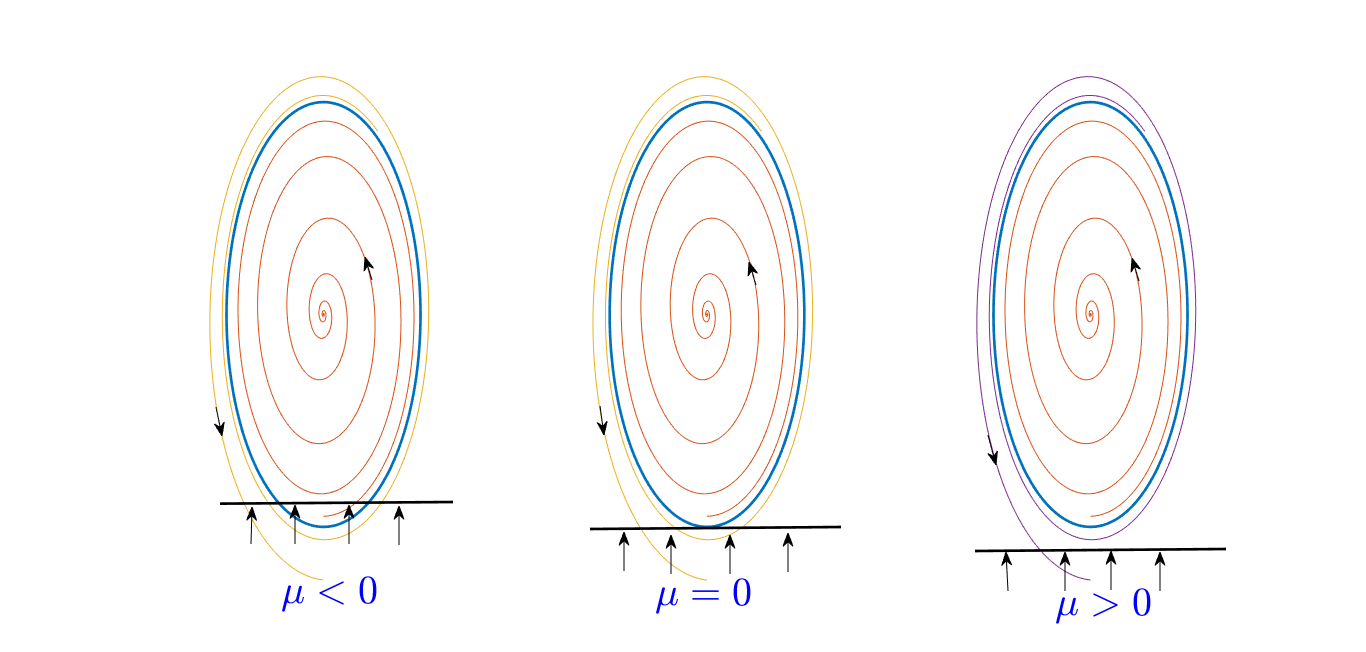}
\includegraphics[width=7.5cm,height=4cm]{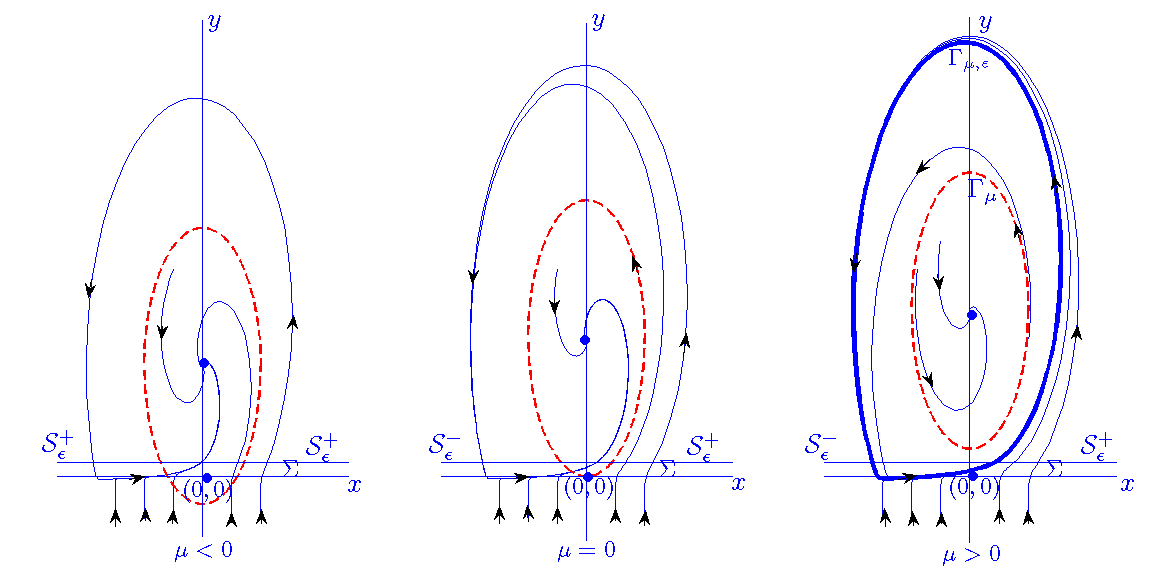}
\caption{On the left, the relative position of the repelling periodic orbit $\Gamma_\mu$ of  $ X^+_\mu$ for different values of $\mu$.
On the right the bifurcation of periodic orbits of the regularized vector field $Z_{\mu,\eps}$}\label{fig:bifurcacioantiga}
\end{center}
\end{figure}

The purpose of this section is to study how the Sotomayor-Teixeira  regularizations affects  a family of Filippov vector fields having a grazing-sliding bifurcation of periodic orbits.
That is, we consider   a family $Z_\mu$ of Filippov  planar systems undergoing a grazing-sliding bifurcation of a hyperbolic attracting or repelling periodic orbit $\Gamma_\mu \subset \VV$ of the vector field $\X_\mu$ at $\mu=0$.
Therefore the case $\mu=0$ corresponds to the case that $X_0$ has a periodic orbit $\Gamma_0$ tangent to $\Sigma$.

Next theorem, which is Theorem 2.4 in \cite{BonetS16},  gives some preliminary results of how these bifurcations behave in  the corresponding regularized family $Z_{\mu,\eps}$ (see Figure \ref{fig:bifurcacioantiga}).

\begin{theorem}[\cite{BonetS16}]\label{thm:sella-node}
Let  $Z_\mu$, $\mu \in \RR$ be a family of non-smooth planar systems that undergoes a grazing-sliding bifurcation
of a hyperbolic periodic orbit $\Gamma_\mu$ of the vector field $\X_\mu$ at $\mu=0$. We assume that, for $\mu>0$
the periodic orbit  $\Gamma_\mu$  is entirely contained in $\VV^+$, it becomes tangent to $\Sigma$ for $\mu=0$
and intersects both regions $\VV^\pm$ for $\mu<0$.

Consider the regularized family $Z_{\mu,\eps}$.
\begin{itemize}
\item
If $\Gamma_\mu$ is attracting,
the regularized system has a  periodic orbit $\Gamma_{\mu,\eps}$ for any
$\eps$, $\mu$ small enough.
No bifurcation occurs in the regularized system.
\item
If $\Gamma_\mu$ is repelling,
for any $\mu>0$ and $0<\eps<\eps _0(\mu)$, the regularized system has a  periodic orbit $\Gamma_{\mu,\eps}$  which co-exists with the periodic orbit $\Gamma_\mu$ contained in
$\VV^+\cap \{ (x,y), \ y>\eps\}$.
Moreover, there exists a constant $\Delta<0$ such that this result is also true for $\mu = \tilde \mu \eps$, if $\tilde \mu >-\Delta>0$.
For $\mu \le  0$ small enough, the system has no periodic orbits near $\Gamma_0$ if $\eps$ is small enough.
Therefore the family $Z_{\mu,\eps}$ undergoes a
bifurcation of periodic orbits near  $\mu=0$.
\end{itemize}
\end{theorem}

\begin{remark}\label{rem:Delta}
The constant $\Delta$ which appears in Theorem \ref{thm:sella-node} has an explicit formula in
\cite{BonetS16} that we don't reproduce here because it does  not play any role in the sequel.
What is  important is that:
\begin{itemize}
\item
$\Delta >0$ if the periodic orbit $\Gamma_0$ is attracting and $\Delta<0$ when it is repelling.
\item
The value of $\mu=|\Delta| \eps+\OO(\eps^2)$ corresponds to the case that the periodic orbit $\Gamma_\mu$ of $\X_\mu$ is tangent to the line $y=\eps$ and, therefore, is still a periodic orbit of the regularized family $Z_{\mu,\eps}$.
\end{itemize}
\end{remark}

The proof of Theorem \ref{thm:sella-node}  needs to match the behavior of $Z_{\mu, \eps}$ inside $\mathcal{V}_\eps$ with the one of $Z_\mu$ outside.
To this end, one considers several maps which give the dynamics near the periodic orbit $\Gamma_\mu$.

The main difficulty is the study of the map
\[
\mathcal{Q}_{\eps} : \SSS_\eps ^- \to \SSS _{\eps}^+
\]
where, given any $y_0$, the sections are defined as:				
\begin{equation}\label{eq:S0}							
\SSS ^-_{y_0} =\{ (x, y_0)  \in \VV, \ x\le 0\}, \quad \SSS ^+ _{y_0}=\{ (x, y_0) \in \VV, \ x\ge 0\}, \quad
\SSS_{y_0}=\SSS ^- _{y_0}\cup \SSS ^+ _{y_0},
\end{equation}
and the map $\mathcal{Q}_{\eps}$ is given by the orbits of the regularized vector field  between the sections $ \SSS_\eps ^- $ and $ \SSS _{\eps}^+$.

The study of this map  is performed in\cite{BonetS16} by using Fenichel theory and rigorous asymptotic methods. One  obtains that there exists  a solution, known as the Fenichel manifold, which attracts all the orbits in a neighborhood of $ \SSS_\eps ^- $.
More concretely, the Fenichel manifold intersects $\SSS^+_\eps$ in a point
\[
F=(\eps ^{ \frac{2}{3} } \eta_0 (0)+ \OO(\eps ),\eps)
\]
and one can prove that the map $\mathcal{Q}_{\eps}$ behaves as:
\begin{equation}\label{eq:Qepsilon}
\mathcal{Q}_{\eps} (x) =  \eps ^{ \frac{2}{3} } \eta _0(0)+ \OO(\eps ), \quad \forall x\in[-L,-\eps^{\lambda}), \ \end{equation}
where $0<\lambda<\frac23$ and $L>0$ is a constant independent of $\eps$.

Actually, $\eta_0(u)$ is the solution of the Ricatti equation associated to the following system:
\begin{equation}\label{Ricatti}
\begin{array}{rcl}
\dot \eta &=&1\\
\dot u &=&2\eta-\frac{\varphi''(1)}{4}u^2.
\end{array}
\end{equation}
satisfying
\[
\eta (u)-\frac{\varphi''(1)}{8}u^2 =\OO(\frac{1}{u}), \ u\to -\infty
\]

For the purposes of this work we also need the next proposition, which is Proposition 2 in \cite{BonetS16}.
It states that the flow of a Sotomayor-Teixeira regularized system $Z_{\eps}$ in $\mathcal{V}_\eps$ (see \eqref{eq:nueps}) of a Filippov system $Z=(X^+,X^-)$ is strictly bounded by the flow of $X^+$ in the regularization strip near a visible fold. More concretely. Let $ P_\eps^+$ and $\mathcal{Q}_{\eps}$ denote the Poincar\'{e} maps associated, respectively, with the flows of $X^+$ and $ Z_\eps$ on $ \SSS_\eps $ . Let be $(x_\eps, \eps)$ the point where these vector fields have a tangency on $ \SSS_\eps $. Let $[\bar{x},x_\eps]\times \{\eps\}\subset \SSS_\eps^-$, for fixed $\bar{x}$ but close to $x_\eps$ in order to guarantee the above maps are defined. Then we have:

Let $[\bar{x},x_\eps]\times \{\eps\}\subset \SSS_\eps^-$, for fixed $\bar{x}$ but close to $x_\eps$ in order to guarantee the above maps are defined. Then we have:
\begin{proposition}[\cite{BonetS16}]\label{prop:comparacio}
If $\eps>0$ is small enough then for any  $x\in [\bar{x}, x_\eps]$ one has that
$$
\mathcal{Q}_{\eps} (x)< P_{\eps}^+(x).
$$
\end{proposition}

\subsection{The saddle node bifurcation}\label{sec:sella-node}
Observe that Theorem \ref{thm:sella-node}
establishes, in the unstable case, and therefore when $\Delta<0$,  the existence of a bifurcation of periodic orbits for:
\begin{equation}\label{intbifurc}
0<\mu\le -\Delta \eps
\end{equation}
and the value $\mu=|\Delta| \eps+\OO(\eps^2)$ corresponds to the value where the periodic orbit $\Gamma_\mu$ of the upper vector field $\X_\mu$ is tangent to the line $\SSS_\eps$ (see Remark \ref{rem:Delta}) and therefore $\Gamma_\mu$ is still a periodic orbit of the vector field $Z_{\mu,\eps}$.

The purpose of this section is to prove next Theorem \ref{thm:sella-node-si} which  completes the results in Theorem \ref{thm:sella-node}  and states that there is only a bifurcation in this interval and this bifurcation is a  saddle-node bifurcation of periodic orbits.

\begin{theorem}\label{thm:sella-node-si}
With the same hypothesis of Theorem \ref{thm:sella-node}, if $\Gamma_0$ is repelling, the regularized vector field $Z_{\mu,\eps}$ has only a bifurcation and it is a saddle node bifurcation of periodic orbits at:
\begin{equation}\label{eq:mustar}
\mu^*=-\Delta \eps+\OO (\eps^{\frac{4}{3}})
\end{equation}
\end{theorem}
In the rest of this section we give the proof of Theorem \ref{thm:sella-node-si}.
In fact, we will provide a more detailed result in Theorem \ref{thm:main}, where we provide an asymptotic formula for the bifurcation value $\mu^*$, see \eqref{eq:muestrella},\eqref{eq:delta0}.
To proof Theorem \ref{thm:sella-node-si} we will construct the return map in a slightly different way as in \cite{BonetS16}. It will be crucial to improve the knowledge of the behavior of the  map $\mathcal{Q}_\eps $ given in \eqref{eq:Qepsilon}, which is  mainly determined by the Fenichel manifold that,  by Proposition 8 in \cite{BonetS16},  exponentially attracts the points of the segment
 $[-L, -\eps ^\la]\times \{\eps\}\subset \SSS_\eps^-,\quad 0<\lambda<\frac{2}{3} $.
In order to avoid technicalities that can hide the essential facts, during this proof and without loss of generality we assume the following hypothesis:
\begin{itemize}
\item
The vector field  $\X_0$ is defined in $\RR^2$ and has a unique repelling periodic orbit $\Gamma_0$ entirely contained in $\VV^+$ except the point $(0,0)$ which is a (visible) fold of $\X_0$ and it is of the form \eqref{def:Xg}. We also assume
that there is a unique  attracting focus inside $\Gamma_0$.
\item
As a consequence of the fact that $\Gamma_0$ is repelling, the Poincar\'{e} map
\begin{equation}\label{eq:poincaremap}
\pi: \{(0,y)\} \to \{(0,y)\}
\end{equation}
is defined locally in the $y$ axis in a neighborhood of $y=0$ and  fulfills $\pi(0)=0$ and $\pi '(0)>1$.
\item
The family $\X_\mu(x,y)$ is given by:
\begin{equation}\label{hypofamily}
\X_\mu(x,y)=\X_0(x,y-\mu).
\end{equation}
This assumption gives that $\X_\mu$  consists on slipping on the $y$ axis the  vector field $\X_0$.
\item
We denote by $\Gamma_\mu$ the periodic orbit of $\X_\mu$ which, by construction, is tangent to $\SSS_\mu$ at the point $(0,\mu)$.
Consequently, for $\mu=\eps$ the periodic orbit $\Gamma_\eps$ is tangent to $\SSS_\eps$ which implies two important facts (see Remark \ref{rem:Delta}): on the one hand the parameter $\Delta$ in Theorem \ref{thm:sella-node} is $\Delta=-1$ and, in the other hand, when $\mu=\eps$, $\Gamma_\eps$ is still a periodic orbit of the regularized system $Z_{\mu,\eps}$ .
\item
We also denote by $(x_\mu,\epsilon)$ the point whose orbit through $\X_\mu$ is tangent to $\SSS_{\eps} $, and by $(\tilde{x}_\mu,\epsilon)$ the first cut of the negative orbit of $(x_\mu,\eps)$ with $\SSS_{\eps} $.(see Figure \ref{fig:xmuxtildemu}).
%}
\item
We take as $\Y_\mu=(0,1)$
\end{itemize}
The hypothesis of the vectors fields  $\X_\mu$ and $\Y_\mu $ jointly with Proposition \ref{prop:comparacio} give that, for $0\le\mu \le\eps$ small enough, all the solutions of the regularized system  $Z_{\mu,\eps}$ departing from any point in the vault of $\Gamma_\mu$ contained $\VV^+\cap \{ (x,y), \ y\ge\eps\}$ are trapped by the focus of $\X_\mu$. In fact it is enough to prove the next proposition
\begin{proposition}\label{prop:trapping}
Let $0\le\mu \le\eps$ small enough and let $ (x_\mu^\pm,\eps)=\Gamma_\mu \cap \SSS_{\eps} ^\pm$. Then the solution of the regularized system $Z_{\mu,\eps}$ departing from $ (x,\eps)$, where $x\in [x_\mu^-,x_\mu^+]$, is trapped by the focus of $\X_\mu$.
\end{proposition}
\begin{figure}
\begin{center}
\includegraphics[width=9cm,height=4cm]{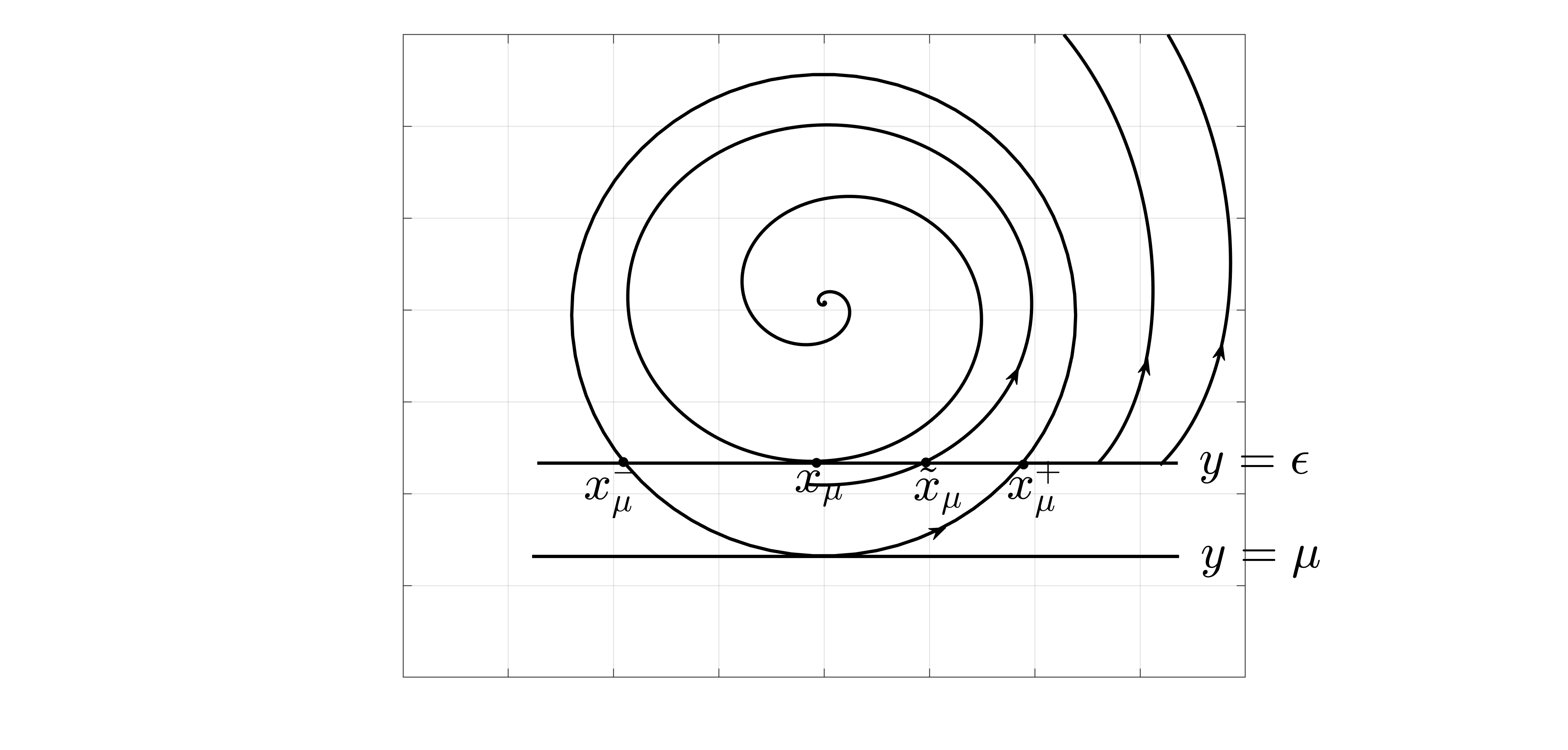}
\caption{The significative points of the intersection of the flow of $X_\mu^+$ with $\SSS_{\eps} $.}\label{fig:xmuxtildemu}
\end{center}
\end{figure}

\begin{proof}
By Proposition \ref{prop:comparacio}, the flow of $X_\mu^+$ strictly  minorizes  the flow of $Z_{\mu,\eps}$ inside $\mathcal{V}_\eps$.
The properties of $X_0$ necessarily imply that the point
$(x_\mu,\eps)\in
\SSS_{\eps} $ will be trapped by the focus.
So a trapping subinterval $[x_\mu, \tilde x_\mu]$ is determined.
Take any point $(x,\eps)$ with $x \in  [x_\mu^-,x_\mu^+]$
outside this interval, for instance, $x\ge \tilde x_\mu$.
Proposition \ref{prop:comparacio} also implies that the orbit of $(x,\eps)$ hits $\SSS_{\eps}^+ $ in a point $(x_1,\eps)$ with $x_\mu<x_1<\tilde x_\mu$.
In this way, the orbit of $(x,\eps)$ follows an spiraling process in concordance with the hypothesis for $X_0$. If any iterates never enters the trapping subinterval, then a periodic orbit of the regularized system  will be determined. But, applying again Proposition \ref{prop:comparacio}, such orbit can not exist.
\end{proof}

\begin{remark}\label{rem:xtitllamu}
The distance between the points $(x_\mu,\epsilon) $ and $ (\tilde x_\mu,\epsilon) $ will give a geometrical view of the bifurcation.(see Figure \ref{fig:mecanismebif1})
\end{remark}

Taking into account the above hypothesis and Proposition \ref{prop:trapping}
in the next proposition we extend  the results in Theorem \ref{thm:sella-node} applied to the regularized family $Z_{\mu,\eps}$ to obtain  that:

\begin{proposition}\label{prop:flowtangency}
In the above hypothesis, we have:
\begin{enumerate}
\item
For $ \mu=\eps $ the regularized  vector field $Z_{\eps,\eps}$, besides  the unstable periodic orbit  $\Gamma_\eps$ which is tangent to $y=\mu=\eps$,  has, at least, another periodic orbit which is attracting.
\item
For $\mu=0$ the regularized field $Z_{0,\eps}$ has no periodic orbits.
\end{enumerate}
\end{proposition}
\begin{proof}
\begin{itemize}
\item
For $\mu=\eps$ the regularized flow inside the regularization strip $\mathcal{V}\eps $ defined in \eqref{eq:nueps} only can exit through $\SSS_\eps^+$.
Moreover, by \eqref{eq:Qepsilon}, the regularized flow sends the whole interval
$[-L, -\eps^{\lambda}]\times \{\eps\}\subset \SSS_\eps^-,\quad \forall\quad 0<\lambda<\frac{2}{3} \ $
to
$ \  \OO(\eps^{\frac{2}{3}})\times\{\eps\} \subset \SSS_\eps^+$.
But $\Gamma_\eps$ is tangent to $y=\eps$ at $ x=0$,
then the flow of $X_\mu(x,y)$ returns it to $x<0$, and so on.
Then a spiraling process take place around the periodic orbit, $\Gamma_\eps $, and because of its instability and two-dimensional topological reasons, at least one attracting periodic orbit must exist.
\item
For $\mu=0$ one can see that, if $\eps$ is small enough,
$[-L, 0]\times \{\eps\}\subset \SSS_\eps^-$ is trapped by
the focus.

The reason is again that the regularized flow sends the whole interval
$[-L, -\eps^{\lambda}]\times \{\eps\}\subset \SSS_\eps^-
$
to
$ \OO(\eps^{\frac{2}{3}})\times\{\eps\}\subset \SSS_\eps^+ $.
But $\Gamma_0\cap\SSS^+_\eps=(x_0^+,\eps)$ with $ x_0^+=\OO(\sqrt{\eps}) $, then the regularized flow enters inside of its vault and is trapped by the attracting focus by Proposition \ref{prop:trapping}.
For the  points $(x,\eps)$ with  $x \in (-\eps^\lambda,0)$, we can take $\frac{1}{2}<\lambda<\frac{2}{3}$, and diminish $\eps$ if needed to achieve that they are already in the vault of $\Gamma_0$ and are also trapped by Proposition \ref{prop:trapping}.
\end{itemize}
\end{proof}

\begin{remark}\label{rem:bifurcacio}

From Proposition \ref{prop:flowtangency} we have two consequences:
\begin{itemize}
\item
A bifurcation will take place for $0<\mu<\eps$, say, at $\mu=\mu^*$. Therefore, from now on, we assume that $\mu$ is inside this range although we will refine it later in Theorem \ref{thm:main}.
\item
We expect the bifurcation occurs when the Fenichel solution(s) and the upper segment of the periodic orbit  $\Gamma_\mu$ "collide" in  $\SSS^+_\eps$
at some order.
Define the following parameter, that will play a role in the rest of this section:
\begin{equation}\label{eq:defdelta}
\delta=\eps-\mu.
\end{equation}
\end{itemize}
Observe that, in the range of $\mu$ considered:
\[
\de=\de(\eps)>0, \ \mbox{and}\ \lim_{\eps\to 0} \de(\eps)=0
\]
\end{remark}

\begin{remark}\label{rem:xdelta+-}
Note that, for $\mu $ small enough, the tangency of $\Gamma_{\mu}$ at $ \SSS_\mu$ is a fold at $(0,\mu)$.
Moreover, under our normalizations, we have that $\X_\mu(x,y)=\X_0(x,y-\mu)$, therefore, the intersection of the periodic orbit $\Gamma_\mu$ of  $\X_\mu$ with  $\SSS_\eps ^\pm$ has the same $x-$coordinate that the intersection of the periodic orbit $\Gamma_0$ of  $\X_0$ with  $\SSS_\de ^\pm$, with $\de$ in \eqref{eq:defdelta}.
Consequently:
\begin{equation}\label{eq:xdelta+-}
\Gamma_{\mu}\cap \SSS_\eps ^\pm=(x_\mu^\pm, \eps), \ \ x_\mu^\pm=\pm \rdelta+\OO(\delta), \quad \de=\eps-\mu .
\end{equation}
\end{remark}

In view of the previous considerations, heuristically,  at the bifurcation value $\mu=\mu^*$,  the point $(x_\mu^+,\eps)$ given in \eqref{eq:xdelta+-} has to match with that of Fenichel whose $x$-coordinate  is  $\OO(\eps^\frac{2}{3})$ (see \eqref{eq:Qepsilon}).
Then $\delta^*:=\eps -\mu^*=\OO(\eps^\frac{4}{3})$
and the bifurcation must be searched at $\mu^*=\eps-K\eps^\frac{4}{3}$.
Later, in Theorem \ref{thm:main} we will provide a rigorous computation of the asymptotic value of $K=\delta_0^*+\OO(\eps^{1/3})$ (see \eqref{eq:delta0}), where the value $\delta_0^*$ will be related to the Ricatti equation \eqref{Ricatti}.

We  consider the  map \eqref{eq:Qepsilon}, that in our case will also depend on $\mu$, in its whole  domain:
\begin{equation}%\label{exteriornou}
\mathcal{Q}_{\mu,\eps}:
[-L,x_\mu]\times \{\eps\}  \subset  \SSS^-_\eps \to \SSS_\eps^+
\end{equation}
where $x_\mu$ is the $x$-coordinate of the tangency point of $\X_\mu=(X_{\mu,1}^+,X_{\mu,2}^+)$ with $\SSS_\eps$
\begin{equation}\label{eq:xepsilon}
\X_{\mu,2}(x_\mu,\eps)=0, \ x_\mu=\OO(\delta)=\OO(\eps-\mu)
\end{equation}
and  the returning exterior map derived by the flow of $X_\mu^+$:
\begin{equation}\label{exteriornou}
\f_{\mu,\eps}:  \MM \times \{\eps\} \subset  \SSS^+_\eps \to \SSS_\eps^-
\end{equation}
where $\MM$ is a suitable domain that will be defined later.
Our objective is to select a range of $\mu $ values
for which a Poincar\'{e} map
$\f_{\mu,\eps}\circ\mathcal{Q}_{\mu,\eps}$ can be defined on an interval
$\mathcal{J}\subset [-L, x_\mu]$, which contains the intersection of the possible periodic orbits of $Z_{\mu,\eps}$ with $\SSS_\eps^-$
and see that the map $\f_{\mu,\eps}\circ\mathcal{Q}_{\mu,\eps}$ is convex in this interval. More concretely we will prove the following theorem, which immediately  implies Theorem  \ref{thm:sella-node-si}.

\begin{figure}
\begin{center}
\includegraphics[width=10cm,height=4cm]{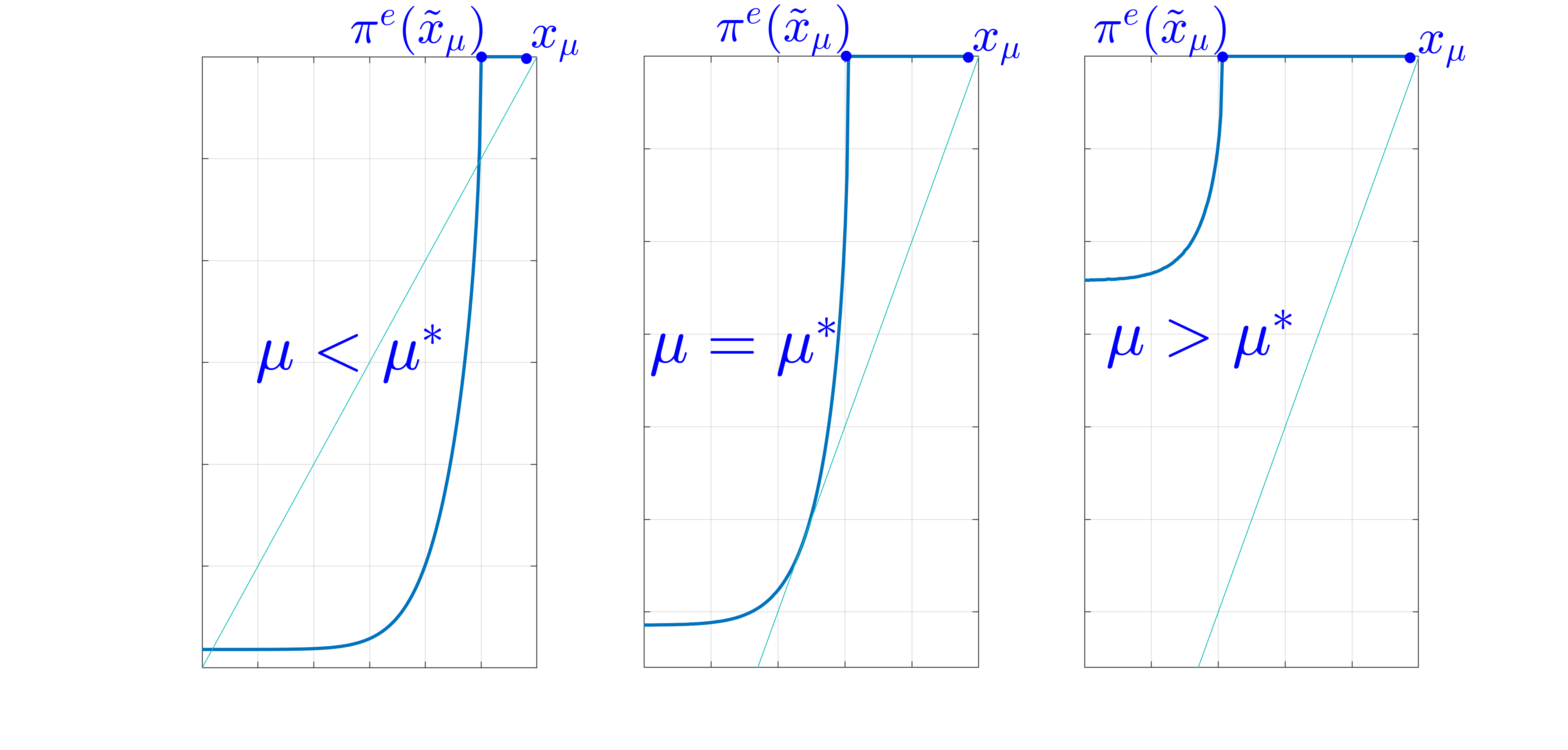}
\caption{In this picture the graphic of the Poincar\'{e} map, $ \f_{\mu,\eps} \circ \QQQ_{\mu,\eps} $, is depicted for values of $\mu <,=,> \mu^{*}$, the bifurcation value. One can see the dependence of the bifurcation on the distance between $\tilde{x}_\mu $ ( in fact $ \f_{\mu,\eps}(\tilde{x}_\mu )$ ) and $x_\mu$. (see definitions in \ref{hypofamily} and Remark \ref{rem:xtitllamu}) }\label{fig:mecanismebif1}
\end{center}

\end{figure}

\begin{theorem}\label{thm:main}
There exist two constants $K_1>0$, $K_2>0$, such that, for $\eps$ small enough, if we consider the values:
\begin{equation}\label{rangmuthm}
\begin{array}{rcl}
\mu_1:&=&\eps-\eps^{\frac{4}{3}}\eta_0^2(0)-K_1\eps^{\frac{5}{3}}\\
\mu_2:&=&\eps- K_2^2\eps^{\frac{4}{3}}
\end{array}
\end{equation}
the map $\f_{\mu,\eps}\circ\mathcal{Q}_{\mu,\eps}$ is smooth and satisfies:
\begin{itemize}
\item
If $\mu \le \mu_1$ has no fixed points.
\item
If $\mu \ge \mu_2$ has two fixed points.
\item
For $\mu \in (\mu_1,\mu_2)$
there exists an interval $\mathcal {J} =[-M \eps ^\frac23,-\overline M \eps ^\frac23]$, where $M>0$ and
$\overline M>0$ are constants independent of $\eps$ and $\mu$,
such that:
\begin{itemize}
\item
The fixed points of $\f_{\mu,\eps}\circ\mathcal{Q}_{\mu,\eps}$, if exist, belong to $\mathring{\mathcal J}$
\item
$\f_{\mu,\eps}\circ\mathcal{Q}_{\mu,\eps}$ is convex in $\mathcal{J}$.
\end{itemize}
\end{itemize}
Consequently the map $\f_{\mu,\eps}\circ\mathcal{Q}_{\mu,\eps}$ has only a bifurcation in $(\mu_1,\mu_2)$ and  is a saddle-node.

Moreover, if $\tilde{\mathcal{Q}_0}$ denotes the map
\[
\begin{split}
\tilde{\mathcal{Q}_0} : \SSS_0 ^- &\to \SSS _0^+
\\
(\eta,0) &\mapsto (\tilde{\mathcal{Q}_0}(\eta),0)
\end{split}
\]
derived from system \eqref{Ricatti}, and we denote by  $\eta_0^*<0$  the unique solution
of the equation:
\[
\pi  '(0)\tilde{\mathcal{Q}_0}(\eta)\tilde{\mathcal{Q}_0}'(\eta)=\eta
\]
where $\pi $ is the Poincar\'{e} map \eqref{eq:poincaremap}
and
$\delta_0^*$ is the value:
\begin{equation}\label{eq:delta0}
\delta_0^*=\frac{\pi '(0)\tilde{\mathcal{Q}^2_0}(\eta_0^*)-(\eta_0^*)^{2}}{\pi'(0)-1}
\end{equation}
then
\begin{itemize}
\item
the bifurcation takes place at the parameter value
\begin{equation}\label{eq:muestrella}
 \mu^*=\eps-\delta_0^*\eps^{\frac{4}{3}}+\OO(\eps^{\frac{5}{3}})
 \end{equation}
\item
the unique fixed point of the map $\f_{\mu^*,\eps}\circ\mathcal{Q}_{\mu^*,\eps}$ is at
$x^*=\eta_0^*\eps^{\frac{2}{3}}+\OO(\eps)$
\end{itemize}
%}
\end{theorem}

The rest of the section is devoted to prove the three first items of theorem \ref{thm:main}. 
The proof of the last two is deferred to Section \ref{sec:last2}.

\subsection{The exterior map $\f_{\mu,\eps}$ }

In this section we study the properties of the map $\f_{\mu,\eps}$ in \eqref{exteriornou} derived from the flow of $X_\mu^+$. We recall that we will perform this study for the range of $\mu \in (0,\eps)$ where the bifurcation $\mu=\mu^*$ takes place.

Recall that
$ (\tilde{x}_\mu, \eps)$ is the last cut of the solution through the tangency point $(x_\mu,\eps)$ by $\X_\mu$ (in backward time) with $\SSS_\eps^+$ (see \ref{hypofamily}).
%(see Figure \ref{fig:tangenciapexterior}).

Then, $\f_{\mu,\eps}$ is defined in $\MM\times \{\eps\}$, where
\begin{equation}\label{eq:dominipie}
\MM=[\tilde{x}_\mu,M],
\end{equation}
for some $M>0$ independent of  $\eps$, and $\f_{\mu,\eps}(\tilde x_\mu)=x_\mu$.
In fact, for a fully understanding of the bifurcation mechanism between $ \mu=0 $ and $\mu=\eps$, we will extend the map $\f_{\mu,\eps}$ to the interval $[x_\mu,\tilde x_\mu]$ by (see Figure \ref{fig:extensiope})
\[
\f_{\mu,\eps}(x)=x_\mu, \ \forall x \in [x_\mu,\tilde x_\mu]
\]

\begin{figure}
\begin{center}
\includegraphics[width=13cm,height=7cm]{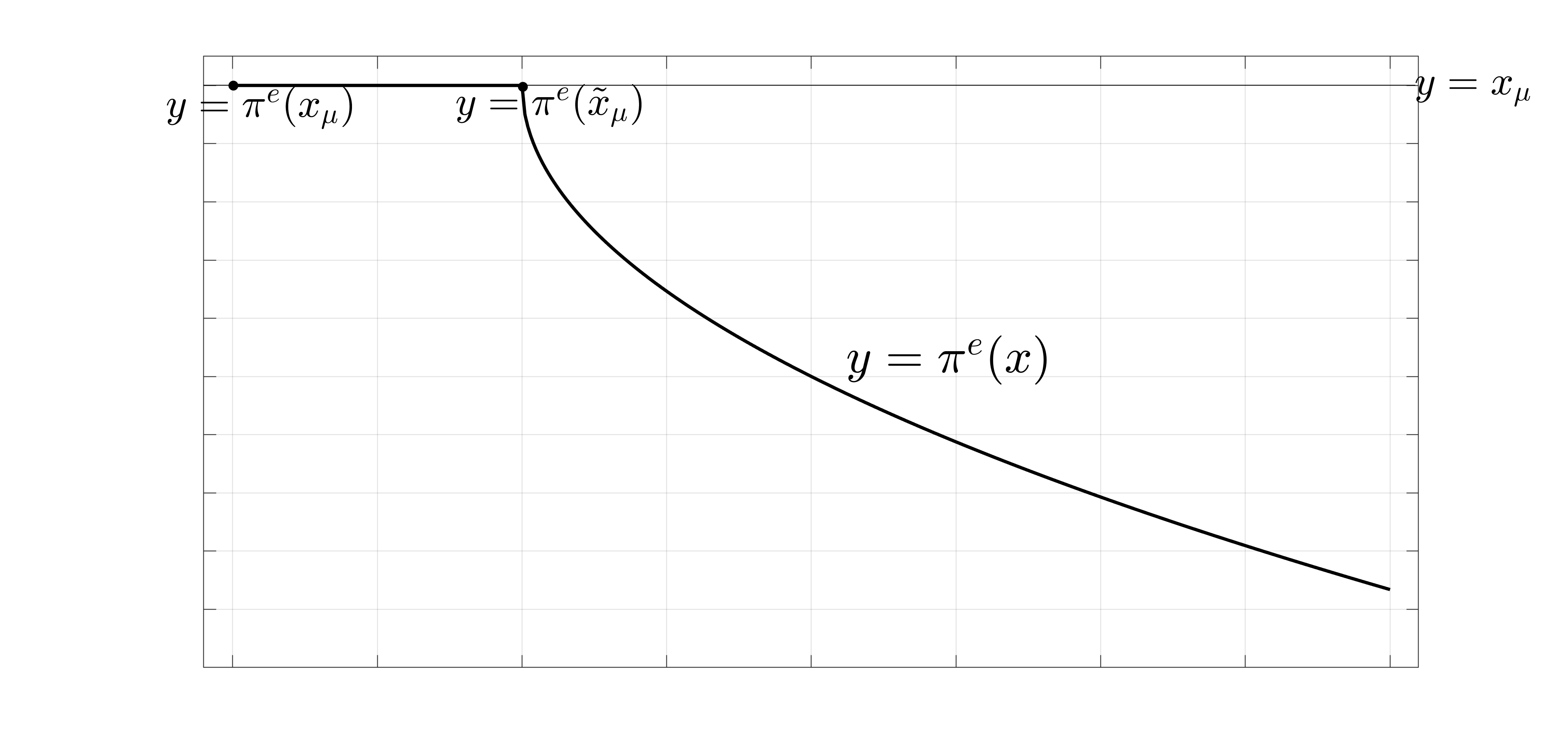}
\caption{The extended map $\f$.}\label{fig:extensiope}
\end{center}
\end{figure}

Next theorem gives the convexity properties of the map $\f_{\mu,\eps}$:
\begin{theorem}\label{thm:extensiopie}
The map
\[
\f_{\mu,\eps} :\MM=[ x_\mu, M]\times \{\eps \}\subset \SSS_\eps^+ \to \SSS^-_\eps
\]
satisfies:
\begin{itemize}
\item
$\f_{\mu,\eps}(x)=x_\mu, \ \forall x \in [x_\mu,\tilde x_\mu]$
\item
The points $x_\mu$ and $\tilde x_\mu$ are given by:
\begin{equation}\label{eq:punts}
x_\mu= -\frac{b}{2}\de+\OO(\de^2)\quad \tilde x_\mu =\sqrt{\de}\sqrt{1-\frac{1}{\pi'(0)}}+\OO(\de), \quad \de=\eps-\mu
\end{equation}
\item
Fix any constants $C>1$, $0<\bar{y}_0<\frac{1}{\pi'(0)}<1$ and
$\sqrt{1-\frac{1}{\pi'(0)}} < \sqrt{1-\bar y_0} <\sigma <1$
 where $\pi$ is the Poincar\'{e} map defined in \eqref{eq:poincaremap}.
Then we have if $\de=\eps-\mu>0$ is small enough:
\begin{enumerate}
\item
$\tilde x_\mu <\sqrt{\delta}\sqrt{1-\bar{y}_0}<\sigma \sqrt{\de} <\sqrt{\de C}$
\item
For $x \in [\sqrt{\delta}\sqrt{1-\bar{y}_0},\sqrt{\delta C}]$:
\begin{equation}\label{eq:formulapiebonathm}
\begin{array}{rcl}
\f_{\mu,\eps}(x)  &=&
-\sqrt{\delta- \pi'(0)(\delta-x^2)}+\OO(\delta), \\
(\f_{\mu,\eps})'(x)  &=&
-\frac{\pi'(0)x}{\sqrt{\delta-\pi'(0)(\delta-x^2)}}+\OO(\sqrt{\delta}) \\
(\f_{\mu,\eps})''(x) &=&
\frac{-\de \pi'(0)(1-\pi'(0))}{\sqrt{(\delta- \pi'(0)(\de-x^2))^3}  }+\OO(1) >0
\end{array}
\end{equation}
\item
For $x\in [\tilde{x}_\mu ,\sigma \sqrt{\de}]$:
\begin{equation}
\begin{split}
\f_{\mu,\eps}(x)-x_\mu&=\OO(\sqrt{x-\tilde{x}_\mu})<0\\
(\f_{\mu,\eps})''(x)&=\OO((x-\tilde{x}_\mu))^{-\frac32}>0
\end{split}
\end{equation}
\end{enumerate}
\end{itemize}
Consequently:
\begin{equation}
(\f_{\mu,\eps})''(x)>0, \forall x\in[\tilde x_\mu,\sqrt{\delta C}]
\end{equation}
\end{theorem}

\subsection{The inner map $\mathcal{Q}_{\mu,\eps} $ }
In this section we will study the map $\mathcal{Q}_{\mu,\eps} $ from $\SSS_{\eps}^-$ to $\SSS_{\eps}^+$, given by the orbits of the  regularized field $Z_{\mu,\eps}$, in the strip $\mathcal{V}_\eps$ for $0<\mu<\eps$.
We already know that its domain is defined on the left of the $x$-coordinate of the  tangency point $(x_\mu,\eps)$, but we need asymptotic formulas for it.

An important observation is that the interval
$[\QQQ_{\mu,\eps} ^{(-1)}(\tilde x_\mu), x_\mu]$ is mapped by $\QQQ_{\mu,\eps} $ to $[x_\mu, \tilde x_\mu]$ and we already know that this interval has no image through $\f_{\mu,\eps}$ (even if we have defined $\f_{\mu,\eps}$ as a constant function for convenience).
In particular, the fixed point of $\f_{\mu,\eps}\circ \QQQ_{\mu,\eps} $ will not belong to this interval.
Therefore we only need to study the map $\QQQ_{\mu,\eps} $ outside this interval, that is, away from the tangency $x_\mu$.

As a first step, next lemma shows that the Fenichel solution of the vector field  $Z_{\mu,\eps}$ intersects $\SSS_{\eps}^+$ in a point $(F_\mu ,\eps)$ whose first order is independent of $\mu$ if
$0<\mu<\eps$.

\begin{lemma}\label{lem:fenichelindependent}
Take $0<\mu<\eps$ and denote by $(F_\mu,\eps) $, the cut of the Fenichel solution of the vector field  $Z_{\mu,\eps}$ with $\SSS_{\eps}^+$.
Then we have:
\begin{equation}\label{tallfenichel}
F_\mu  =\eps^{\frac{2}{3}}\eta_0(0)+ \OO(\eps)
\end{equation}
where $\eta_0 (0)$ is given in \eqref{eq:Qepsilon}.
\end{lemma}

\begin{proof}
Calling $\alpha= \frac{\mu }{ \eps}$, we know that $0<\alpha < 1$.
Taking account of definition (\ref{hypofamily}), and the normal form \eqref{def:Xg}, the Sotomayor-Teixeira regularization $Z_{\mu,\eps}$ in the variables
$(x,v=\frac{y}{\eps})$, will be
\begin{equation}\label{eq:fastsystem}
\begin{array}{rcl}
\dot x &=&
\frac{1+\varphi (v)}{2}(1+f_1(x,\eps v-\mu))
=\frac{1+\varphi (v)}{2}(1+f_1(x,\eps (v-\alpha)))\\
\eps \dot v &=&
 \frac{1+2x }{2} +\frac{1}{2}\varphi (v)(2x-1 ) +
\frac{1+\varphi(v)}{2}(b(\eps v-\mu)+  f_2(x,\eps v-\mu))\\
&=&
 \frac{1+2x }{2} +\frac{1}{2}\varphi (v)(2x-1 ) +
\frac{1+\varphi(v)}{2}(b\eps( v-\alpha))+  f_2(x,\eps (v-\alpha)))
\
\end{array}
\end{equation}
Expanding in $\eps$ we obtain:
\begin{equation}\label{eq:fastgg}
\begin{array}{rcl}
\dot x &=&
\frac{1+\varphi (v)}{2}(1+f_1(x,0))+\eps \frac{1+\varphi (v)}{2}(v-\alpha)\frac{\partial f_1(x,0)}{\partial y}+...\\
\eps \dot v &=&
 \frac{1+2x }{2} +\frac{1}{2}\varphi (v)(2x-1 ) +
\eps\frac{1+\varphi(v)}{2}(v-\alpha)(b+\frac{\partial f_2(x,0)}{\partial y})+...
\end{array}
\end{equation}
where the dots $...$ indicate terms of superior order than $\eps$.

As $0<\alpha<1$ the fields $ X_\mu^+$ are identical at order zero in $\eps$, therefore
their Fenichel manifolds
have the same expression till order $\eps$ :
$$
x = n(v;\eps)= n_0(v)+\OO(\eps)
$$
with
 $$
 n_0(v) = \frac{1}{2}\frac{\varphi(v)-1}{\varphi(v)+1}
 $$
Moreover, if we denote by $F_\mu$, the cut of the Fenichel solution with $\SSS_\eps^+$, then (see \cite{BonetS16})
\begin{equation}\label{tallfenichel}
F_\mu =\eps^{\frac{2}{3}}\eta_0(0)+ \OO(\eps)
\end{equation}
with $\eta_0(0)$
the same for all of $0<\alpha<1$ and $\eta_0(u)$ is the solution of \eqref{Ricatti}.
%given ib \eqref{eq:Qepsilon}.
\end{proof}

As a consequence of the results of the previous lemma we can ensure that the map $\QQQ_{\mu, \eps}$ satisfies \eqref{eq:Qepsilon} if $0<\mu<\eps$.
Consequently, we know the behavior of $\QQQ_{\mu,\eps}$ for points $x\le -\eps^\lambda$, $\lambda <\frac23$.
Next step is to understand its behavior near the tangency $x_\mu$, more concretely in intervals of the form $[-M\eps^\frac23, -\overline  M\eps^\frac23]$.
This is done in next theorem.

\begin{theorem}\label{prop:asymptoticsQ}
Take any constants  $0<\overline M<M$.
Then, there exists $\eps_0$ small enough
such  for  $0<\eps<\eps_0$, $\mu \in (0,\eps)$ we have:
\begin{itemize}
\item
For all $x \in [-M\eps^\frac23, -\overline M\eps^\frac23]$, the map $\QQQ_{\mu,\eps}$
satisfies:
\begin{equation}
\QQQ_{\mu,\eps}'<0, \quad \QQQ''_{\mu,\eps}(x)<0
\end{equation}
\item
Take $0<\overline C<C$ small enough, then we have for $x\in [-C\eps^\frac23, -\overline C\eps^\frac23]$ we have:
\begin{equation}\label{derivadesQQvell}
\mathcal{Q}_{\mu,\eps}(x)=-x(1+\OO(\frac{x}{\eps^{\frac23}}))+\OO(\eps),
\end{equation}
\end{itemize}

\end{theorem}

\subsection{Proof of Theorem \ref{thm:main}}
Now we refine the range of $\mu$ where the bifurcation will take place.
In Remark  \ref{rem:xdelta+-} we have seen that the intersection of the periodic orbit $\Gamma_\mu$ with  $\SSS_{\eps}^+$ is the point $(x_\mu^+,\eps)$ with $x_\mu^+=\rdelta+\OO(\delta)$, and $\de=\eps-\mu$ (see \eqref{eq:xdelta+-}).
Moreover in Remark \ref{rem:bifurcacio}
we have  seen that the bifurcation will take place when $0<\mu<\eps$ and  we expect it to happen when the Fenichel solution(s) \eqref{tallfenichel} and the upper segment of $\Gamma_\mu$ "collide" in $\SSS_{\eps}^+$ at some order.
So, heuristically, we expect:
\begin{equation}
\begin{array}{l}
x_\mu^+ \simeq F_\mu
\\
\sqrt{\eps-\mu}+\OO(\eps-\mu)=\eps^{\frac{2}{3}}\eta_0(0)+\OO(\eps)\\
\mu=\eps-\eps^{\frac{4}{3}}\eta_0^2(0)+\OO(\eps^{\frac{5}{3}})
\end{array}
\end{equation}
This suggests to take the range of $\mu$ and $\delta$ as
\begin{equation}\label{rangmu}
\begin{array}{l}
\mu_1:=\eps-\eps^{\frac{4}{3}}\eta_0^2(0)-K_1\eps^{\frac{5}{3}}<\mu<\mu_2:=\eps-K_2^2\eps^{\frac{4}{3}}\\
\mbox{equivalently:}\\
\delta_2\equiv\eps-\mu_2=K_2^2\eps^{\frac{4}{3}}<\delta<\delta_1\equiv\eps-\mu_1=\eps^{\frac{4}{3}}\eta_0^2(0)+K_1\eps^{\frac{5}{3}}
\end{array}
\end{equation}
And the constants $K_1$, $K_2$ will be chosen later on.
Let's compute the intersections of $\Gamma_\mu$ with $\SSS^\pm_\eps$ for the values of $\mu_1$ and $\mu_2$:
\begin{equation}\label{rangxdelta}
\begin{array}{l}\\
x_{\mu_{1}}^+=\sqrt{\eps-\mu_1}+\OO(\eps-\mu_1)=\sqrt{\eps^{\frac{4}{3}}\eta_0^2(0)+K_1\eps^{\frac{5}{3}}}
+\OO(\eps^{\frac{4}{3}})
\\
=\eps^{\frac{2}{3}}\eta_0(0)\sqrt{1+\frac{K_1\eps^{\frac{1}{3}}}{(\eta_0(0))^2}}+\OO(\eps^{\frac{4}{3}})=\eps^{\frac{2}{3}}\eta_0(0)+\frac{K_1\eps}{2\eta_0(0)}
+\OO(\eps^{\frac{4}{3}});\\
x_{\mu_{2}}^+=K_2\eps^{\frac{2}{3}}+\OO(\eps^{\frac{4}{3}})
\end{array}
\end{equation}
Then if we take $K_1$ large enough, we will have
 \[
 F_{\mu_1}<x_{\mu_1}^+
 \]
and therefore the Fenichel manifold will be inside the vault of $\Gamma_{\mu_1}$, that is, $ F_{\mu_1}\in [x_{\mu_1}^-, x_{\mu_1}^+]$.
But then, by Proposition \ref{prop:trapping}, the  orbit through $F_{\mu_1}$ is trapped by the attracting focus.
Consequently, reasoning analogously as in Proposition \ref{prop:flowtangency}, there is not a periodic orbit.
The same phenomenon happens for $\mu \le \mu_1$ and we summarize these results in next proposition:
\begin{proposition}\label{prop:bifsegura1}
If $K_1>0$ big enough, then  the regularized field $Z_{\mu,\eps}$
has no  periodic orbits for $\mu \le \mu_1=\eps-\eps^{\frac{4}{3}}\eta_0^2(0)-K_1\eps^{\frac{5}{3}}$.
\end{proposition}
On the other hand, if $K_2>0$ is small enough
\[
 F_{\mu_2} >x_{\mu_2}^+
 \]
Next Proposition  \ref{prop:bifsegura}  ensures that, if $K_2>0$ is small enough, the regularized field $Z_{\mu,\eps}$ has two periodic orbits if $\mu \ge \mu_2$.
\begin{proposition}\label{prop:bifsegura}
If $K_2>0$ small enough, then  the regularized field $Z_{\mu,\eps}$
has two periodic orbits   for $\mu \ge \mu_2=\eps-K_2^2\eps^{\frac{4}{3}}$.
\end{proposition}
\begin{proof}
Let's take $\mu = \mu_2$ and therefore $\delta=\de_2=K_2^2\eps^{\frac43}$.
Consider the point  $(2K_2\eps^{\frac{2}{3}},\eps)$.
Assuming $K_2$ is small enough we can ensure that
$2K_2\eps^{\frac{2}{3}}\le F_{\mu_2}$.
We will see that:
\begin{equation}\label{eq:dsigualtatde2}
\mathcal{Q}_{\mu,\eps}(\f_{\mu,\eps}(2K_2\eps^{\frac{2}{3}}))>2K_2\eps^{\frac{2}{3}}.
\end{equation}
In fact, taking the constant $C$ in Theorem \ref{thm:extensiopie} satisfying $C>4$, we have that $2K_2\eps^{\frac{2}{3}} \in[\sqrt{\delta_2}\sqrt{1-\bar{y}_0},\sqrt{\delta_2 C}]$.
Then   we can use formula \eqref{eq:formulapiebonathm}, obtaining:
\begin{equation}
\begin{array}{l}
\f_{\mu,\eps}(2K_2\eps^{\frac23})=
-\sqrt{ \delta_2- \pi'(0)(\delta_2-4 K_2^2\eps^{\frac43 })}+\OO(\delta_2)
=
-K_2\eps^{\frac{2}{3}}\sqrt{1+3\pi'(0)}+\OO(\eps^{\frac{4}{3}})
\end{array}
\end{equation}\\
But using that $x_{\mu_2}= \OO(\de_2)=\OO(\eps^\frac43)$ (see \eqref{eq:xepsilon}) we can ensure that
\[
\f_{\mu,\eps}(2K_2\eps^{\frac23})=
-K_2\eps^{\frac{2}{3}}\sqrt{1+3\pi'(0)}+\OO(\eps^{\frac{4}{3}}) <x_{\mu_2}
\]
Now, if $K_2$ is small enough,  we can use formula \eqref{derivadesQQvell}  for $\QQQ_{\mu,\eps}$ obtaining:
\[
\mathcal{Q}_{\mu,\eps}(\f_{\mu,\eps}(2K_2\eps^{\frac{2}{3}}))
=K_2\eps^{\frac{2}{3}}\sqrt{1+3\pi'(0)}\left(1+\OO\left(K_2\sqrt{1+3\pi'(0)}\right)^{4}\right)+\OO(\eps) >2K_2\eps^{\frac{2}{3}}
\]
where we have used that $\pi' (0)>1$ and that $K_2$ is small enough.

Once we have proved inequality \eqref{eq:dsigualtatde2} we have that the solution issuing from $(2K_2\eps^{\frac{2}{3}},\eps) $ spirals outside and is bounded by the Fenichel solution
(which leaves $\SSS^+_\eps $ at $(F_{\mu_2}, \eps)$ with
$F_{\mu_2}=\eta_0 (0) \eps ^\frac23 + \OO(\eps)$ ).
Then, between the two solutions must be a periodic orbit, which intersects $\SSS_\eps ^+$ at a point $x^*_{\mu_2} \in (2K_2\eps^{\frac{2}{3}}, F_\mu)$. Moreover it is a stable periodic orbit.

To see that there is another periodic orbit we proceed as follows.
Consider the map $f(x)=\mathcal{Q}_{\mu,\eps}(\f(x))-x$.
We have:
\begin{itemize}
 \item
 $f(2K_2\eps^{\frac{2}{3}})>0$.
 \item
 $f(x_{\mu_2}^+)<0$.
 The reason is that, using Proposition \ref{prop:comparacio}, the  orbit through $(x_{\mu_2}^+,\eps)$ will intersect $\SSS_\eps ^+$ in a point inside the vault of $\Gamma_{\mu_2}$.
\end{itemize}
Therefore, they will be $x_{\mu_2}^{**} \in (x_{\mu_2}^+, 2K_2\eps^{\frac{2}{3}})$ such that $f(x_{\mu_2}^{**})=0$, giving rise to another periodic orbit.
The proof for $\mu \ge \mu_2$ is analogous.

\end{proof}

\begin{remark}
 We stress that if there exists any periodic orbit for a given value of $\mu$ necessarily it hits $\SSS^+_\eps$ in a point which is on the right of $x_\mu^+$ and on the left of $F_\mu$.
\end{remark}
Propositions \ref{prop:bifsegura1} and \ref{prop:bifsegura} ensure that  the bifurcation will take place for
\begin{equation}\label{eq:bifurcationdomain}
\mu \in (\mu_1,\mu_2)=(
\eps-\eps^{\frac{4}{3}}\eta_0^2(0)-K_1\eps^{\frac{5}{3}},\eps-K_2^2\eps^{\frac{4}{3}}),
\end{equation}
if we take the constants $K_1$ and $K_2$ with the required conditions.
Therefore, from now on, we will restrict our study to this rank of $\mu$.

Observe that, by Theorem \ref{thm:extensiopie}, we already know the map $\f_{\mu,\eps}$ and have asymptotic formulas for it for $x\ge \tilde x_\mu$. Analogously, Theorem \ref{prop:asymptoticsQ} gives the needed properties of the map $\QQQ_{\mu,\eps}$.

Now we have all the ingredients to prove Theorem \ref{thm:main}.
Consider the interval
\[
\mathcal{I}=[\sigma\rdelta,\sqrt{2}\eta_0(0)\eps^{\frac{2}{3}}],
\]
where $\sigma>0$ is the constant given in Theorem \ref{thm:extensiopie}.
It is clear that in the considered range of $\mu\in [\mu_1,\mu_2]$, we have, by \eqref{rangmu}, if $\eps$ is small enough:
\[
F_\mu  \, , x_\mu^+ \in \mathcal{I}
\subset
[\tilde{x}_\mu,\sqrt{2}\eta_0(0)\eps^{\frac{2}{3}}] ,
 \ \mu \in [\mu_1,\mu_2], \ \de=\eps-\mu
\]
where $x_\mu^+$ is given in \eqref{eq:xdelta+-}, $F_\mu$ in \eqref{tallfenichel}.
Consequently, if there is a fix point of the map
$\mathcal{Q}_{\mu,\eps}\circ \f_{\mu,\eps}$ it must be in $\mathcal{I}$ and the corresponding fix point of $\f_{\mu,\eps}\circ \mathcal{Q}_{\mu,\eps}$ must be in $\f(\mathcal{I})$.

We apply Theorem \ref{thm:extensiopie} to this interval and we have:
\[
\f_{\mu,\eps} (\mathcal{I})= [\f_{\mu,\eps}(\sqrt{2}\eta_0(0)\eps^{\frac{2}{3}}), \f_{\mu,\eps}(\sigma\rdelta)]
\]
On the other hand, using the formula for $\f_{\mu,\eps}$ given in \eqref{eq:formulapiebonathm}, it is straightforward to see that:
\begin{equation}
\begin{split}
&\f_{\mu,\eps}(\sqrt{2}\eta_0(0)\eps^{\frac{2}{3}})\ge -\eta_0(0)\eps^{\frac{2}{3}}\sqrt{2\pi'(0)}>-M \eps^{\frac{2}{3}}\\
&\f_{\mu,\eps}(\sigma \sqrt{\de})<-K_2 \eps^{\frac{2}{3}}\sqrt{1-\pi'(0)(1-\sigma^2)}<- \overline M \eps^{\frac{2}{3}},
\end{split}
\end{equation}
for some constants $M$ and $\overline M$. Consequently
\[
\f_{\mu,\eps}(\mathcal{I})
\subset [-M\eps^{\frac{2}{3}},-\overline M \eps^{\frac{2}{3}}]:=\mathcal{J},
\]
and we can apply the results of Theorem \ref{prop:asymptoticsQ}.
In conclusion we have that
for $x\in \mathcal {I}$, we have that $\bar x=\f_{\mu,\eps}(x)\in \mathcal{J}$ and
$$
(\f_{\mu,\eps}\circ  \QQQ_{\mu,\eps}  )'' (\bar x)=
({\f}_{\mu,\eps})'' (\QQQ_{\mu,\eps}(\bar x)) ((\QQQ_{\mu,\eps})'(\bar x))^2+(\f_{\mu,\eps})'( \QQQ_{\mu,\eps}(\bar x))(\QQQ _{\mu,\eps})''(\bar x) >0
$$
where we have used the convexity of $\f_{\mu,\eps}$, the concavity of  $ \QQQ_{\mu,\eps}$ and the fact that $\f_{\mu,\eps}$ is decreasing in $\mathcal{J}$.

This concludes the proof of the first three items of Theorem \ref{thm:main}.
\subsection{An example}
As an example, let's take the family of vector fields $Z_\mu =(X^+_\mu, X^-_\mu)$ where $X^+$ is given by
\begin{equation}\label{exemple}
\left.
\begin{array}{rcl}
\dot x &=f(x,y,\mu)=& -y +\mu+1+\kappa x(r-1)\\
\dot y &=g(x,y,\mu)=& x + \kappa(y-\mu-1)(r-1)
\end{array}
\right\} \quad r=\sqrt {x^2+(y-\mu-1)^2},
\end{equation}
and  $ X^-=(0,1) $.

\begin{figure}
\begin{center}
\includegraphics[width=8cm,height=3cm]{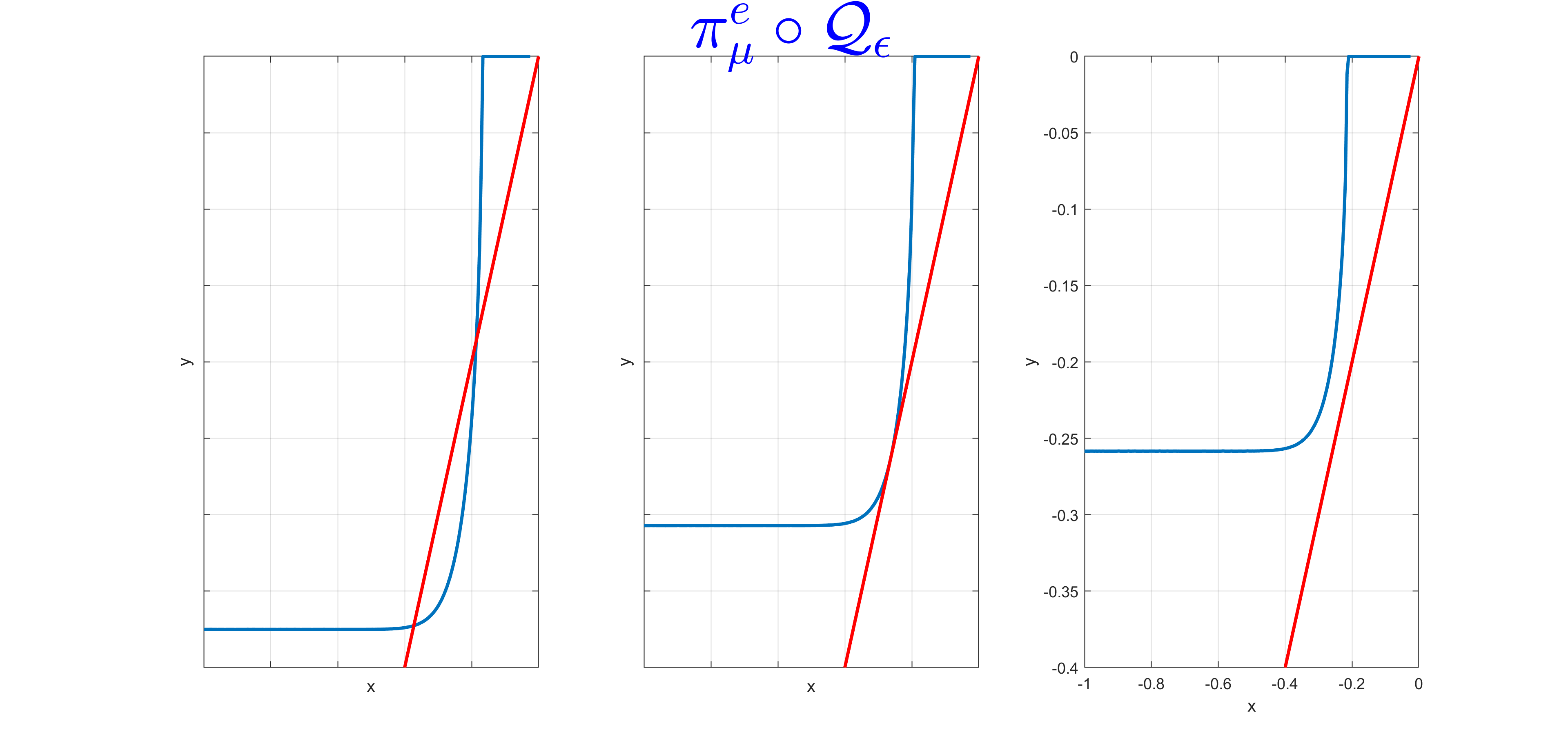}
\caption{The convex Poincar\'{e} map for the regularized system $Z_{\mu, \eps}$ of example \eqref{exemple}.}\label{fig:poincaremap}
\end{center}
\end{figure}
We see in Figure \ref{fig:poincaremap} the Poincar\'{e} map $ \f_{\mu,\eps}\circ\mathcal{Q}_\varepsilon $ defined in $[-1,0]$
and for $\kappa=1$, $\varepsilon=.05$ and $ \mu_{1,2,3}=\varepsilon-(.5,.5623,.6)\varepsilon^{\frac{4}{3}} $ has two, one and zero fixed points.\\

\section{Hysteresis} \label{sec:H}
In this section we will study the effects of a different regularization of the Filippov system  \eqref{def:Filippov}.
This is the so-called hysteresis, which can be seen  as another way of regularizing discontinuous systems.

Let us first recall that, for a given Filippov system as in \eqref{def:Filippov}, one can define the Filippov vector field in the sliding region,  a subset of the  switching manifold $\Sigma^s\subset \Sigma$ where both vector fields point towards $\Sigma$. In our case, where $\Sigma$ is given by $y=0$:
\begin{equation}\label{Filippov}
\dot x=Z_F(x)= \frac{\Y_2\X_1-\Y_1\X_2}{\Y_2-\Y_1}(x,0)
\end{equation}
and it is well known \cite{TeixeiraBS2006,TeixeiraS2012} that, in a neighborhood $W$ of any sliding region $\Sigma^s \subset W$, the orbits of the Sotomayor-Teixeira regularization $Z_{\eps}(x,y)$ tend to the orbits of the Filippov vector field  \eqref{Filippov}.

Let us now recall how hysteresis is applied to a system like \eqref{def:Filippov} if we are in a sliding region.
The main idea is that in a 'negative' boundary layer we define an overlap in the non smooth system:
\begin{eqnarray}\label{hyst}
Z_h (x,y)= \left\{\begin{array}{lll}\X (x,y) &\rm if&y>-\eps
\\
\Y(x,y) &\rm if&y<+\eps
\end{array}\right.
\end{eqnarray}
and a trajectory of $\X$
switches to a trajectory of
$\Y$ when it reaches $y=-\eps$, and  a trajectory of $\Y$ switches to a trajectory of $\X$ when it reaches $y=\eps$ and so on.

We can illustrate this regularization method with the next simple example.
Consider the planar piece-wise smooth system
\begin{equation}\label{eq:exutkin}
\dot x=0.3+u^3\qquad\dot y=-0.5-u\qquad u={\rm sign}(y)\;.
\end{equation}
If we perform the hysteretic regularization we obtain the trajectories shown in Figure~\ref{fig:exutkin2}.
\begin{figure}
\begin{center}
\includegraphics[height=3cm, width=7cm]{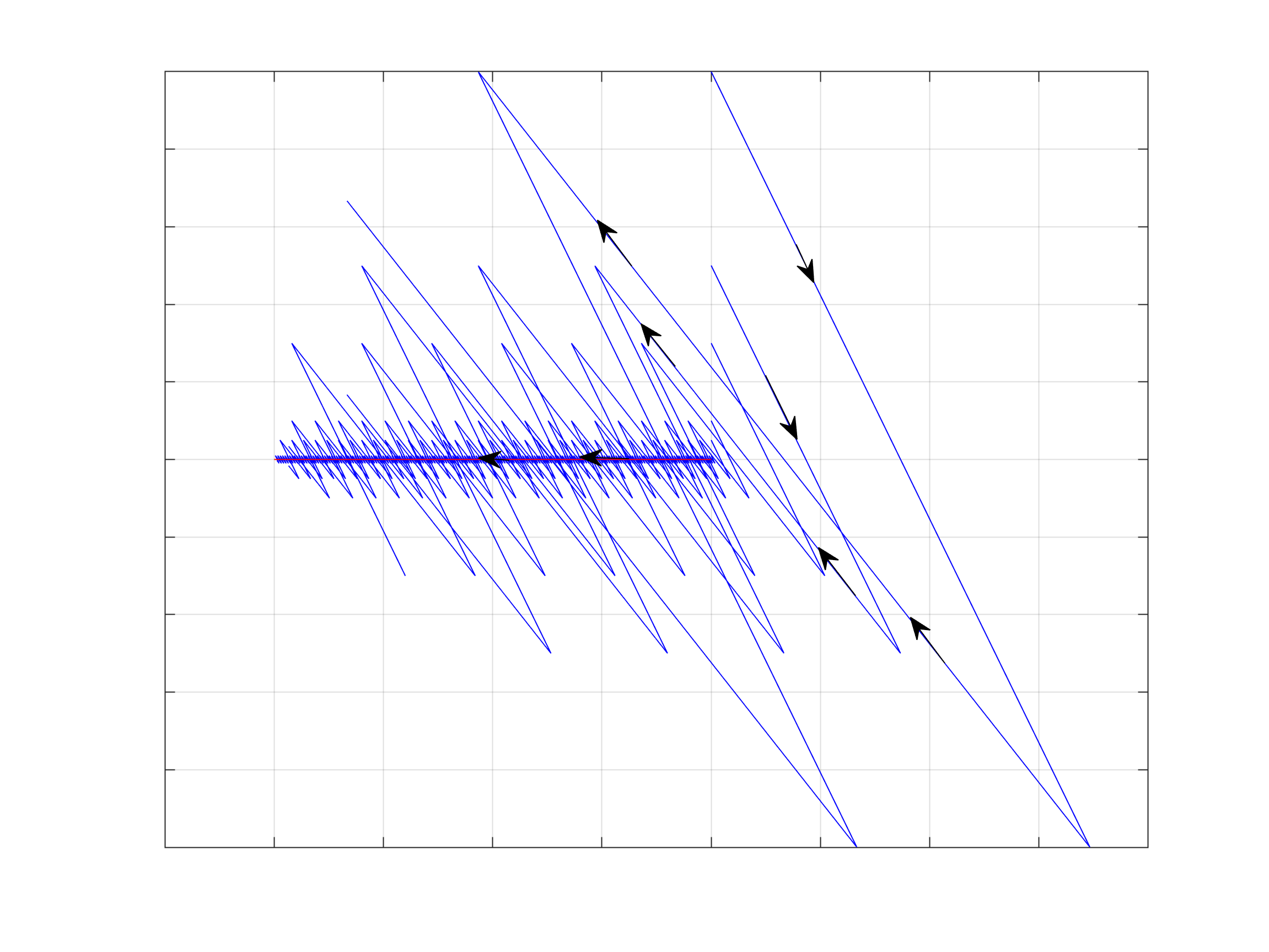}
\caption{The hysteretic behavior for the example \eqref{eq:exutkin} for diminishing values of $\alpha$. The line in red is the solution $x_F(t)$ of the Filippov system~\eqref{Filippov} with $x_F(0)=5$ and $0\le t \le 10$}\label{fig:exutkin2}
\end{center}
\end{figure}
Next Theorem, which is Theorem 1 in \cite{BonetSFG2017}, proves that, in sliding regions, the orbits generated through the hysteretic regularization tend to the orbits of the system generated by the Filippov vector field \eqref{Filippov} in $\Sigma^s$ as the parameter  $\eps\to 0$:
\begin{theorem}[\cite{BonetSFG2017}]
Fix $T >0$ and consider a solution $x_F(t)$ of the Filippov System \eqref{Filippov},
and assume that $|x_F(t)| <M$ for $0\le t\le T$ .
Then, there exists $\eps _0>0$ and  a constant $L>0$ such that,
for $0<\eps\le \eps_0$, if we
consider the hysteretic solution $(x_h(t), y_h(t))$ of \eqref{hyst} with initial condition
$(x_h(0), y_h(0))=(x_0, -\eps)=(x_F(0), -\eps)$, we have
\begin{equation}\label{uta}
|x_h(t)-x_F(t)| \le L \eps \quad 0\le t\le T
\end{equation}
\end{theorem}
Therefore, we have that, in sliding regions,  both the  Sotomayor-Teixeira  and the hysteretic regularizations  tend to the orbits of Filippov system in $\Sigma^s$.
But the way that both regularizations approximate is different. Hysteresis approximates in a chattering manner, the Sotomayor-Teixeira approximates  in a smooth manner.
So they can produce quite different behaviors when the ``hyperbolicity'' which exists in the sliding region is lost, as happens, for instance, in a visible fold point.

In Theorem \ref{thm:sella-node-si} of  Section \ref{sec:sella-node} we have completed the work in \cite{BonetS16}, and we have  seen the effects of the Sotomayor-Teixeira regularization $Z_{\mu,\eps}$ of the family of Filippov vector fields \eqref{hypofamily} having a grazing-sliding bifurcation of (repelling) periodic orbits. We have seen that the regularized vector field $Z_{\mu,\eps}$ undergoes a saddle-node bifurcation.

In this section we will consider the same family $Z_{\mu}$ in  \eqref{hypofamily} and its regularization $Z_{\mu,h}$ by hysteresis and we will see
that a cascade of bifurcations leading to chaos appears in a interval of the parameter $\mu$ and later, for $\mu>0$ and $\eps$ small enough.

We consider the Filippov system $Z_{\mu}$ in \eqref{hypofamily} of the previous section, which has a grazing-sliding bifurcation, and we perform the hysteretic process.
Moreover, we can suppose that, locally, for $|y|$ small:
\begin{equation}\label{eq:X0nou}
X_0(0,y)=(1+O(y),0).
\end{equation}
In fact, using the implicit function theorem to the second equation of \eqref{def:Xg}, one obtains $x=x(y)$ satisfying:
\[
2x(y) + by +f_2(x(y),y)=0; x(0)=0
\]
and after the change
$ \bar{x}=x-x(y)$ we have a system of the form \eqref{eq:X0nou}.
Terefore, near $(0,0)$ the orbits through $(0,y_0)$ are tangent to $y=y_0$, that is, the points $(0,y_0)$ are folds.
This is not strictly necessary, but simplifies the exposition.

Recall that the vector field $\X_\mu$ has a periodic orbit $\Gamma_\mu$ which is tangent to the line $y=\mu$.
Therefore $\Gamma_\mu$ is entirely contained in  the region $\{y>\eps\}$ if $\mu >\eps$, is tangent to $\SSS_\eps$ for $\mu=\eps$ and intersects the hysteretic region $\{|y|\le \eps\}$ for $ <\mu<\eps$.

Now, instead of $\eps$, we call $\alpha>0$ the regularization parameter. So the boundary layer where will take place the hysteretic process is the strip $|y|\le\alpha$.
As we did in Section \ref{sec:sella-node} we begin by studying the three cases: $\mu<0$, $\mu=0$ and $\mu>0$.
In the hysteretic regularization, the bifurcation
(a sort of) will take place when $\mu=\OO(\alpha)$. In fact,
in the scope of our hypothesis,
it will occur
exactly
at $\mu=\alpha$, and we will see that for $\mu\ge\alpha$, it appears chaotic
behavior.

\subsection{The Poincar\'{e} map}
To understand the dynamics of the hysteretic vector field $Z_{\mu,h}$ we will consider a Poincar\'{e} map defined in the section $\SSS_\al^-$ (see \eqref{eq:S0}) in the following way:
\[
P_\mu:[A,0]\times \{y=\al\}\subset \SSS_\al^-\rightarrow [A,0]\times \{y=\al\}
\]
defined for $A<0$ small enough, but fixed.
The definition of $P$ is as follows:
\begin{itemize}
 \item
If $x\in[A,0]$ $P_\mu(x)$ will be obtained by the hysteretic process applied to the point $(x,\alpha)$.
That is, we consider  the positive orbit beginning at $(x,\alpha)$ of the field $\X_{\mu}$ till it intersects  $y=-\alpha$ at a point $(\bar x,-\al)$, then we consider the orbit of the lower field $\Y=(0,1)$ beginning at $(\bar x,-\al)$ till it arrives to $y=\alpha$, and we define $P_\mu(x)$ as the $x$ coordinate of this last point. Observe that the form of $\Y$ implies that $P_\mu(x)=\bar x$.
\item
If the orbit through $(x,\al)$ does not intersect $y=-\al$ for positive times, we define $P_\mu(x)=0$.
\end{itemize}
We call $(E_\mu,\al)$ to the point whose positive orbit through $\X_{\mu}$ is tangent to  $y=-\alpha$ (in fact at $(0,-\al)$), and therefore:
\begin{equation}\label{eq:discontinuitatE}
 P_\mu (E_\mu)=0,\quad \lim_{x\to E_\mu^-} P_\mu (x)=0 \quad  \forall \mu
\end{equation}
Another important point will be  the point:
\begin{equation}\label{eq:pointD}
(D_\mu,\al)=\Gamma_\mu\cap \SSS_\al^-
\end{equation}
In the next sections we will see that the value of  the right limit $\lim_{x\to E_\mu^+} P_\mu (x)$ will depend on the relative position between $E_\mu$ and $D_\mu$ and therefore of $\mu$ and $\al$.
\\
This will be related with the fact that,  sometimes, some turns around the focus
will be needed to reach $y=-\al$.
As our vector filed has a tangency with $y=-\al$ at the point $(0,-\al)$, this happens when $x>E_\mu:=P_\mu^{-1}(0)$ and we will see that this can be the cause of chaotic behavior.

\subsection{ The case $\mu\le -\al$ }

 Let's begin studying the orbits of the Poincar\'{e} map $P_\mu$ when $\mu<0$.
 In this case, for $|\al|$ small enough, the periodic orbit $\Gamma_\mu$ intersects the hysteretic region $|y|\le \al$ which implies that $D_\mu <E_\mu$ (see \eqref{eq:pointD}).
 In fact, one can take any  $A<D_\mu<E_\mu$.
The Figure \ref{fig:poin_mu_negativa} is a model for this case.
\begin{itemize}
\item
If $x<E_\mu$ then it is clear that $x<P_\mu(x)<0$.
\item
If $E_\mu<x<0$, the positive orbit through $(x,\al)$ does not intersect $y=-\al$ anymore, because it is in the vault of $\Gamma_\mu$ which is repelling.
Therefore, following our convention, we define:
 \[
 P_\mu(x)=0 \quad \mbox{ for } E_\mu \le x\le 0.
 \]
\end{itemize}
 \begin{figure}
\begin{center}
\includegraphics[width=15cm,height=6cm]{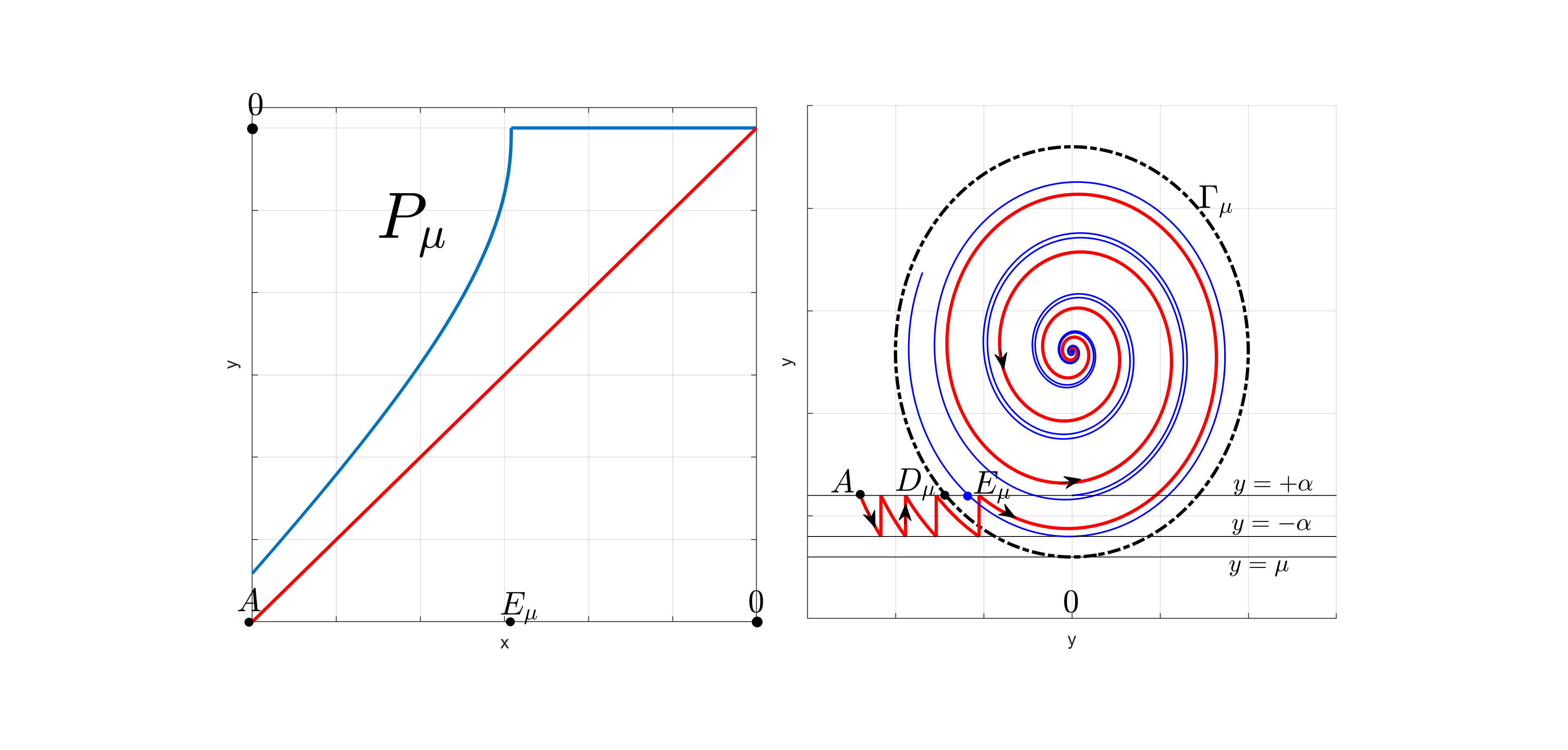}
\caption{The Poincar\'{e} map $P_\mu$ and some iterates of the orbit of the hysteretic regularization of the system \ref{exemple} scaled by $y/\alpha$ and with $\kappa=0.2$, beginning at $(-1.2,\alpha)$ for $\alpha=0.1$ and $\mu=-0.2$}\label{fig:poin_mu_negativa}

\end{center}
\end{figure}
Therefore, for $\mu<0$ the dynamics of the Poincar\'{e} map is simple:
 \[
\forall x\in[A,0], \quad \exists n>0 \mid P_\mu^{n}(x)=0,
 \]
and therefore $(0,0)$ is a global attractor and there is no periodic orbit for $Z_{\mu,h}$.

Observe that the same behavior occurs for $\mu<-\alpha<0$ because, in these cases, $D_\mu<E_\mu$. In fact, also for $\mu = -\al$. In this last case $D_{-\al}=E_{-\al}$.

 \subsection{ The case   $-\al <\mu<\al$}
In this section we will see that the Poincar\'{e} map $P_\mu$ satisfies, on the one hand, that is has a countable set of fixed points $\gamma_n$ giving rise to countable many periodic orbits, but in the other hand that $
\lim_{k\to\infty}P_\mu^k (x)=0 $ almost everywhere.

The positive orbit of  $X_\mu^+$ through  $(0,-\al)$  will intersect  $\SSS_\al^-$. From now and on we call $(A'_\mu,\al)$ its first intersection and $A_\mu=P_\mu(A'_\mu)$ its image through the hysteretic process.
We will first study in detail the case $\mu=0$:

 \subsubsection{ The case $\mu=0$ }
 \begin{figure}
\begin{center}
\includegraphics[width=9cm,height=5cm]{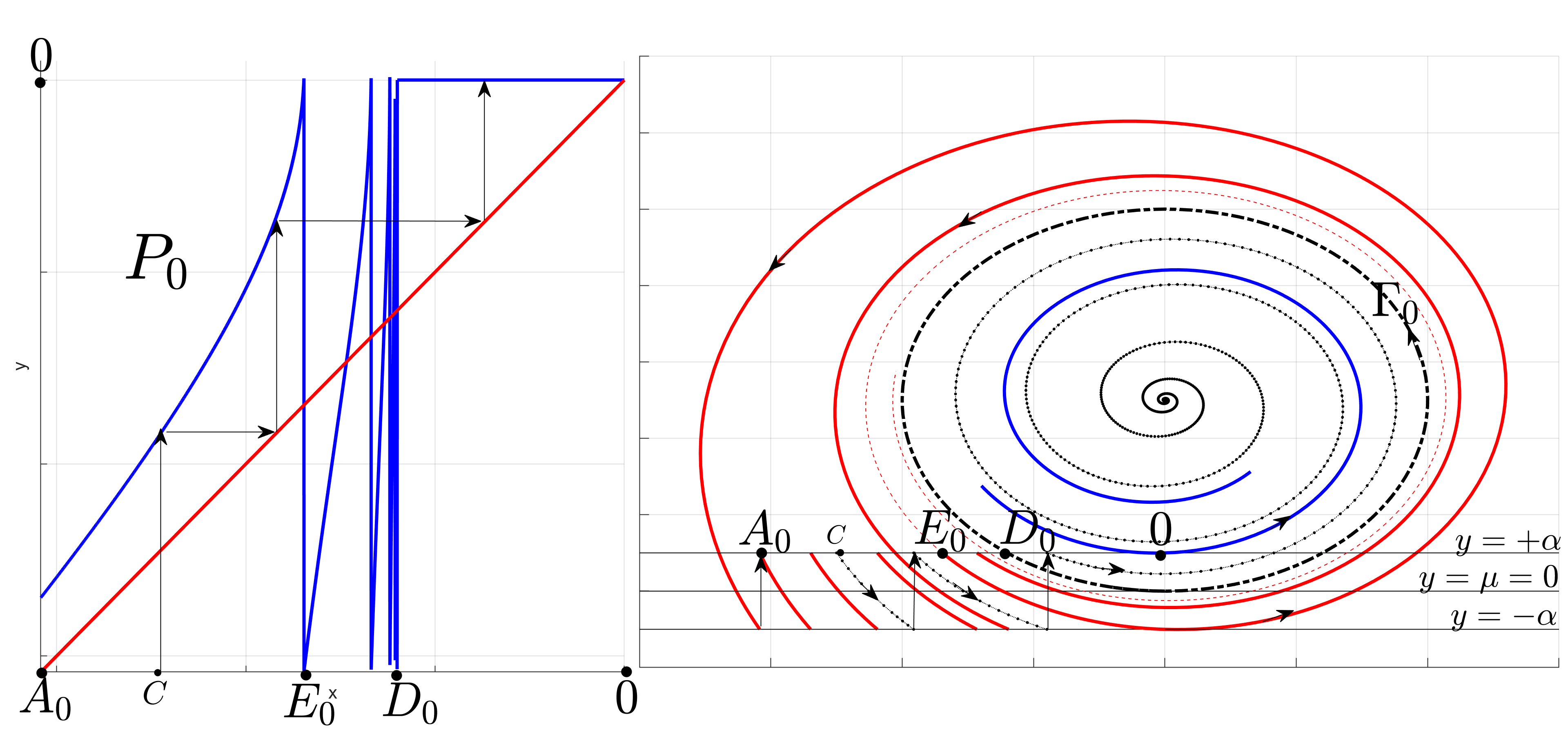}
\caption{The Poincar\'{e} map  $ P_0 $ and some iterates of the hysteretic process for $\mu=0$ for the hysteretic regularization of the system \eqref{exemple} scaled by $y/\alpha$ and with $\kappa=0.2$, $\alpha=0.2$ and $\mu=0$. The orbit tangent to $(0,-\alpha)$ separates the behavior and its infinite intersections on the interval $[E_0,D_0]$ produce on $D_0$ an accumulation of the discontinuities of the Poincar\'{e} map.}
\label{fig:poin mu zero}
\end{center}
\end{figure}
The model for this case is Figure \ref{fig:poin mu zero}.
As in this case the periodic orbit $\Gamma_0$ is tangent to $y=0$, its orbit never intersects $y=-\al$ and therefore $P_0(D_0)=0$.
Moreover, as $\Gamma_0$ is repelling the positive orbit of any point $(x,\al)$ with $x>D_0$ does not intersect $y=-\al$, therefore:
 \[
  P_0(x)=0,\quad D_0\le x\le 0.
 \]
In this case we have that $E_0<D_0$, therefore the positive orbit of  $\X_0$ through  $(E_0,\al)$, after its tangency at $y=-\al$, will intersect  $\SSS_\al^-$ at $(A'_0,\al)$ and $A_0=P_0(A'_0)$ its image through the hysteretic process.
Clearly $A'_0<A_0<E_0<D_0$ and we can consider the Poincar\'{e} map in $[A_0,0]$, then:
\begin{equation}\label{eq:discontinuitatE}
 P_0(E_0)=0,\quad \lim_{x\to E_0^-} P(x)=0 \quad  \lim_{x\to E_0^+}P(x)= A_0.
\end{equation}
Therefore the Poincar\'{e} map $P_0$ has  a discontinuity at $x=E_0$.

As $E_0<D_0$, the backward orbit of $(E_0,\al)$ through $\X_0$ spirals and accumulates to $\Gamma_0$.
 Let's call $(E_n, \al)$, $n \ge 1$,    the  infinite cuts of this negative orbit with $\SSS^-_\al$.
 Then we have, for $n\ge 1$:
 \[
E_n \in [E_0,D_0], \quad E_n<E_{n+1},\quad  \lim_{n\to\infty}E_n=D_0, \quad P_0(E_n)=0.
\]
Moreover, as in \eqref{eq:discontinuitatE}:
\begin{equation}\label{eq:discEn}
 \lim_{x\to E_n^-}P_0(x)=0 \quad  \lim_{x\to E_n^+}P_0(x)= A_0
\end{equation}
Summarizing, $P_0$ has  an accumulation of the discontinuities at $x=E_n$  which accumulate to $D_0$, is increasing   in $[E_n,E_{n+1}]$ and covers the interval $[A_0,0]$:
\[
P_0([E_n,E_{n+1}])=[A_0,0].
\]
Moreover, if we consider the Poincar\'{e} map $\pi_\mu$ associated to the periodic orbit $\Gamma_\mu$ through the flow of $\X_\mu$ on the section $\Sigma_\al^-$ (recall that $\Gamma_\mu$ intersects transversely this section):
\begin{equation}\label{eq:linealinterval}
\begin{split}
\pi_0:  [E_n,E_{n+1}] &\to [E_{n-1},E_n]\ n \ge 1 \\
\pi_0:  [E_0,E_{1}] &\to [A'_0,E_0]\\
\end{split}
\end{equation}
and this has several consequences:
\begin{itemize}
\item
As $\Gamma_0$ is repelling, we can assume that in the interval $[A_0, D_0]$ there exists constants $1<\xi < \eta$ such that $\xi <\pi_0 '<\eta$, therefore we have that the intervals shrink by a factor:
\begin{equation}\label{eq:factor}
\rho_1 |E_n-E_{n-1}|< |E_{n+1}-E_n|\le \rho_2 |E_n-E_{n-1}|, \ \rho_2=\xi^{-1} <1, \ \rho_1=\eta^{-1}<1
\end{equation}
\item
The definition of the map $P_0$ implies that, for $x \in [E_n,E_{n+1}]$,
\[
{P_0}_{|[E_n,E_{n+1}]}(x)={P_0}_{|[E_0, E_1]}\circ \pi_0^{(n)}(x) , \ n \ge 1
\]
\item
Observe that for $x\in [E_0,E_{1}] $,  $\pi_0 (x)\in [A'_0,E_0]$ and the definition of $P_0$ is again
$P_0(x)=P_0\circ \pi_0 (x)$.
Heuristically, if $x\in [E_n,E_{n+1}]$ it will take $n$ turns around the focus till $P_0(x)$ will be settled.
\item
For points in $[A,E_0]$ the map is given by the hysteretic map: $P_0(x)=P^h(x)$.
\item
As a consequence, $P_0$ has, near $x=D_0$, an accumulation of infinitely many fixed points:
\[
\gamma_n \in [E_n,E_{n+1}]
\]
which correspond to periodic orbits $\Gamma_n$ with increasing  periods which  give several turns before closing.
\end{itemize}
But, in despite of this apparently intricate behavior,  we can ensure that the measure of points $x\in[A,0]$ such that $\exists k>0 \mid P_0^{k}(x)=0$ is the total measure of interval $[A,0]$.
That is,
\[
\lim_{k\to\infty}P_0^k (x)=0 \quad \mbox{almost everywhere.}
\]
This is suggested in Figure \ref{fig:poin mu zero 4}.

\begin{figure}
\begin{center}
\includegraphics[width=10cm,height=5cm]{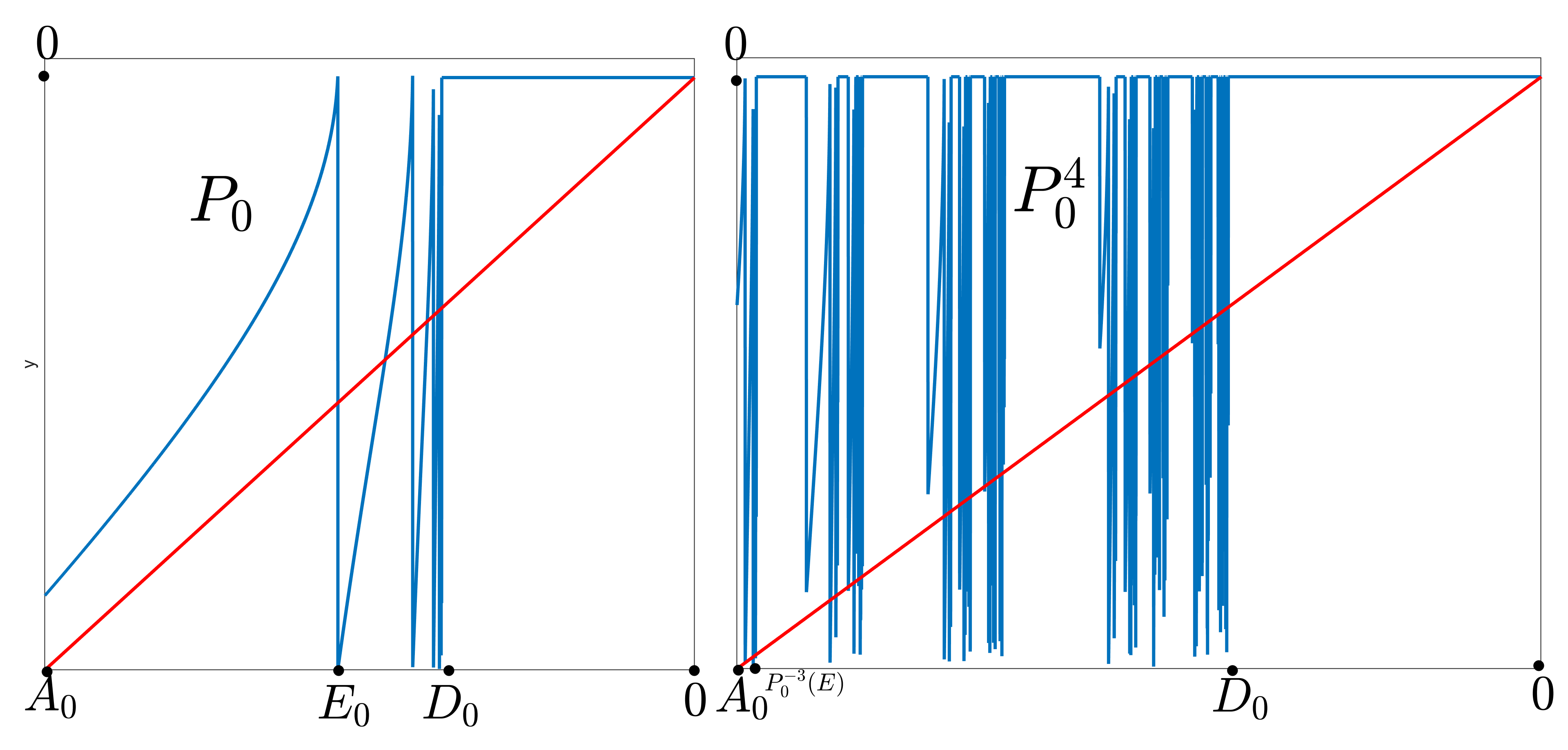}
\caption{The forth iterate, $P_0^4$, of the Poincar\'{e} map for the example of Figure \ref{fig:poin mu zero}. One can see the accumulation of discontinuities of $ P_0 $ on point $ D_0 $, and how this behavior is repeating in a growing number of subintervals. Also the flat parts of $P_0^n$ in all these intervals is growing to the full interval $[A_0,0]$ }
\label{fig:poin mu zero 4}
\end{center}
\end{figure}

To see this fact, let's first assume that $P_0(A_0)=A_0$, and consider an idealized linear model:
shown in the Figure \ref{fig:model ideal}.

Let
$T:[0,1]\rightarrow[0,1]$, and there exist a sequence $0=E_{-1}<E_0<\dots E_n <\sigma<1$ such that
\begin{itemize}
\item
$T(0)=0$, $T(x)=1$, for $x \ge \sigma$.
\item
$T$ is linear and increasing in the intervals $ I_n:=(E_n,E_{n+1})$ and
$T(E_n)=1$, $\lim _{x\to E_n^+}T(x)=0$, $\lim _{x\to E_{n+1}^-}T(x)=1$
\end{itemize}
Then the size of the pre-images of $1$ of the iterates of $T$ tends to $1$:
\[
\lim_{k\to \infty} \mu (T^{-k}(1))=1
\]
To see this we proceed as  follows:
\begin{itemize}
\item
Call $M_k=\{ x\in [0,1], \ T^k(x)=1\}$.
\item
Clearly $M_1=\cup _{n\ge -1}{E_n}\cup [\sigma, 1]$ therefore $\mu(M_1)=\sigma$, and the measure of points that $T$ does not send to $1$ is
$1-\sigma $
\item
To find $\mu(T^{-2}(1))$, only one has to take account that  $T ([E_n,E_{n+1}])=[0,1]$ and that
\[
\mu \{x\in [E_n,E_{n+1}], \ T(x)\ge \sigma\})=(E_{n+1}-E_n)\sigma
\]
and that
\[
\sum _{n\ge -1}(E_{n+1}-E_n)\sigma=(1-\sigma)\sigma
\]
Therefore:
\[
\mu(T^{-2}(1))= \sigma+(1-\sigma)\sigma .
\]
Moreover, the measure of points that $T^2$ does not send to $1$ is
$1-\left(\sigma+(1-\sigma)\sigma \right)$
\item
Proceeding by induction and taking into account that the map $T^k$ has the same structure than $T$ we have
\[
\begin{split}
\mu(T^{-k}(1))&=\sigma+(1-\sigma)\sigma +(1-\sigma-(1-\sigma)\sigma)\sigma+...=\sigma(1+(1-\sigma)+(1-\sigma)^2+..+(1-\sigma)^k)\\
&= \sigma \frac{1- (1-\sigma)^{k+1}}{1-(1-\sigma)}=1- (1-\sigma)^{k+1}
\end{split}
\]
Then we have
\[
\lim\limits_{n\to\infty}\mu(T^{-k}(1))=1
\]
\end{itemize}
\begin{figure}
\begin{center}
\includegraphics[width=7cm,height=2.3cm]{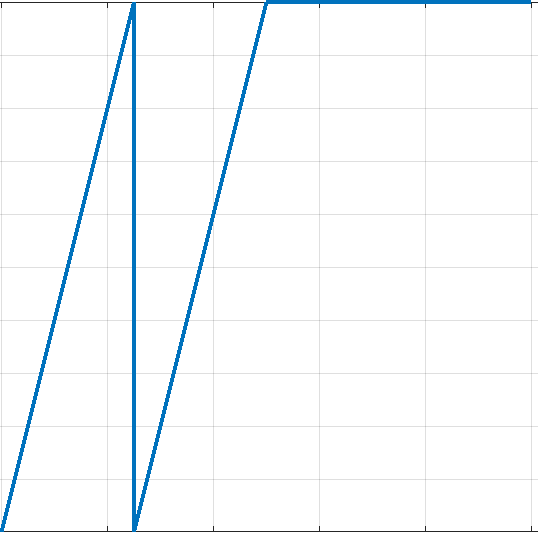}
\caption{The linear idealized model for $\mu=0$}
\label{fig:model ideal}
\end{center}
\end{figure}
In the case that we have a map  $\bar T$ satisfying the same properties than $T$ except in the first interval $[0,E_0]$ where it satisfies $\bar T(0)=\gamma>0$, $\lim _{x\to E_{1}^+}T(x)=1$, the main observation is to compare this map with the previous one and observe that:
\[
\mu (\{x \in [0,1], \ \bar T(x)\ge \sigma\}) \ge\mu (\{x \in [0,1], \  T(x)\ge \sigma\}), \ T(x)=\bar T(x), \ x\in [0,1]\setminus[0,E_0]
\]
and that both $T$ and $\bar T$ are increasing functions.
This gives that
\[
\lim\limits_{n\to\infty}\mu(\bar T^{-k}(1))\ge \lim\limits_{n\to\infty}\mu(T^{-k}(1))=1.
\]

\subsubsection{ The general case  $-\al <\mu<\al$ }

The dynamics when $-\al <\mu<\al$ is analog to  the one for $\mu=0$ because $E_\mu <D_\mu <0$ and the model  is again Figure \ref{fig:poin mu zero}.
The periodic orbit $\Gamma_\mu$ is tangent to $y=\mu$, which is in the regularity zone.
Its orbit never intersects $y=-\al$ and therefore:
 \[
  P_\mu(x)=0,\quad D_\mu\le x\le 0.
 \]
We call  $(A'_\mu,\al)$ the first intersection of this positive orbit  of $(E_\mu,\al)$ through $\X_\mu$ with   $\SSS_\al^+$ and $A_\mu=P_\mu(A'_\mu)$ its image through the hysteretic process.
We consider the Poincar\'{e} map in $[A_\mu,0]$, then as in \eqref{eq:discontinuitatE}:
\begin{equation}\label{eq:discontinuitatEmu}
P_\mu(E_\mu)=0,\quad \lim_{x\to E_\mu^-} P(x)=0 \quad  \lim_{x\to E_\mu^+}P(x)= A_\mu.
\end{equation}
Therefore the Poincar\'{e} map $P_\mu$ has  a discontinuity at $x=E_\mu<D_\mu$.

If we take the backward orbit of $(E_\mu,\al)$ through $\X_0$ spirals and accumulates to $\Gamma_\mu$.
 Lets call $(E_n, \al)$   the  infinite cuts of this negative orbit with $\SSS^-_\al$
 Then we have that, as in \eqref{eq:discEn}:
 \[
E_n \in [E_\mu,D_\mu], \quad \lim_{n\to\infty}E_n=D_\mu, \quad P_\mu(E_n)=0,\quad P_\mu(x)=0, \forall x\in[D_\mu,0].
\]
Moreover we have the same property \eqref{eq:discEn} and therefore $P_\mu$ has the same properties of $P_0$.

\subsection{ The route to chaos. The case $\mu=\alpha$ }

 \begin{figure}
\begin{center}
\includegraphics[width=10cm,height=5cm]{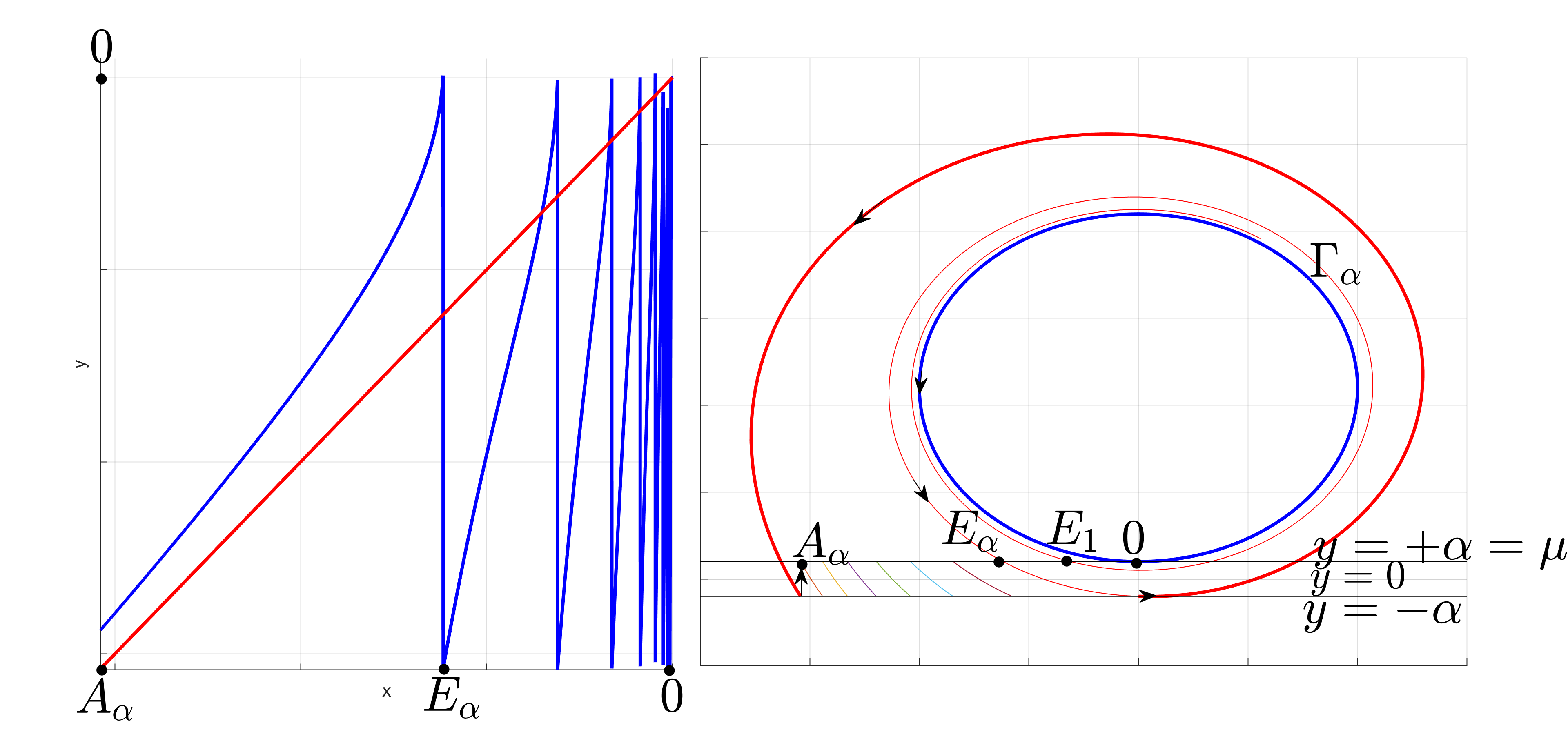}
\caption{The Poincar\'{e} map and some iterates of the hysteretic process for $\mu=\alpha$. The accumulation of discontinuities is on $0$. This will produce chaos as Baker-like map with infinitely many discontinuity points.The picture is made with the hysteretic regularization of the system \ref{exemple} scaled by $y/\alpha$ and with $\kappa=0.2$, $\alpha=0.1$ and $\mu=0.1$ }
\label{fig:poin mu alpha}
\end{center}
\end{figure}
In this case, the periodic orbit $\Gamma_\al$ is tangent to $y=\al$ at   the point  $(0,\al)$ and therefore $D_\al=0$. Then the discontinuities $E_n$,  $n \ge 1$  accumulate at $0$, and there is properly chaos. The model is in Figure \ref{fig:poin mu alpha}.
More concretely, as in  case $-\al<\mu<\al$, we also have $(E_n, \al)$, $n \ge 1$, the  infinite cuts of the negative orbit of $(E_\al,\al)$ with $\SSS^-_\al$.
 Then we have, for $n\ge 1$, and calling $E_0=E_\al$:
 \[
E_n \in [E_0,0], \quad E_n<E_{n+1},\quad  \lim_{n\to\infty}E_n=0, \quad P_\alpha(E_n)=0.
\]
and as in \eqref{eq:discontinuitatE}:
\begin{equation}\label{eq:discEn}
 \lim_{x\to E_n^-}P_\alpha(x)=0 \quad  \lim_{x\to E_n^+}P_\alpha(x)= A
\end{equation}
That is, $P_\alpha$ has   discontinuities at $x=E_n$  which accumulate to $x=0$, is increasing in $[E_n,E_{n+1}]$ and covers the interval $[A,0]$:
\[
P_\alpha([E_n,E_{n+1}])=[A,0].
\]
Now
the periodic orbit $\Gamma_\alpha$ does not intersect transversely the section $\Sigma_\al^-$, but it s tangent to it.  Then, the Poincar\'{e} map $\pi_\al$ associated to
the periodic orbit $\Gamma_\alpha$ it is not a properly Poincar\'{e} map in a neighborhood of $D_0=0$. For this reason we will call this map
$\tilde \pi$ and we observe that it satisfies:
\begin{equation}\label{eq:linealinterval}
\begin{split}
\tilde \pi:  [E_n,E_{n+1}] &\to [E_{n-1},E_n]\ n \ge 1 \\
\tilde \pi:  [E_0,E_{1}] &\to [A',E_0]\\
\end{split}
\end{equation}
Next proposition, whose proof is given in Section \ref{sec:provapropos}, shows that  $\tilde \pi'(0)>1$.
\begin{proposition}\label{prop:tildepi}
$\tilde \pi'(0)=\sqrt{\pi'(0)}$, where $\pi$ is Poincar\'{e} map  defined in \eqref{eq:poincaremap}.
\end{proposition}
As a consequence of the previous proposition, we can as well determine constants like in \eqref{eq:factor} (in fact its square roots) as in the case $\mu=0$:
\begin{figure}
\begin{center}
\includegraphics[width=9cm,height=5cm]{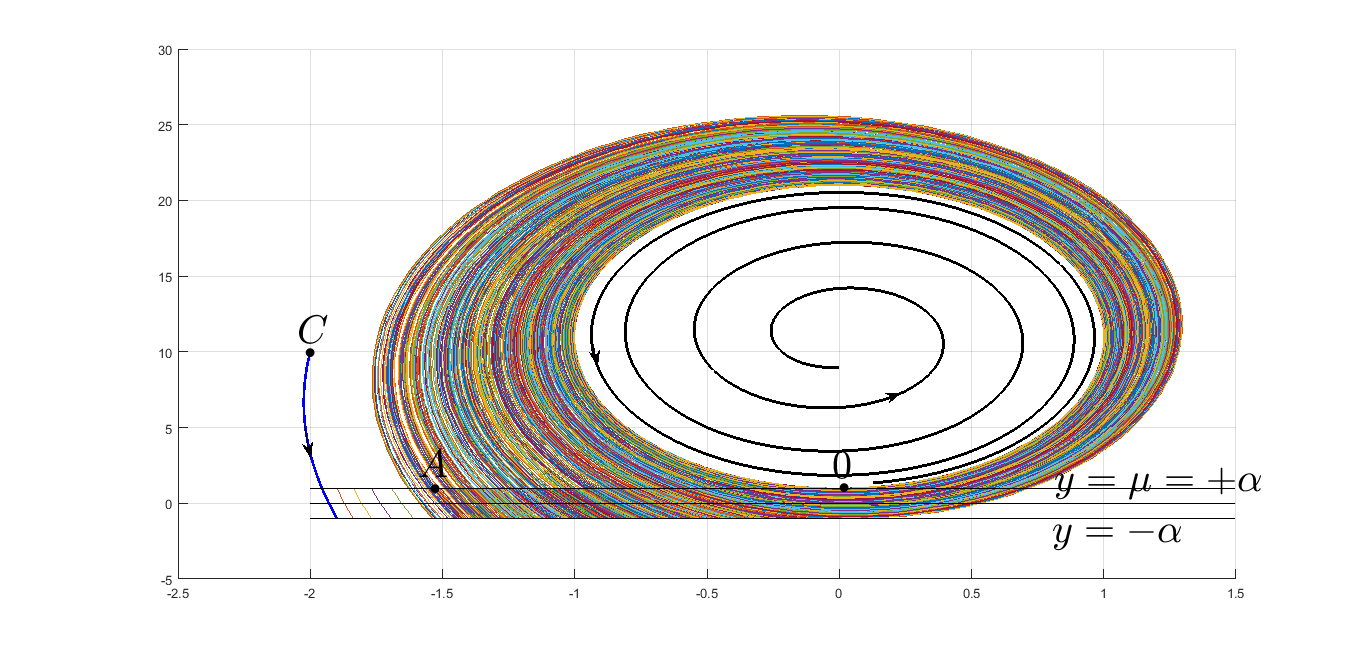}
\caption{The evidence of the chaotic behavior for $\mu=\alpha$. The orbit beginning at C fills densely a chaotic region. For better understanding the vertical orbits of the lower field are dismissed. The picture is made with the hysteretic regularization of the system \ref{exemple} scaled by $y/\alpha$ and with $\kappa=0.2$, $\alpha=0.1$ and $\mu=0.1$ }
\label{fig:attractor mu alpha}
\end{center}
\end{figure}
\begin{itemize}
\item
In the interval $[A, D_0]$ there exists constants $1<\tilde \xi < \tilde \eta$ such that $\tilde \xi <\tilde \pi '<\tilde \eta$, therefore we have that the intervals shrink by a factor:
\[
\tilde \rho_1 |E_n-E_{n-1}|< |E_{n+1}-E_n|\le \tilde \rho_2 |E_n-E_{n-1}|, \ \tilde \rho_2=\tilde \xi^{-1} <1, \ \tilde \rho_1=\tilde \eta^{-1}<1
\]
\item
The definition of the map $P_\alpha $ implies that, for $x \in [E_n,E_{n+1}]$,
\[
{P_\alpha}_{|[E_n,E_{n+1}]}(x)={P_\alpha}_{|[E_0, E_1]}\circ \tilde \pi^{(n)}(x) , \ n \ge 1
\]
\item
Observe that for $x\in [E_0,E_{1}] $,  $\tilde \pi(x)\in [A',E_0]$ and the definition of $P_\alpha$ is again
$P_\alpha(x)=P_\alpha\circ \tilde \pi(x)$.
Heuristically, if $x\in [E_n,E_{n+1}]$ it will take $n$ turns around the focus till $P_\alpha(x)$ will be settled.
\item
As a consequence, $P_\alpha$ has at $x=0$, an accumulation of infinitely many fixed points:
\[
\gamma_n
\in [E_n,E_{n+1}]
\]
which correspond to periodic orbits $\Gamma_n$ with increasing  periods.
\end{itemize}
The model is in Figure \ref{fig:poin mu alpha}. Is easy to see that ${P'_\alpha}_{|[E_n,E_{n+1}]}(x)>1$ (the singularity on the right extreme is $O(\sqrt{ E_{n+1}-x})$).
This kind of maps are  studied in \cite{GardiniM2015} and are called Baker-like maps. Among other problems where they appear, these maps rely on the study of the grazing bifurcations of impacting mechanical oscillators. In \cite{Nordmark1997} a one-dimensional limit mapping can be obtained through renormalization as we let the bifurcation  parameter go to zero. The mapping obtained is piece-wise continuous with an infinite number of branches, that is, a Baker-like map. These maps present robust chaotic attractors with the three conditions of Devaney: Transitivity, Density and Sensitivity. See Figure \ref{fig:attractor mu alpha}.

We can therefore conclude that, for $\mu=\al$, the hysteretic regularization $Z_{\mu,h}$ exhibits chaotic behavior.

\subsection{ The persistence of chaos. The case $\mu>\alpha\rightarrow 0$ }.

When $\mu \ge \al$
there will successive bifurcations
as the number of points  points $(E_n,\al)$, which correspond to the cuts of the negative orbit passing trough
$(E_\mu,\al)$ with $\SSS^{-}_\al$, changes.
More concretely, when $\mu=\alpha$ there where infinite numerable  $E_n$, $n\ge 1$, but as
$\mu$ increases only a finite number of cuts $(E_n,\al)$, $1\le n\le N=N(\mu)$, persist until just one.
In Figure \ref{fig:elstalls}
% \ref{fig:untall}, \ref{fig:dostalls} and \ref{fig:trestalls}
we see the cases of one, two and three cuts, with their relative Poincar\'{e} maps.
One can guess that the maps are still chaotic.

\begin{figure}
\begin{center}
\includegraphics[width=14cm,height=8cm]{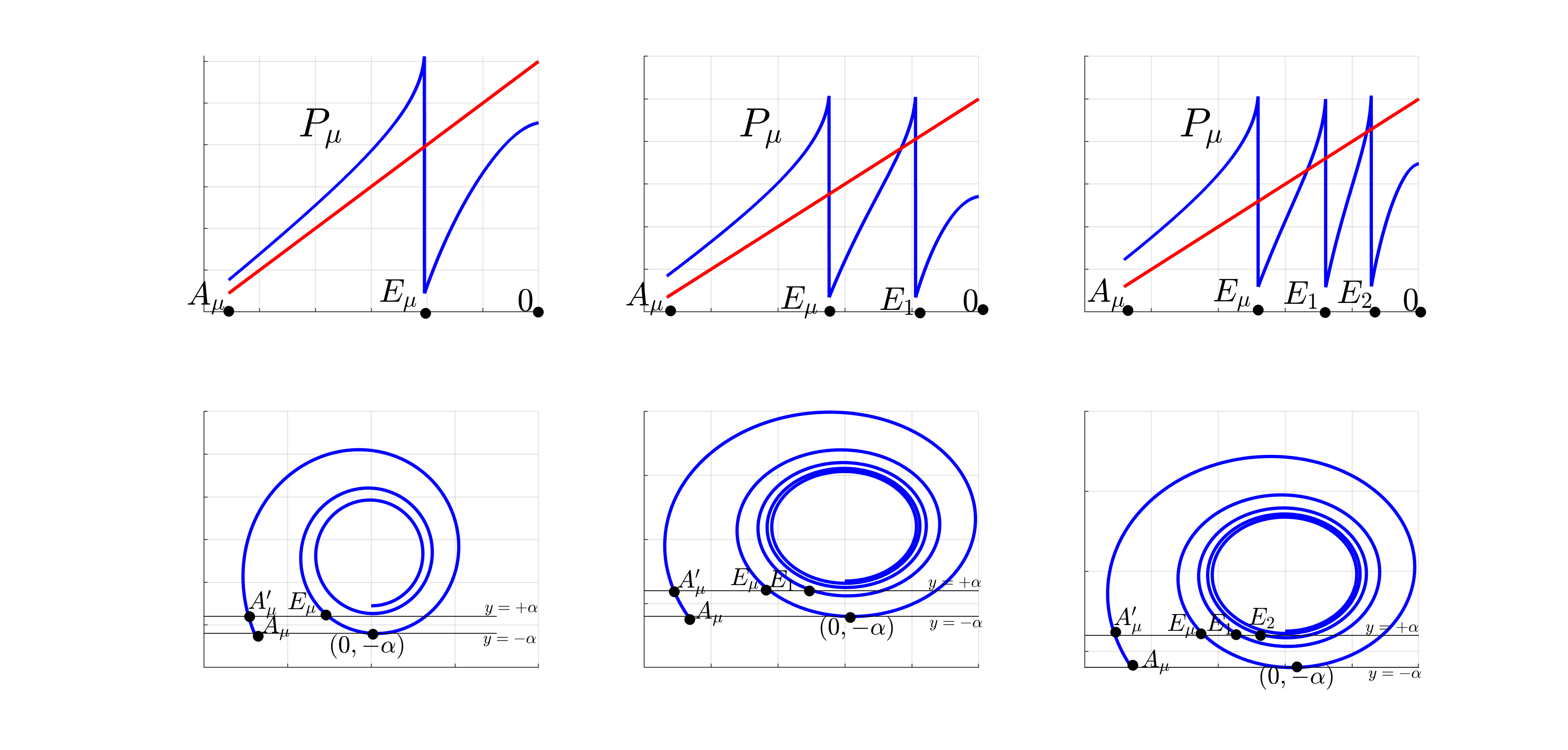}
\caption{The orbits passing trough $(0,-\alpha)$ for negative times with one, two and three cuts on $y=+\alpha$, and their relative Poincar\'{e} maps with one, two and three singularities. The picture is made with the hysteretic regularization of the system \ref{exemple} scaled by $y/\alpha$ and with $\kappa=0.047$, $\alpha=0.25$ and $\mu=1$ }
\label{fig:elstalls}
\end{center}
\end{figure}

In the next section we will find the relation between $\mu$ and $\alpha$ small enough, $\mu=\mu_1(\al)$  such that there is only one cut, and we will prove that in this case the hysteretic system still presents chaos.
We will see it finding $k\ge 1$ such that $P_{\mu_1}^k$ has two sub-intervals forming a horseshoe graph (see \cite{GlendinningJ2019})  and therefore the map $P_{\mu_1}$ is conjugated to a shift of two symbols.
In Figure \ref{fig:horseshoedetall} these intervals are shown.
The case $\mu=\mu_N(\al)$ where one has $N$ cuts is analogous but one should have $N+1$ intervals and therefore a shift of $N+1$ symbols.

This section is devoted to prove the existence of chaos in the case that the orbit in backward time through the point $(E_\mu,\al)$ does not cut the section $y=\al$ anymore.
Equivalently we can consider the orbit through $(0,-\al)$ (recall that the forward orbit of $(E_\mu,\al)$ is tangent to $y=-\al$ at this point) and compute $\pi _\mu^{-1}((0,-\al))$ where $\pi_\mu$ is the Poincar\'{e} map associated to the periodic orbit $\Gamma_\mu$ defined on the section $x=0$ in a neighborhood $\II$ of the point $y=\mu$:
\begin{equation}\label{eq:poincarex0}
\begin{split}
\pi_\mu :\{0\}\times \II  & \to \{x=0\} \\
(0,y) & \mapsto (0, \pi_\mu (y))  , \ \pi_\mu(\mu)=\mu
\end{split}
\end{equation}
Then we have
\[
\begin{split}
\pi_\mu^{-1}(-\al)&=\pi^{-1}_\mu(\mu)+(\pi_\mu^{-1})'(\mu)(-\al-\mu)+\OO((-\al-\mu)^2)\\
&=
\mu +(\pi_\mu^{-1})'(\mu)(-\al-\mu)+\OO((-\al-\mu)^2)
\end{split}
\]
Then $\pi^{-1}(-\al)>\al$ if
\[
\mu > \sigma\al
\]
where $\sigma$ is any value such that $\sigma >\frac{1+(\pi_\mu^{-1})'(\mu)}{1-(\pi_\mu^{-1})'(\mu)}$.
\subsection{The case  $\mu = \mu_1(\al)$; Finding chaos for a single singularity}
\begin{figure}
\begin{center}
\includegraphics[width=9cm,height=5cm]{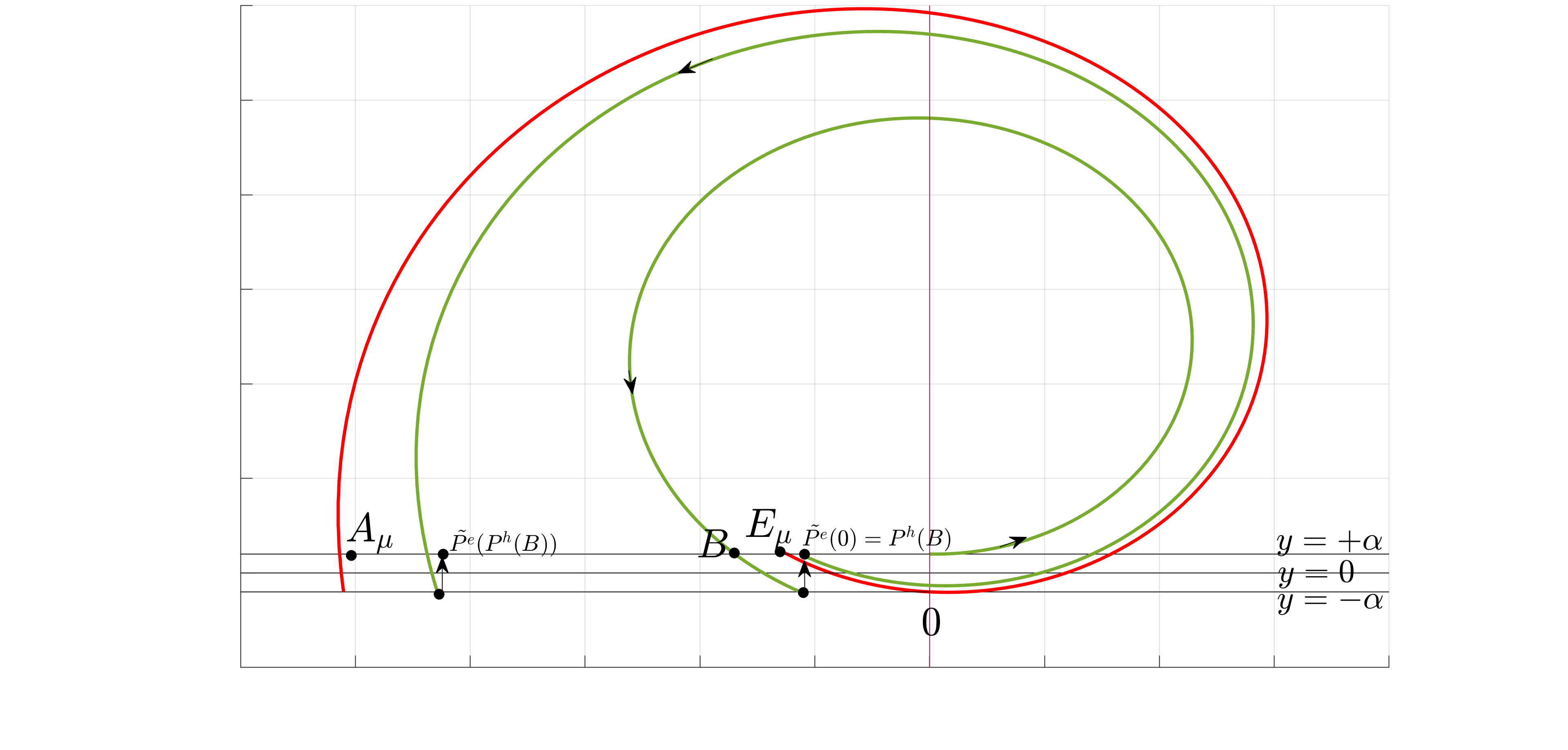}
\caption{The definition of the Poincar\'{e} map $P_\mu$, depending if $x\in [A_\mu,E_\mu]$ or $x\in [E_\mu,0]$. The point $B$ is the intersection with $\SSS_\al^-$ of the forward orbit of system $X_\mu^+$ from $(0,+\alpha)$ }
\label{fig:p_h}
\end{center}
\end{figure}
Recall that  the value $(\pi_\mu^{-1})'(\mu)$ is independent of $\mu$ and therefore  $(\pi_\mu^{-1})'(\mu)=(\pi^{-1})'(0)=\frac1{\pi'(0)}<1$, where $\pi$ is the Poincar\'{e} map given in \eqref{eq:poincaremap}.
Therefore, from now on in this section, we consider $\mu$ and $\al$ related by:
\begin{equation}\label{eq:mu1}
\mu =\mu_1(\al)=\sigma_1 \al, %\quad
\end{equation}
where
\begin{equation}\label{eq:sigma1}
\sigma_1 >%2
\frac{1+(\pi ^{-1})'(0)}{1-(\pi ^{-1})'(0)}=%2
\frac{\pi'(0)+1}{\pi'(0)-1}>2.%4.
\end{equation}
In the sequel, we will change the condition on $\sigma_1$ several times, but a finite number of them.
In this case the Poincar\'{e} map $P_\mu$:
\[
P_\mu: [A_\mu, 0]\times \{y=\al\}\subset \SSS_\al^- \to
\SSS_\al^-
\]
has only one singularity at $E_\mu$
%, like Figure \ref{fig:horseshoedetall},
and our goal is to define and obtain asymptotic formulas for it.
Recall that the point
$(A_\mu',\al)$ corresponds to the first cut of the positive orbit of the point $(E_\mu, \al)$ with $\SSS^-_\al$.
Equivalently, it is the first cut of the positive orbit of the point  $(0,-\al)$ with $\SSS^-_\al$.
The point $(A_\mu,\al)$ is the image of $(A'_\mu,\al)$ through the hysteretic map $P^h$.
This will be important to study the map $P_\mu$ that, as we did  in the previous sections, will be constructed as a combination of two maps: the exterior  return map $\tilde \pi$, and the "interior" hysteretic map $P^h$.
More concretely, consider the maps:
\begin{equation}\label{eq:pitilde}
\begin{split}
{\tilde \pi} :[E_\mu,0]\times \{y=\al\}  & \to [A'_\mu,0]\times\{y=\alpha\}\\
(x,\al) & \mapsto ({\tilde \pi} (x),\al)
\end{split}
\end{equation}
%}
defined following the flow of $\X_\mu$ until its first cut with $\SSS_\al^-$, and
\begin{equation}\label{eq:phisteretic}
\begin{split}
P^h:[A_\mu',E_\mu]\times\{y=\alpha\}  & \to [A_\mu,0]\times\{y=\alpha\}\\
(x,\al) & \mapsto (P^h(x),\al)
\end{split}
\end{equation}
defined by the hysteretic process determined by the fields $\X_\mu$ and $\Y_\mu=(0,1)$.
Next proposition, whose proof is deferred to section \ref{sec:provapropos}, gives the main properties and asymptotic formulas for the map $P_\mu$:

\begin{proposition}\label{prop:histereticmap}
Take $\mu=\mu_1(\al)$ as  \eqref{eq:mu1} with $\sigma_1$ satisfying \eqref{eq:sigma1}.
Then, the Poincar\'{e} map
\[
 P_\mu : [A_\mu,0]\times\{y=\alpha\} \to [A_\mu,0]\times\{y=\alpha\}
\]
is given by:
\begin{itemize}
 \item
For $x \in [A_\mu, E_\mu]$ we have:
\[
P_\mu (x)=P^h(x)=-\sqrt{-2\al +x^2}+\OO(\al),
\]
\item
For  $x\in (E_\mu,0]$ we have:
\[
\begin{split}
P_\mu (x) =P^h({\tilde \pi}(x)) &=-\sqrt{-2\al +(\al-\mu)(1-(\pi')(0))+(\pi')(0)x^2(1+\OO(\sqrt{\al})}+\OO(\al)\\
&=-\sqrt{-2\al+\al(\sigma_1-1)( (\pi')(0)-1)+(\pi')(0)x^2}+\OO(\al)
\end{split}
 \]
 \end{itemize}
Moreover,
\begin{itemize}
\item
$P_\mu(E_\mu)=0 $
\item
$\lim\limits_{x \to E_\mu^{-}}(P_\mu(x))=0 $,
$\lim\limits_{x \to E_\mu^{+}}(P_\mu(x))=A_\mu $
\item
$P_\mu(0)=
-\sqrt{-2\al+\al(\sigma_1-1)( (\pi')(0)-1)}+\OO(\al)$.
\item
$
P_\mu (A_\mu) = -\sqrt{-2\al +\alpha(\sigma_1+1)((\pi)'(0)+1)}+\OO(\al)$
\item
$P_\mu'(0)=0$ and $P_\mu'(x)>0$ for any $x\in [A_\mu, 0)$
\end{itemize}
Moreover, if we assume the extra condition:
\begin{equation}\label{eq:sigma2}
\sigma_1 >2
\frac{1+(\pi^{-1})'(0)}{1-(\pi^{-1})'(0)}=2
\frac{\pi'(0)+1}{\pi'(0)-1}>4.
\end{equation}
\begin{figure}
\begin{center}
\includegraphics[width=9cm,height=4cm]{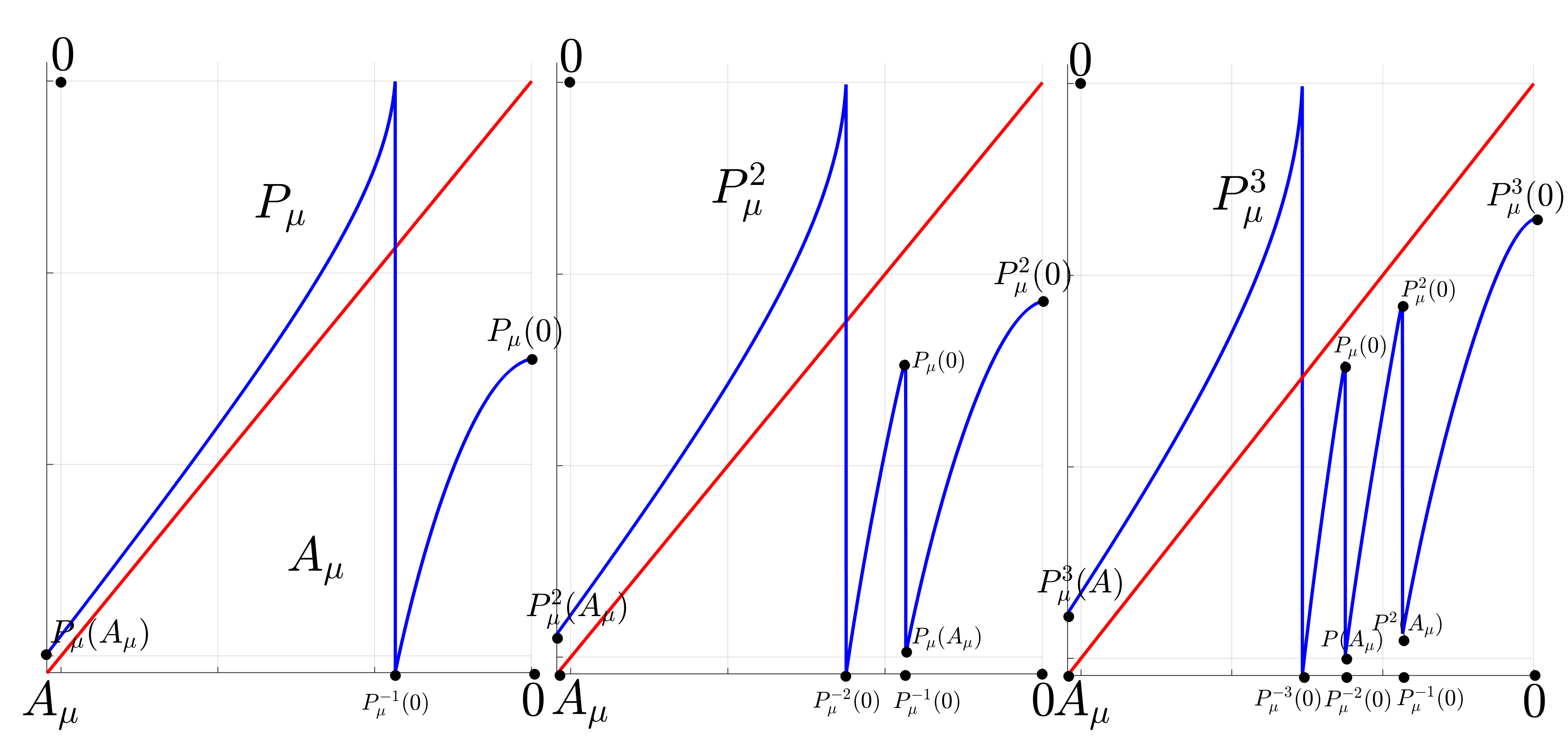}
\caption{The shape of the Poincar\'{e} map for $P_\mu$, $P_\mu^2$ and $P_\mu^3$. Note that $E_\mu=P_\mu^{-1}(0)=P_\mu^{-1}(A)$, expressing the jump discontinuity. The picture is made with the hysteretic regularization of the system \ref{exemple} scaled by $y/\alpha$ and with $\kappa=0.2$, $\alpha=0.05$ and $\mu=0.2$}
\label{fig:tresiteracions}
\end{center}
\end{figure}
then we have the relative position:
\begin{equation}\label{eq:relativeposition}
A_\mu<P_\mu(A_\mu)<P_\mu(0)<E_\mu
\end{equation}
\end{proposition}\label{prop:model}
To better understanding the behavior of the Poincar\'{e} map, in Figure \ref{fig:tresiteracions} are depicted the graphics of the maps $P_\mu$,$P_\mu^2$ and $P_\mu^3$.

Next proposition will show that a suitable iterate of the map $P_\mu$ has symbolic dynamics. The main idea is to prove that we can find  two intervals which "cover" each other. Then the results of \cite{GlendinningJ2019} give the existence of a horseshoe, and therefore symbolic dynamics and chaos.

The first observation is that  for $x \in [A_\mu,E_\mu]$ we have:
\begin{equation}
\begin{split}
P_\mu^{-n} (x)&=-\sqrt{x^2+2n\al}\\
\end{split}
\end{equation}
and,  whilst the forward orbit of $x \in [A_\mu,E_\mu] $ stays in this interval we also have:
\begin{equation}
\begin{split}
P_\mu^n (x)&=-\sqrt{x^2-2n\al}\\
\end{split}
\end{equation}
Observe that $P_\mu^{-1}(0)=E_\mu$, therefore
$P_\mu^{-n} (0)=P_\mu^{-n+1} (E_\mu)=\sqrt{2n\al}+\OO(\al)$.
\\
On the other hand, we observe that
\[
P_\mu([ P_\mu^{-1}( 0), 0])=P_\mu([E_\mu, 0])=[A_\mu, P_\mu(0)].
\]
Now, we consider  the subintervals
 \begin{equation}\label{subintervals}
I_n=[P_\mu^{-n}(0),P_\mu^{-(n-1)}(0)]=[-\sqrt{2n\alpha}+\OO(\al),-\sqrt{2(n-1)\alpha}+\OO(\al)]
\end{equation}
and we have:
\begin{proposition}\label{parellahorseshoe}
Take $\mu=\mu_1(\al)$ as  \eqref{eq:mu1} with $\sigma_1$ satisfying \eqref{eq:sigma2}.
Then, there exists natural numbers $n\ge 1$ satisfying the following condition:
 \begin{equation}\label{eq:condiciointervals}
\frac12 (\sigma_1-1)( (\pi')(0)-1) < n-1 < \frac12 (\sigma_1+1)( (\pi')(0)+1)
\end{equation}
Choose one of these numbers $n$.
Then, the  map $P_\mu^n$ fulfills the graph
\begin{equation}
I_n\rightarrow I_{n-1}\rightarrow I_n
\end{equation}
that is
\begin{equation}
I_{n-1} \subset P_\mu^n (I_n);  \ I_n \subset P_\mu^n (I_{n-1});
\end{equation}
and therefore $P_\mu^n$  has a horseshoe. Consequently there is a subset in $[A_\mu,0]$ where $P_\mu$ is conjugated to a shift of two symbols.
\end{proposition}
\begin{figure}
\begin{center}
\includegraphics[width=11cm,height=5cm]{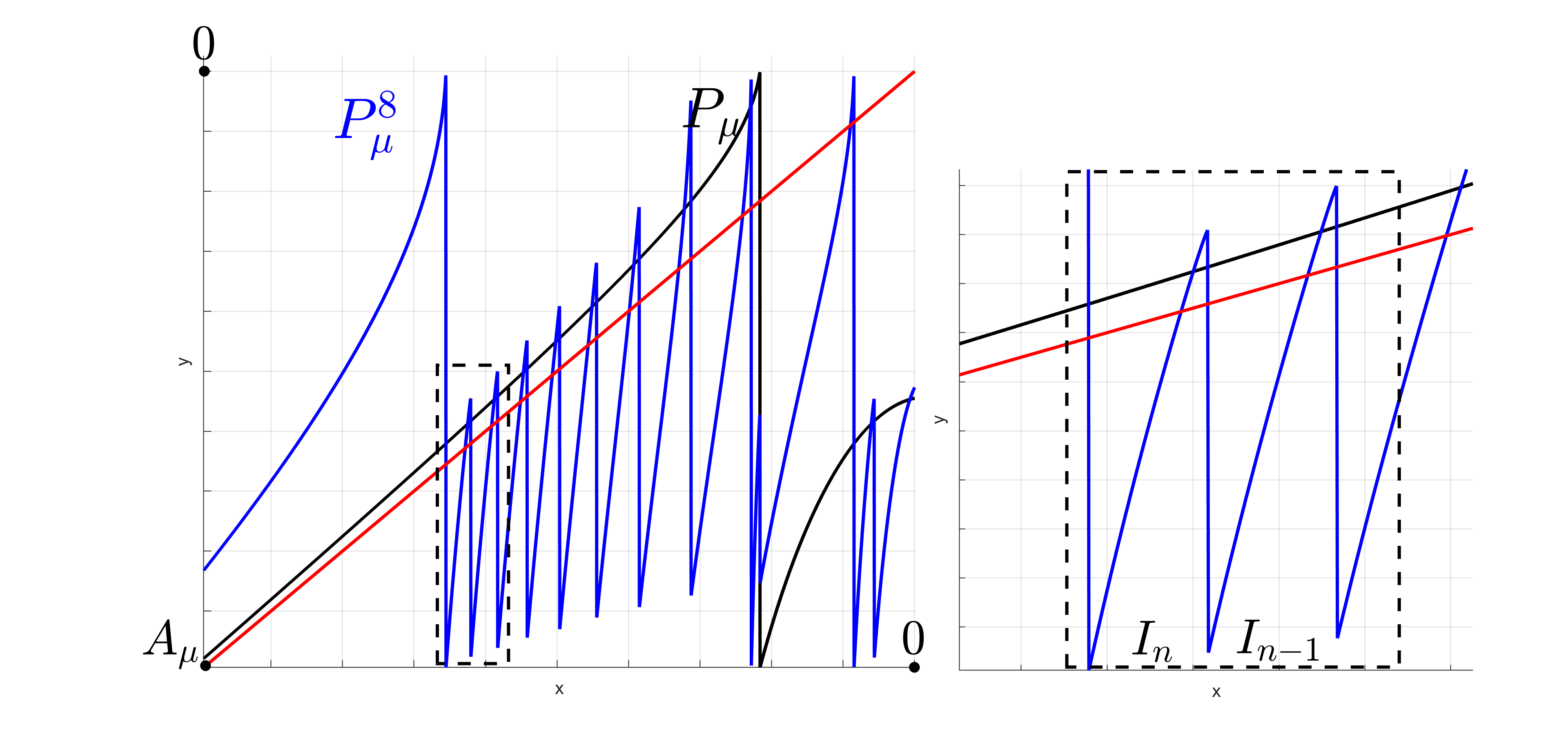}
\caption{The Poincar\'{e} map for $P_\mu$ and $P_\mu^n$ overlapped, and the two intervals forming a horseshoe pair. In this example $n=8$ with the parameters of Figure \ref{fig:tresiteracions}}
\label{fig:horseshoedetall}
\end{center}
\end{figure}

\begin{proof}
Observe that:
\[
I_{n-1}=[P_\mu^{-(n-1)}(0),P_\mu^{-(n-2)}(0)]=[-\sqrt{2(n-1)\alpha}+\OO(\al),-\sqrt{2(n-2)\alpha}+\OO(\al)]
\]
and:
$P_\mu^n (I_n)=[A_\mu,P_\mu(0)]$ and
$P_\mu^n (I_{n-1})=[P_\mu(A_\mu),P_\mu^2(0)]$, therefore we must prove (see Figure \ref{fig:horseshoedetall})
  \begin{itemize}
\item
$A_\mu<P_\mu^{-(n-1)}(0)$, equivalently $ -\sqrt{\al(\sigma_1+1)((\pi_0)'(0)+1)}<-\sqrt{2(n-1)\alpha}$
\item
$P_\mu(0)>P_\mu^{-(n-2)}(0)$, equivalently
$-\sqrt{-2\al+\al(\sigma_1-1)( (\pi'_0)(0)-1)}
>-\sqrt{2(n-2)\alpha}$
\item
$P_\mu(A_\mu)<P_\mu^{-(n)}(0)$, equivalently
$-\sqrt{-2\al +\alpha(\sigma_1+1)((\pi_0)'(0)+1)} <-\sqrt{2n\alpha}$
\item
$P_\mu^2(0)>P_\mu^{-(n-1)}(0)$, equivalently
$-\sqrt{-4\al+\al(\sigma_1-1)( (\pi'_0)(0)-1)}
>-\sqrt{2(n-1)\alpha}$
\end{itemize}
All these conditions are satisfied if we can ensure that there exists $n\ge 1$ such that:
\begin{equation}\label{eq:condiciointervals}
\frac12 (\sigma_1-1)( (\pi'_0)(0)-1) < n-1 < \frac12 (\sigma_1+1)( (\pi'_0)(0)+1)
\end{equation}
Observe that
\[
\frac12 (\sigma_1+1)( (\pi'_0)(0)+1)-\frac12 (\sigma_1-1)( (\pi'_0)(0)-1) =(\sigma_1+ (\pi'_0)(0))>2
\]
therefore we can always find  a number $n$ such that \eqref{eq:condiciointervals} is satisfied.
Then, $P^n$ contains a horseshoe, and chaotic dynamics is assured.
This concludes the proof.
\end{proof}

\section{Proofs}\label{sec:proofs}

\subsection{The exterior map $\f_{\mu,\eps}$, proof of Theorem  \ref{thm:extensiopie}}

This section is devoted to prove Theorem \ref{eq:dominipie}.

\begin{remark}\label{rem:Pdeltadependence}
Even if the map $\f_{\mu,\eps}$ depends on parameters,
in the sequel, we will keep in mind this dependence, but, as not being a matter of confusion we will write $\f$.
This remark also applies to the related objects defined.
\end{remark}

\begin{figure}
\begin{center}
\includegraphics[width=7cm,height=4.5cm]{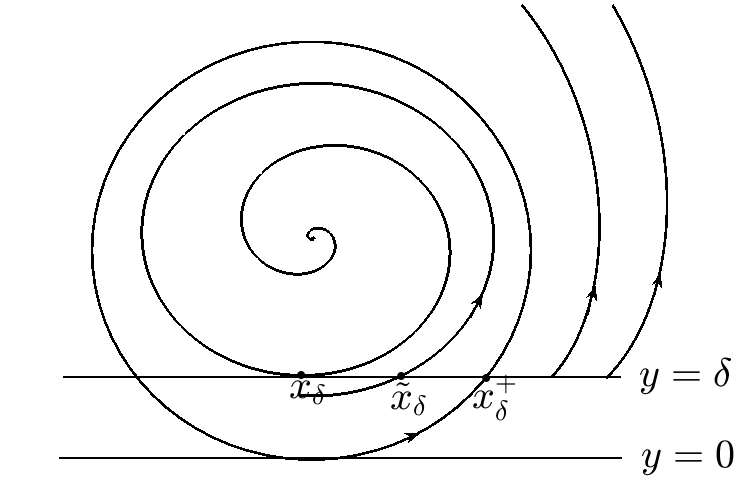}
\caption{$\f$, the periodic orbit of $X_0$ and its tangent orbit to $y=\delta$ and $ \tilde{x}_\delta$.}\label{fig:tangenciapexterior}
\end{center}

\end{figure}

Under our normalizations, we have that $\X_\mu(x,y)=\X_0(x,y-\mu)$, to study the exterior map $\tilde \pi$ for $X_\mu$ defined in \eqref{exteriornou} and \eqref{eq:dominipie}
is equivalent to  study the  map
\begin{equation}\label{exteriorbar}
\overline \f: \overline  \MM\times \{\de\}\subset  \SSS_{\de}^+ \to \SSS^{-}_\de , \quad \de=\eps-\mu,
\end{equation}
for the vector field $\X_0$.
Moreover, the formulas  for $\f(x)$ and $\overline \f(x)$ will be the same and the  domains $\overline  \MM=\MM$, for this reason we drop the bars and we use $\f$.

If we call $x_\delta$ the value such that the solution of $\X_0$ with initial condition $(x_\delta,\delta)$ is tangent to $\mathcal{S}_\delta$ and  define $ \tilde{x}_\delta$ the last cut of this solution (in backward time) with $\SSS_\delta^+$ before the tangency (see Figure \ref{fig:tangenciapexterior}), then, our normalizations imply that $x_\mu=x_\delta$ and $\tilde x_\mu=\tilde x_\delta$ (see \ref{hypofamily}), but during the proof we will use the $\delta$-notation to be consistent.

Then, $\f$ is defined in $\MM\times \{\de\}$, where
\begin{equation}\label{eq:dominipie}
\MM=[\tilde{x}_\delta,M]
\end{equation}
for some $M>0$, independent of  $\delta$ and $\f(\tilde x_\delta)=x_\delta$.
In fact, for our purposes, it will be enough to obtain information for $\f$ in a smaller domain of the form $[\tilde{x}_\de, \sqrt{\de C}]\subset \MM$ where $C>1$ is a constant independent of $\de$.

The proof will  be the consequence of propositions \ref{prop:exteriortangent} and  \ref{prop:extensiopie}.
First, in Proposition \ref{prop:exteriortangent}, we will study $\f$ for points slightly away from the point $\tilde x_\de$.
The local study near $\tilde x_\de$ is done in Proposition \ref{prop:extensiopie}.

The idea is to write $\f$ as:
\begin{equation}\label{descomposicioexterior}
\f (x) =(\Q)^{-1} \circ \pi \circ (\bar{\Q})^{-1}(x),
\end{equation}
where
$$
\Q: \SSS_\de^- \to \{(0,y), \ |y|\le y_0\} , \quad
\bar \Q: \{(0,y), \ |y|\le y_0\}  \to \SSS_\de^+
$$
are the maps derived by the orbits of $\X_0$ followed in positive time,
$\pi(y)$, is the Poincar\'{e} return map \eqref{eq:poincaremap} on $ \{(0,y), \ |y|\le y_0\}$ defined around  the periodic orbit $\Gamma_0$ of
$ \X_0$.
Next proposition gives the asymptotics for the maps $\Q$, $\bar{\Q}$ and their inverses in suitable domains:

\begin{proposition}\label{prop:exterior}
Consider any constant $C>1$.
Let $\X_0$ satisfying \eqref{def:Xg} and take any  $0<y_0<\delta$.
Then, if $\delta>0$ is small enough, the flow of $\X_0$ defines  diffeomorphisms:
\begin{itemize}

\item

\begin{eqnarray*}
\Q: [-\sqrt{\de C}-\sqrt{\de-y_0}]\times \{\delta\} \subset \SSS^-_\de  & \to & \{0\} \times
[\de (1-C) ,y_0]\\
(x,\delta) &\mapsto & (0, g(x))
\end{eqnarray*}
with:
\begin{equation}
\begin{array}{rcl}
\Q (x) &=& \delta[1-(\frac{x}{\sqrt{\delta}})^2+\OO(\sqrt{\delta})]=
\delta-x^2+\OO(\sqrt{\delta^{3}}) =\OO(\delta),
\\
\Q' (x) &=& \sqrt{\delta}[-2\frac{x}{\sqrt{\delta}}+\OO(\sqrt{\delta})]=
-2x+\OO(\delta)=\OO(\sqrt{\delta}),
\\
\Q '' (x) &=& -2+\OO(\sqrt{\delta})=\OO(1),
\end{array}
\end {equation}

and its inverse\\
\begin{eqnarray*}
{\Q}^{-1}: \{0\}  \times  [\de (1-C) ,y_0]  &\to &  [-\sqrt{\de C}-\sqrt{\de-y_0}]\times \{\delta\} \subset \mathcal{S}_\delta^-\\
(0, y) & \mapsto & (\Q^{-1}(y),\de)
\end{eqnarray*}
with

\begin{equation}
\begin{array}{rcl}
{\Q}^{-1}(y) &=& -\rdelta[\sqrt{1-\frac{y}{\delta}}+\OO(\sqrt{\delta})] =
-\sqrt{\delta-y}+\OO(\delta)= \OO(\sqrt{\delta}),
\\
{\Q}'^{-1}(y) &=&\frac{1}{\sqrt{\delta}}[\frac{1}{2\sqrt{{1-(\frac{y}{\delta})}}}+\OO(\sqrt{\delta})]=
[\frac{1}{2\sqrt{\delta-y}}+\OO(1)]= \OO(\frac{1}{\sqrt{\delta}}) , \\
{\Q}''^{-1}(y) &=& \frac{1}{(\sqrt{\delta})^3}[\frac{1}{4 \sqrt{(1-\frac{y}{\delta})^3}}+\OO(\sqrt{\delta})]=
\frac{1}{4 \sqrt{(\delta-y)^3}}+\OO(\frac{1}{\delta})
= \OO(\frac{1}{(\sqrt{\delta})^3}),
\end{array}
\end{equation}
\item
\begin{eqnarray*}
\bar{\Q} : \{0\}\times[\de(1-C) ,y_0]  & \to   [\sqrt{\de-y_0}, \sqrt{\de C}]\times \{\delta\} \subset & \mathcal{S}_\delta^+\\
(0, y) & \mapsto & (\bar{\Q}(y),\de)
\end{eqnarray*}
with

\begin{equation}
\begin{array}{rcl}
\bar{\Q}(y) &=& \rdelta[\sqrt{1-\frac{y}{\delta}}+\OO(\sqrt{\delta})] =
\sqrt{\delta-y}+\OO(\delta)= \OO(\sqrt{\delta}), %\quad\forall y\in[-y_0,y_0]
\\
\bar{\Q}'(y) &=&\frac{1}{\sqrt{\delta}}[-\frac{1}{2\sqrt{{1-(\frac{y}{\delta})}}}+\OO(\sqrt{\delta})]=[-\frac{1}{2\sqrt{\delta-y}}+\OO(1)]= \OO(\frac{1}{\sqrt{\delta}}) , %\quad\forall y\in[-y_0,y_0]
\\
\bar{\Q}''(y) &=& \frac{1}{(\sqrt{\delta})^3}[-\frac{1}{4 \sqrt{(1-\frac{y}{\delta})^3}}+\OO(\sqrt{\delta})]=
-\frac{1}{4 \sqrt{(\delta-y)^3}}+\OO(\frac{1}{\delta})
= \OO(\frac{1}{(\sqrt{\delta})^3}) ,
\end{array}
\end{equation}
and its inverse:\\
\begin{eqnarray*}
{\bar{\Q}}^{-1}:  [\sqrt{\de-y_0},\sqrt{\de C}]\times \{\delta\}\subset \SSS^+_\de
 &\to &    \{0\}\times[\de(1-C), y_0] \\
(x,\de) &\mapsto & (0,\bar\Q ^{-1}(x))
\end{eqnarray*}
with:
\begin{equation}
\begin{array}{rcl}
({\bar{\Q}}^{-1})(x) &=& \delta[1-(\frac{x}{\sqrt{\delta}})^2+\OO(\sqrt{\delta})]=
\delta-x^2+\OO(\sqrt{\delta^{3}}) =\OO(\delta),
\\
({\bar{\Q}}^{-1})'(x) &=& \sqrt{\delta}[-2\frac{x}{\sqrt{\delta}}+\OO(\sqrt{\delta})]=
-2x+\OO(\delta)=\OO(\sqrt{\delta}),
\\
({\bar{\Q}}^{-1})''(x) &=& -2+\OO(\sqrt{\delta})=\OO(1),
\end{array}
\end {equation}
\end{itemize}

\end{proposition}

\begin{proof}
We will do the computations for $\bar\Q$.
The ones for $\Q$ are analogous.
First recall the normal form of $\X_0$:

\begin{equation}\label{normalformX0}
\X_0(x,y)=\left(\begin{array}{l}
        1 + f_1(x,y)\\
        2x + by +f_2(x,y)
       \end{array}\right)
\end{equation}
where $f_i(x,y)=O_i(x,y)$ and $f_2(x,0)=0$.
Recall that the periodic orbit $\Gamma_0$ is tangent to $\Sigma$ at  $(0,0)$

Near $y=0$ we perform the change
\begin{equation}\label{canvidelta}
\bar{x}= \frac{x}{\sqrt\delta};\quad\bar{y}=\frac{y}{\delta}
\end{equation}
Then vector field \eqref{normalformX0} transforms to the system:

\begin{equation}\label{deltaregularized}
\begin{array}{rcl}
\sqrt{\delta}\dot{\bar{x}} &=&
1+\OO(\sqrt{\delta}\bar{x},\delta\bar{y})\\
\sqrt{\delta}\dot{\bar{y}} &=&
2\bar{x}+\OO(\sqrt{\delta}{\bar{x}}^2,\delta\bar{x}\bar{y},\sqrt{\delta}\bar{y})
\end{array}
\end{equation}
system that for $\delta\ne0$ has the same orbits than

\begin{equation}\label{deltadifferentiable}
\begin{array}{rcl}
\dot{\bar{x}} &=&
1+\OO(\sqrt{\delta}\bar{x},\delta\bar{y})\\
\dot{\bar{y}} &=&
2\bar{x}+\OO(\sqrt{\delta}{\bar{x}}^2,\delta\bar{x}\bar{y},\sqrt{\delta}\bar{y})
\end{array}
\end{equation}
and we will study the scaled map $\bar \Q_r$ associated to the vector field \eqref{deltadifferentiable} and its inverse:
\begin{eqnarray*}
\bar \Q_r: \{0\} \times [1- C, \bar y_0] & \to & \SSS_1^+ \\
(0, \bar y) & \mapsto & (\bar \Q_r( \bar y), 1)
\end{eqnarray*}
where $\bar y_0=\frac{y_0}{\delta}$ and therefore $0<\bar y_0<1$.
and
\begin{eqnarray*}
\bar \Q_r^{-1}: [\sqrt{1-\bar{y}_0},\sqrt{C}]\times \{1\} \subset \SSS_1^+  &\to &  \{0\} \times [1-C, \bar y_0]\\
(\bar x,1) & \mapsto & (0, \bar \Q_r^{-1}(\bar x))
\end{eqnarray*}

Clearly we have
$$
\bar \Q(y)=\sqrt{\delta}\bar \Q_r(\frac{y}{\delta}), \quad \bar\Q^{-1}(x)=\delta \bar\Q_r^{-1}(\frac{x}{\sqrt{\delta}})
$$
Observe that  system \eqref{deltadifferentiable} is a regular $\OO(\sqrt{\delta})$- perturbation of the system:

\begin{equation}\label{deltadifferentiable0}
\begin{array}{rcl}
\dot{\bar{x}} &=&
1\\
\dot{\bar{y}} &=&
2\bar{x}
\end{array}
\end{equation}
therefore we can use the theorem of regularity to initial conditions and parameters to study the scaled maps as a regular perturbation of the ones in this simpler system.
Is clear that in this system, near $(0,0)$, and for any fixed $ \bar{y}_0<1$ can be defined the scaled map
\begin{equation}
\bar{\Q}_{r,0}(\bar{y})=+\sqrt{1-\bar{y}},\quad \forall\bar{y}\in[1-C,\bar{y}_0],
\end {equation}
and its inverse

\begin{equation}
{\bar{\Q}_{r,0}}^{-1}(\bar{x})=1-{\bar{x}}^2,\quad \forall\bar{x}\in[\sqrt{1-\bar{y}_0},\sqrt{C}]\\
\end {equation}
The derivatives are:
\begin{equation}
\begin{array}{lcr}
\bar{\Q}_{r,0}'(\bar{y})=-\frac{1}{2\sqrt{1-\bar{y}}},\quad\bar{\Q}_{r,0}''(\bar{y})= -\frac{1}{4(\sqrt{{1-\bar{y}}})^3} ,\quad\forall\bar{y}\in[1-C,\bar{y}_0]\\
\quad({\bar{\Q}_{r,0}}^{-1})'(\bar{x})=-2\bar{x},\quad ({\bar{\Q}_{r,0}}^{-1})''(\bar{x})=-2,\quad\forall\bar{x}\in[\sqrt{1-\bar{y}_0},\sqrt{C}]
\end{array}
\end {equation}

Then for $\delta$ small enough, there exists perturbed scaled maps $\bar{\Q}_r(\bar{y})$ and  ${\bar{\Q}_r}^{-1}(\bar{x})$, defined in the same sections, with:
\begin{equation}\label{pperturbat}
\begin{array}{lcr}
\bar{\Q}_r(\bar{y})=\bar{\Q}_{r,0}(\bar{y})+\OO(\sqrt{\delta}); \quad
\bar{\Q}_r '(\bar{y})=\bar{\Q}_{r,0}'(\bar{y})+\OO(\sqrt{\delta});\quad
\bar{\Q}_r''(\bar{y})=\bar{\Q}_{r,0}''(\bar{y}) +\OO(\sqrt{\delta}) ,\quad
\forall \bar{y}\in[1-C,\bar{y}_0]\\
\\
({\bar{\Q}_r}^{-1})(\bar{x})=({\bar{\Q}_{r,0}}^{-1})(\bar{x})+\OO(\sqrt{\delta});\quad
({\bar{\Q}_r}^{-1})'(\bar{x})=({\bar{\Q}_{r,0}}^{-1})'(\bar{x})+\OO(\sqrt{\delta});\\
\quad ({\bar{\Q}_r}^{-1})''(\bar{x})=({\bar{\Q}_{r,0}}^{-1})''(\bar{x})+\OO(\sqrt{\delta}), \quad
\forall \bar{x}\in[\sqrt{1-\bar{y}_0},\sqrt{C}]
\end{array}
\end {equation}
Returning to the $x,y$ variables we have
\begin{equation}\label{eq:returnxy}
\begin{array}{lcr}
\bar{\Q}(y)=\sqrt{\delta}{\bar{\Q}_\delta}(\frac{y}{\delta}),\quad\forall y\in[\de(1-C),y_0]\\
\bar{\Q}^{-1}(x)=\delta{\bar{\Q}_\delta}^{-1}(\frac{x}{\sqrt{\delta}}),\quad   \forall x   \in
[\rdelta\sqrt{1-\bar{y}_0},\rdelta\sqrt{C}]=  [\sqrt{\de-y_0},\sqrt{\de C}]
\end{array}
\end {equation}
where $ y_0:=\delta\bar{y}_0 $.
Differentiating, we will have
\begin{equation}
\begin{array}{lcr}
\bar{\Q}'(y)=\frac{1}{\sqrt\delta}\bar{\Q}'_\delta(\frac{y}{\delta}),\quad\bar{\Q}''(y)=\frac{1}{({\sqrt\delta})^3}\bar{\Q}''_\delta(\frac{y}{\delta})\quad\forall y\in[\de(1-C) ,y_0]\\
(\bar{\Q}^{-1})'(x)=\sqrt{\delta}({\bar{\Q}_\delta}^{-1})'(\frac{x}{\sqrt{\delta}}),\quad(\bar{\Q}^{-1})''(x)=({\bar{\Q}_\delta}^{-1})''(\frac{x}{\sqrt{\delta}}) \quad\forall x  \in
 [\sqrt{\de-y_0},\sqrt{\de C}]
\end{array},
\end {equation}
and Proposition \ref{prop:exterior} easily follows.

\end{proof}

Next lemma gives the asymptotic expression of $x_\de$ and $\tilde x_\de$ and therefore it is useful to
understand the domain of the map $\f$.
\begin{lemma}\label{prop:exteriortangent}
Let $\X_0$ with the hypothesis \eqref{hypofamily} and $x_\delta$ such that the solution with initial condition $(x_\delta,\delta)$ is tangent to $\mathcal{S}_\delta$.
Recall that $ \tilde{x}_\delta$ is the last cut of this solution (in backward time) with $\SSS_\delta^+$ before the tangency (see Figure \ref{fig:tangenciapexterior}).
Then,
$ x_\delta$ and  $ \tilde{x}_\delta$ satisfy:
\begin{equation}
x_\delta=-\frac{b}{2}\delta +\OO(\delta^2), \quad \tilde{x}_\delta=\rdelta\sqrt{1-\frac{1}{\pi'(0)}}+\OO(\delta)
\end{equation}
 \end{lemma}
\begin{proof}
As $\X_0$ has the form  \eqref{def:Xg} it is clear that $x_\delta=-\frac{b}{2}\delta +\OO(\delta^2)$.
Also if $\delta$ is small enough, the flow of $X_0$ is also a fold on $(x_\delta,\delta)$, therefore, the intersection of the solution issuing from it will cut $x=0$ at
\begin{equation}\label{eq:ydelta}
y_\delta=\delta+\OO(\delta^2).
\end{equation}
To compute $\tilde{x}_\delta$, we use that it satisfies of the equation
$$
\pi(\bar{\Q}^{-1}(\tilde{x}_\delta))=y_\delta=\delta (1+\OO(\delta)).
$$
Writing $\tilde x_\de=\sqrt{\delta}\bar x_\de$, using  the expression of $\bar \Q$ given in Proposition \ref{prop:exterior}  and Taylor expanding the return map $\pi$ around $x=0$ one obtains:
$$
\pi(\bar{\Q}^{-1}(\sqrt{\delta}\bar x_\de ))=
\pi (\delta -\delta \bar x_\de ^2+\OO(\delta^{3/2}))=
\pi'(0) (\delta -\delta \bar x_\de ^2+\OO(\delta^{3/2}))+ \OO(\delta^2)=
\delta[(\pi'(0)[1 - \bar x_\de^2]+\OO(\delta^{1/2})]
$$
solving
$$
\delta[(\pi'(0)[1 - \bar x_\de^2]+\OO(\delta^{1/2})]=\delta(1 +\OO(\delta))
$$
one obtains the result.
\end{proof}

Next step is to give an asymptotic formula for  $\f$
% is defined for $x\in \MM=[\tilde{x}_\delta,M]$, we want to obtain an asymptotics for it
that allows us to prove that $(\f)''(x)>0$.
Observe that, by definition $\f(\tilde x_\de)=x_\de$, and one can easily extend $\f$ to $[x_\de,\tilde x_\de]$ by the constant function $\f (x)=x_\de$, for any $x\in[x_\de,\tilde x_\de]$.

The main difficulty will be to obtain an asymptotics of $\f$ for $x\ge \tilde x_\delta$, very close to $\tilde x_\delta$.

As a first step, in next proposition, we obtain the asymptotics of $\f$ for points on the right of $\tilde x_\de$ but strictly separated from it.
This will allow us to prove that, for this range of points, $(\f)'' >0$.
\begin{proposition}\label{prop:exteriorderivada}
Let  $C>1$ be any constant.
Take  $\X_0$ with the hypothesis \eqref{hypofamily}, and fix $0<\bar{y}_0<\frac{1}{\pi'(0)}<1$.
Then, if $\delta>0$ small enough we have  %$ \f $ is defined in the interval
$[\sqrt{\delta}\sqrt{1-\bar{y}_0},\sqrt{\delta C}
%\sqrt{1+\bar{ y}_0}
]
\subset [\tilde x_\delta, \sqrt{\delta C}
%\sqrt{1+\frac{1}{ \pi'(0)}}
]\subset \MM$, where $\MM$ is  the domain of $\f$ defined in \eqref{eq:dominipie}, and
$\forall x\in[\sqrt{\delta}\sqrt{1-\bar{y}_0},\sqrt{\delta C}]$
\begin{equation}\label{eq:formulapiebona}
\begin{array}{rcl}
\f(x)  &=&
\sqrt{\delta- \pi'(0)(\delta-x^2)}+\OO(\delta), \\
(\f)'(x)  &=&
-\frac{\pi'(0)x}{\sqrt{\delta-\pi'(0)(\delta-x^2)}}+\OO(\sqrt{\delta}) \\
(\f)''(x) &=&
\frac{-\de \pi'(0)(1-\pi'(0))}{\sqrt{(\delta- \pi'(0)(\de-x^2))^3}  }+\OO(1) >0
\end{array}
\end{equation}
Consequently:

\begin{equation}
(\f)''(x)>0, \forall x\in[\sqrt{\delta}\sqrt{1-\bar{y}_0},\sqrt{\delta C}]
\end{equation}
\end{proposition}

\begin{proof}

As we have decomposed $ \f(x)=(\Q)^{-1}\circ\pi\circ(\bar{\Q})^{-1}(x) $, (see \ref{descomposicioexterior}),
we will have:

\begin{equation}\label{eq:Pedd}
\begin{split}
(\f)'   & =(\Q^{-1})'\pi'(\bar{\Q}^{-1})'  \\
(\f)''  &=(\Q^{-1})'' (\pi')^2 ((\bar{\Q}^{-1})')^2
+(\Q^{-1})' \pi'' ((\bar{\Q}^{-1})')^2
+(\Q^{-1})'  \pi'  (\bar{\Q}^{-1})''
\end{split}
\end{equation}
For $x\in[\sqrt{\delta}\sqrt{1-\bar{y}_0},\sqrt{\delta C}]$
where all the functions are evaluated in the respective argument according the chain rule.
For instance
$(\Q^{-1})''\equiv(\Q^{-1})''(\pi(\bar{\Q}^{-1}(x)))$, etc

To obtain the asymptotic expression of these formulas as $\delta\rightarrow0 $  for
$ x\in[\sqrt{\delta}\sqrt{1-\bar{y}_0},\sqrt{\delta C}]$,
we will apply Proposition \ref{prop:exterior} which allows us to reduce the calculation to the dominant terms of these expressions:
\begin{equation}
\begin{array}{rcl}
\f(x)  &=& \Q^{-1}\circ\pi\circ{\bar{\Q}}^{-1}(x)
=\sqrt {\delta- \pi (\bar\Q^{-1}(x))}+\OO(\delta)=
\sqrt{\delta- \pi(\delta-x^2+\OO(\delta^{3/2}))}+\OO(\delta)\\
&=&
\sqrt{\delta- \pi'(0)(\delta-x^2)+\OO(\delta^{3/2})}+\OO(\delta)
=
\sqrt{\delta- \pi'(0)(\delta-x^2)}+\OO(\delta)
\end{array}
\end{equation}
as $ x\in[\sqrt{\delta}\sqrt{1-\bar{y}_0},\sqrt{\delta C}]$
then
$ \de(1-C) <\delta-x^2<\de \bar{y}_0<\frac{\de}{\pi'(0)}$, and
$ \de-\pi'(0)(\delta-x^2)>0$, and $\f$ is defined.

For $(\f)'$, first we calculate separately the dominant terms of the three factors using Proposition \ref{prop:exterior}:
\begin{equation*}
\begin{array}{rcl}
(\Q^{-1})'(\pi(\bar \Q^{-1}(x)))&=&
\frac{1}{2\sqrt{\delta-\pi(\delta-x^2+\OO(\delta^{3/2}))}}+\OO(1)=
\frac{1}{2\sqrt{\delta-\pi'(0)(\delta-x^2)+\OO(\delta^{3/2})}}+\OO(1)\\
&=&
\frac{1}{2\sqrt{\delta}\sqrt{1-\pi'(0)(1-\frac{x^2}{\delta})+\OO(\delta)}}+\OO(1)
=\frac{1}{2\sqrt{\delta}\sqrt{1-\pi'(0)(1-\frac{x^2}{\delta})}}+\OO(1)\\
&=& \frac{1}{2\sqrt{\delta-\pi'(0)(\delta-x^2)}}+\OO(1)\\
\pi'(\bar{\Q}^{-1}(x)) &=& 
\pi'(\delta-x^2+\OO(\delta^{3/2}\sqrt{\delta}))=\pi'(0)+\OO(\delta).\\
(\bar{\Q}^{-1})'(x) &=&
-2x+\OO(\delta)
\end{array}
\end{equation*}

Finally we obtain:

\begin{equation}\label{eq:exteriorderivada}
\begin{array}{rcl}
(\f)'(x)  &=&
\left(\frac{1}{2\sqrt{\delta-\pi'(0)(\delta-x^2)}}+\OO(1)\right) \left(\pi'(0)+\OO(\delta)\right)\left(-2x+\OO(\delta)\right)\\
&=&
-\frac{\pi'(0)x}{\sqrt{\delta-\pi'(0)(\delta-x^2)}}+\OO(\sqrt{\delta})
\end{array}
\end{equation}

Analogously, we proceed with $(\f)''$ using formula \eqref{eq:Pedd}.
We compute the asymptotics of the three terms in \eqref{eq:Pedd} using Proposition \ref{prop:exterior}.

\begin{equation*}
\begin{split}
&(\Q^{-1})'' (\pi(\bar \Q^{-1}(x)) (\pi'(\bar \Q^{-1}(x)))^2  ((\bar{\Q}^{-1})'(x))^2  \\
&=
\left(\frac{1}{4\sqrt{(\delta- \pi)^3}  } +\OO(\frac{1}{\delta})\right)
\left( \pi'(0)^2+\OO(\delta^{3/2})\right)
\left(4x^2+\OO (\delta^{3/2})\right)\\
&=  \frac{\pi'(0)^2 x^2}{\sqrt{(\delta- \pi)^3}  }+\OO(1)
\end{split}
\end{equation*}
where
\begin{equation}\label{piabreviat}
\pi=\pi(\bar \Q^{-1}(x))=\pi(\de-x^2+\OO(\de^{3/2}))=\pi'(0)(\de-x^2)+\OO(\de^{3/2})
\end{equation}
Similarly
\begin{equation*}
\begin{split}
&(\Q^{-1})' (\pi(\bar \Q^{-1}(x)) (\pi'(\bar \Q^{-1}(x))) (\bar{\Q}^{-1})''(x)\\
&=
\left(\frac{1}{2\sqrt{\de-\pi}  }+\OO(\de)\right)
\left(\pi'(0)+\OO(\de^{3/2})\right)
\left(-2+\OO(\rdelta)\right)\\
&=\frac{-\pi'(0)}{\sqrt{\de-\pi}  }+\OO(1),
\end{split}
\end{equation*}
and
\begin{equation*}
\begin{split}
&(\Q^{-1})' (\pi(\bar \Q^{-1}(x))   (\pi'' (\bar \Q^{-1}(x))) ((\bar{\Q}^{-1})')^2\\
&=
\left(\frac{1}{2\sqrt{\de-\pi}  }+\OO(\de)\right)
\left(\pi''(0)+\OO(\de^{3/2})\right)
\left(4x^2+\OO (\delta^{3/2})\right)\\
&=\frac{2x^2}{\sqrt{\de-\pi}  }+\OO(\de)
\end{split}
\end{equation*}

Using the previous formulas we obtain, by \eqref{eq:Pedd}, recalling \eqref{piabreviat} and that $x=\OO(\rdelta)$:
\begin{equation}\label{eq:exteriorderivada2}
\begin{array}{rcl}
(\f)'' &=& (\Q^{-1})'' (\pi')^2 ((\bar{\Q}^{-1})')^2
+(\Q^{-1})' \pi'' ((\bar{\Q}^{-1})')^2
+(\Q^{-1})'  \pi'  (\bar{\Q}^{-1})''\\
&=&
\frac{\pi'(0)^2 x^2}{\sqrt{(\delta- \pi)^3}  }+\OO(1)+
\frac{2x^2}{\sqrt{\de-\pi}  }+\OO(\de)-
\frac{\pi'(0)}{\sqrt{\de-\pi}  }+\OO(1)\\
&=&
\frac{\pi'(0)^2 x^2}{\sqrt{(\delta- \pi)^3}  }-
\frac{\pi'(0)}{\sqrt{\de-\pi}  }+\OO(1)=
\frac{\pi'(0)^2 x^2-\pi'(0)(\de-\pi'(0)(\de-x^2)}{\sqrt{(\delta- \pi)^3}  }+\OO(1)\\
&=&
\frac{-\de \pi'(0)(1-\pi'(0))}{\sqrt{(\delta- \pi)^3}  }+\OO(1) >0
\end{array}
\end{equation}
where
$ \de-\pi=\delta-\pi'(0)(\delta-x^2)+\OO(\delta^{3/2})>0$.

As $\pi'(0)>1$, then:
\[
(\f)''(x)>0, \quad         \forall x    \in   [\sqrt\delta\sqrt{1-\bar{y}_0},\sqrt{\delta C}].
\]
\end{proof}

\begin{remark}
By Remark \ref{rem:xdelta+-}
 we know that
$$
x_\delta^+ = \rdelta+\OO(\delta)>\tilde x_\de
$$
Therefore $x_\de^+ \in \MM $ and by \eqref{eq:exteriorderivada} and \eqref{eq:exteriorderivada2}:
\begin{equation}
\begin{split}
(\f)'(x_\delta^+) &=-\pi'(0)+\OO(\rdelta)<0, \\
(\f)''(x_\delta^+) &=-\frac{\pi'(0)(1-\pi'(0))}{\sqrt{\de}}=
\OO(\frac{1}{\rdelta}),\quad\delta\rightarrow0.
\end{split}
\end{equation}
\end{remark}

In Proposition \ref{prop:exteriorderivada}, we have seen that $(\f)''>0$ in the interval $[\sqrt\delta\sqrt{1-\bar{y}_0},\sqrt{\delta C}]$, the value of $0<\bar y_0<\frac{1}{\pi'(0)}$ can be fixed from now on.

In  proposition \ref{prop:extensiopie}  we will give formulas for $\f$ also to the interval $[x_\delta,\sigma\rdelta]$ for any $\sigma$ such that $1>\sigma> \sqrt{1-\bar y_0}>\sqrt{1-\frac{1}{\pi'(0)}}$.

We state previously a technical lemma that will be needed during the proof of proposition \ref{prop:extensiopie}.

\begin{lemma}\label{lem:partfinalPe}
Let the visible fold determined by the system
\begin{equation}\label{eq:partfinalPe}
\begin{array}{rcl}
\dot x &=&
1\\
\dot y &=&
 g(x,y,\delta)
\end{array}
\end{equation}
with $g(0,0,\delta)=0,\forall |\delta|<\delta_0$, and $\frac{\partial g}{\partial x}(0,0,0)=a>0$.
Consider the map
\[
\begin{split}
 D^{-1} : \{ (x,y), \ x=0, -b_0 \le  y \le 0 \} & \to \{ (x,y), \ -c_0\le  x\le 0, \ y=0 \}\\
(0,y) & \mapsto ( D^{-1}(y), 0)
\end{split}
\]
induced by
the orbits of \eqref{eq:partfinalPe} (in negative time).
This map  is given by:
\begin{equation}\label{eq:D-1}
 D^{-1} (y)=-\sqrt{-\frac{2}{a} y}+\OO(\delta\sqrt{-y},y)
\end{equation}
Where the term $\OO$ is valid in the $\mathcal{C}^2$ topology.

In particular the map is convex near $(0,0)$ and the singularity at $(0,0)$ is
$\OO(\sqrt{-y})$.
\end{lemma}

\begin{proof}
First we find the map $D(x)$ defined by the cut in the $y$ negative semi-axis of the orbits of \ref{eq:partfinalPe}.
Its solution issuing from a point $(x,0)$ in the negative $x$-axis has the form
\[
y(t;x,0,\delta)=\int_0^t g(x+s,y(s;x,0,\delta),\delta)ds
\]
Then we have
\[
D(x)=y(-x;x,0,\delta)=\int_0^{-x} g(x+s,y(s;x,0,\delta),\delta)ds
\]
and
\begin{equation}
\begin{split}
D' (x)& =-g(0,y(-x;x,0,\delta),\delta)\\
&+\int_0^{-x}
[\frac{\partial{g}}{\partial{x}}(x+s,y(s;x,0,\delta),\delta)+
\frac{\partial{g}}{\partial{y}}(x+s,y(s;x,0,\delta),\delta)\frac{dy}{dx}(s;x,0,\delta)]ds
\end{split}
\end{equation}
where $ \frac{dy}{dx}(s;x,0,\delta)$ is the first variational of the solution.
In particular $D'(0)=-g(0,0,\delta)=0$.

For the second derivative we have
\begin{equation}
\begin{split}
D''(0)  &=-\frac{\partial{g}}{\partial{y}}(0,0,\delta)[-\dot{y}(0,0,\delta)+\frac{dy}{dx}(0;x,0,\delta)]- \frac{\partial{g}}{\partial{x}}(0,0,\delta)
 -\frac{\partial{g}}{\partial{y}}(0,0,\delta)\frac{dy}{dx}(0;x,0,\delta)\\
 & =- \frac{\partial{g}}{\partial{x}}(0,0,\delta)
\end{split}
\end{equation}
as $\dot{y}(0,0,\delta)=0$ and the first variational $ \frac{dy}{dx}(0;x,0,\delta)=0$

Hence we have, using Taylor formula
\begin{equation}
\begin{split}
D (x)&=D(0)+D'(0)x+\frac{D''(0)}{2}x^2+\mathcal{G}(x,\delta)x^3\\
&=-\frac{ \frac{\partial{g}}{\partial{x}}(0,0,\delta)}{2}x^2+\mathcal{G}(x,\delta)x^3=
  -\frac{a}{2}x^2+\mathcal{G}_1(\delta)\delta x^2+\mathcal{G}(x,\delta)x^3
\end{split}
\end{equation}
where $\mathcal{G}$, $\mathcal{G}_1$ denote smooth and uniformly bounded functions of their arguments and with bounded derivatives for $0\le \de \le \de_0$ and $-c_0\le x \le 0$.
Finally for $c_0$ and $\de _0$ small enough,  we invert the formula
$y=-\frac{a}{2}x^2+\mathcal{G}_1(\delta)\delta x^2+\mathcal{G}(x,\delta)x^3$
and we obtain the result \eqref{eq:D-1}.
\end{proof}

\begin{proposition}\label{prop:extensiopie}
Take $\sigma >\sqrt{1-\bar y_0}$, where $\bar y_0$ is the constant given in Proposition
\ref{prop:exteriorderivada}.
Then, the map $\f$ satisfies for $x\in [\tilde{x}_\de ,\sigma \sqrt{\de}]$:
\begin{equation}%\eqref{eq:extensiopie}
\begin{split}
\f(x)&=x_\de-\OO(\sqrt{x-\tilde{x}_\delta})\\
(\f)''(x)&=+\OO((x-\tilde{x}_\delta))^{-\frac32}
\end{split}
\end{equation}
and therefore
$(\f ) ''(x)\ge 0$, for  $x\in  [\tilde{x}_\de ,\sigma \sqrt{\de}]$:
\end{proposition}

\begin{proof}

Proposition \ref{prop:exteriorderivada} gives that $\f$ is convex  for $x\in [\sqrt{\delta}\sqrt{1-\bar{y}_0},\sqrt{\delta C}]$.
Now we will see that it is also convex in $[\tilde{x}_\delta,\sigma\rdelta ]$.

The definition of $\f$ in the  interval  $[\tilde{x}_\delta,\sigma\rdelta ]$ through the orbits of $\X_0$ is clear.
Nevertheless,  we cannot use  the approximation formulas seen in proposition \ref{prop:exteriorderivada}.
For these points, even if $\f$ exists, the formulas obtained through the identity
$\f= \Q^{-1}\circ \pi\circ \bar\Q^{-1}$ are not valid anymore.
The problem to apply these formulas is in $\Q^{-1}$, the last step of the definition of $\f$, but not in the map
$\pi\circ\bar{\Q}^{-1}$, which is well defined in a neighborhood of $\tilde{x}_\delta$ and formulas of the proposition \ref{prop:exterior} are valid: 
$$
\pi(\bar{\Q}^{-1}(x))
=\pi'(0)[\delta-x^2]+\OO(\delta^{3/2}), \ x \in \sqrt{\de}\sqrt{1-\bar y_0},\sqrt{\de C}],
\quad \pi(\bar{\Q}^{-1}(\tilde x_\de))=y_\de
$$
where $y_\de$ is given in \eqref{eq:ydelta}.

To study $\f$ near $\tilde{x}_\delta$, consider an interval around it,  $[\sigma_1\sqrt{\delta}, \sigma_2\sqrt{\delta}]$ with $ \sigma_1<\sqrt{1-\frac{1}{\pi'(0)}}<\sigma_2$.
Letting $\sigma_{1,2}$ be closer to $\sqrt{1-\frac{1}{\pi'(0)}}$ if needed, we can achieve that calling
$$
Y_\de:= \pi\circ\bar{\Q}^{-1}([\sigma_1\sqrt{\delta}, \sigma_2\sqrt{\delta}])=
[\pi'(0)\delta(1-\sigma_2^2)+\OO(\delta^{3/2}), \pi'(0)\delta(1-\sigma_1^2)+\OO(\delta^{3/2})]
$$
we have that
$$
\pi(\bar{\Q}^{-1}(\tilde{x}_\delta))=y_\delta =\delta+\OO(\delta^2)\in Y_\de
$$
if $\delta$ is small enough.

Next step is to "extend" the definition of the map $\Q^{-1}$ into $Y_\de$.
Observe that, modifying if needed, $\delta$ and $\sigma_{1,2}$ again, we can achieve that the flow is transversal to the sections  $x=0$ and $x=x_\delta$. See Figure \ref{fig:finalPe}.

Therefore, points $(0,y)$ with $y \in Y_\de$ have to be "classified" in different sets to extend $\Q^{-1}$:
\begin{enumerate}
\item
Points with $y\ge y_\de$. For these points the geometric definition of  $\Q^{-1}$ is not possible because the flow $\phi(t,0,y)$ does not cut $y=\de$, $x<0$. \\
We define  $\Q^{-1}(y)=x_\de$, for any $y\ge y_\de$.
\item
$\Q^{-1}(y_\de)=x_\de$
\item
%\textcolor{blue}{
Points with $\de<y<y_\de$:
\begin{itemize}
\item
If the tangency point $x_\de >0$, we can define $\Q^{-1}(y)$ as the first  cut of the flow
 $\phi(t,0,y)$  (in backwards or forward time) with $y=\de$, $x<x_\de$.
\item
If the tangency point $x_\de <0$, we can define $\Q^{-1}(y)$ as the second  cut of the flow
 $\phi(t,0,y)$  (in backwards or forward time) with $y=\de$, $x<x_\de <0$.
\end{itemize}
%}
\item
Points with $0<y<\de$, where we can define $\Q^{-1}(y)$ as the first   cut of the flow  $\phi(t,0,y)$  (in backwards or forward time) with $y=\de$, $x<0$.
\end{enumerate}
To obtain an asymptotics for $\Q^{-1}$  in $Y_\de$, in fact in  $Y_\de\cap \{y\le y_\de\}$, we observe that:
\[
\begin{split}
\Q^{-1} &= D^{-1} \circ C\\
C &: \{ (x,y), \ x=0, \ y \in Y_\de, \ y\le y_\de \}  \to  \{ (x,y), \ x=x_\de, \ y \in \tilde Y_\de , \ y\le \de \} , \quad C(y_\de)=\de \\
D^{-1} &: \{ (x,y), \ x=x_\de, \ y \in \tilde Y_\de, \quad y \le \de \}  \to \{ (x,y), \ x\le x_\de, \ y=\de \}
\end{split}
\]

The first map $C$ is  a diffeomorphism do to the transversality of the flow to both lines $x=0$ and $x=x_\de$.
Moreover, we know that $C(y_\de)=\de $ and
we have that

\begin{equation}\label{eq:C}
C(y)=\de+(1+\OO(\de))( y-y_\de)+\OO((y-y_\de)^2)
\end{equation}
To study the map $ D^{-1}$, we perform the change
$$
\bar{x}=x-x_\delta, \quad \bar y=y-\delta,
$$
to system \eqref{def:Xg} and we can apply the  lemma \ref{lem:partfinalPe} to the resulting system, which has a fold point at $(0,0)$.

\begin{figure}
\begin{center}
\includegraphics[width=9cm,height=5cm]{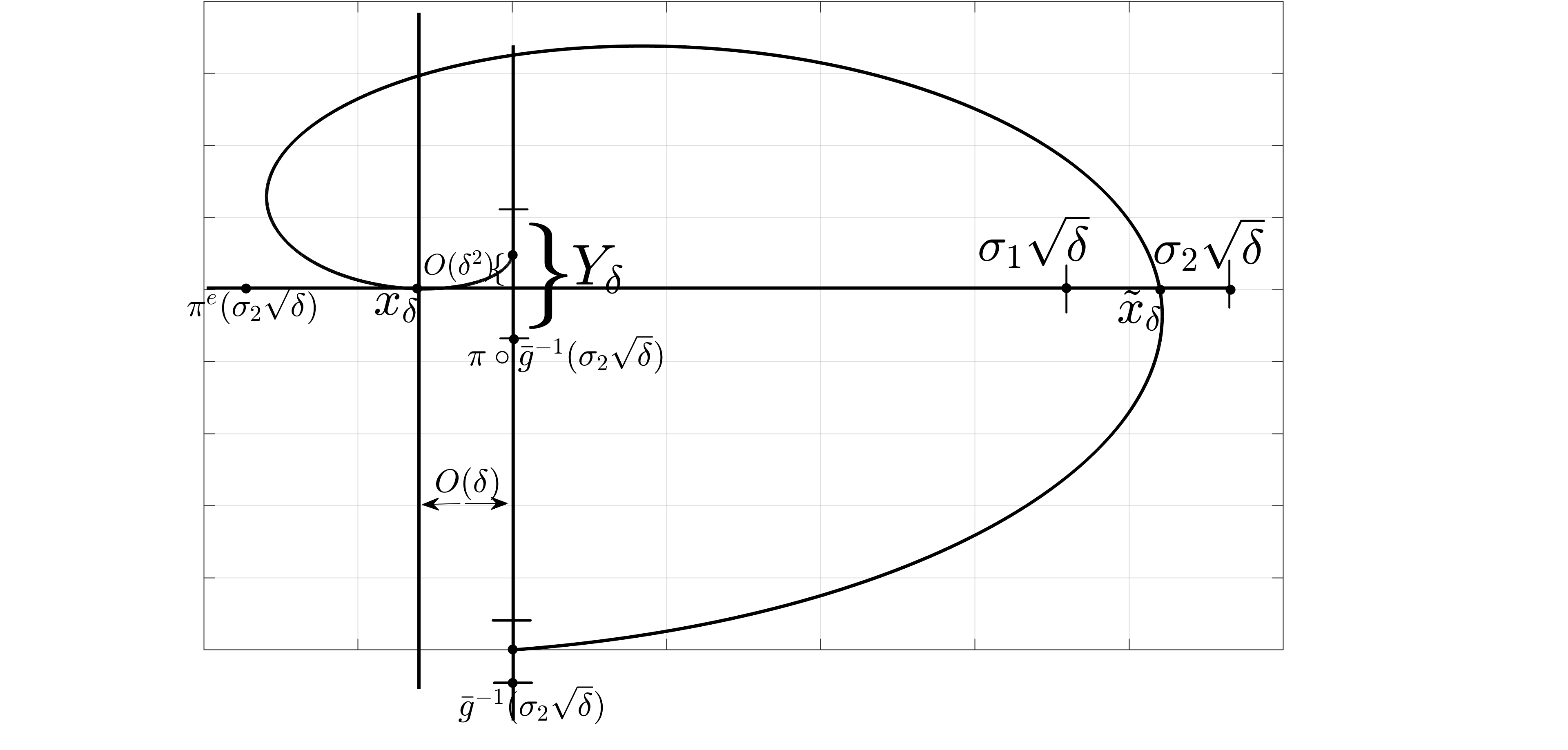}
\caption{The map $\f$ around $\tilde{x}_\delta$.}\label{fig:finalPe}
\end{center}

\end{figure}

This lemma  provides formulas for $D^{-1}$:
\[
D^{-1}(y)=x_\de -\sqrt{\frac{2}{a}(\de -y)}+\OO(\de \sqrt{\de -y}, \de-y)
\]
This formula combined with \eqref{eq:C} allows us to obtain asymptotic formulas for $\Q^{-1}$ for $y\in Y_\de$, $y\le y_\de$:
\begin{equation}\label{eq:g-1propxtilde}
\begin{split}
\Q^{-1}(y) &= D^{-1}(C(y)-\de)=x_\de -\sqrt{\frac{2}{a}(\de-C(y))}+\OO(\de \sqrt{\de-C(y)}, \de-C(y))\\
&=
x_\de-\sqrt{\frac 2a (y_\de-y)}
+\OO((\de \sqrt{y_\de-y}, (y-y_\de))
\end{split}
\end{equation}
Recalling that
$
\f(x)=\Q^{-1}(\pi(\bar\Q^{-1}(x)))
$
we have:
\[
\f(x)=x_\de-\sqrt{\frac 2a (y_\de-\pi(\bar\Q^{-1}(x)))}
+\OO(\de \sqrt{y_\de-\pi(\bar\Q^{-1}(x)}, \pi(\bar\Q^{-1}(x))-y_\de)
\]
Now, using that
\[
\pi(\bar\Q^{-1}(\tilde x))= \pi(\bar\Q^{-1}(\tilde x_\de)) +\OO(x-\tilde x_\de)=
y_\de + \OO(x-\tilde x_\de)
\]
we obtain:
\[
\f(x)=x_\de-\OO(\sqrt{x-\tilde{x}_\delta})
\]
and therefore
\[
(\f)''(x)=\OO((x-\tilde{x}_\delta))^{-\frac32}
\]
and consequently is convex.

\end{proof}

The result of this proposition,  combined with proposition \ref{prop:exteriorderivada} assures that the full extension of $\f$ is convex on the interval $[x_\delta,\sqrt{\de C}]$ and near $\tilde{x}_\delta$ the singularity has the form $\OO(\sqrt{x-\tilde{x}_\delta)}$.
This concludes the proof of Theorem \ref{thm:extensiopie}.

\subsection{The inner  map $\mathcal{Q}_{\mu,\eps} $: proof of Theorem \ref{prop:asymptoticsQ} }

In this section we prove Theorem \ref{prop:asymptoticsQ}.
We recall that the map  $\QQQ_{\mu,\eps}$ is defined by the orbits of the system \eqref{eq:fastsystem}  between $x<x_\mu=\OO(\eps^\frac43),v=1$ and $x>x_\mu, v=1$.
Even  the map $\QQQ_{\mu,\eps}$  depends on $\mu$, during this section we will simplify the notation and call it $\QQQ_{\eps}$.

To study the map $\QQQ_{\eps}$, we perform  the blow-up variables $ x=\eps^{\frac{2}{3}}\eta,v=1+\eps^\frac{1}{3}u $, and system \eqref{eq:fastsystem} is transformed into
\begin{equation}
\label{eq:sistemablowup}
\begin{array}{rcl}
\dot \eta &=&1+O(\eps^\frac{2}{3})\\
\dot u &=&2\eta-\frac{\varphi''(1)}{4}u^2+O(\eps^\frac{1}{3}).
\end{array}
\end{equation}
In these new variables, an  interval of the form  $x\in [-M\eps^\frac23, -\overline M\eps^\frac23]$ transforms to
$\eta \in [-M, -\overline M]$
and the relation between the map $\QQQ_{\eps}$ associated to system \eqref{eq:fastsystem} and the map $\tilde{\mathcal{Q}_\eps}$ associated to system \eqref{eq:sistemablowup} in these new variables will be
\begin{equation}\label{Qescalada}
\mathcal{Q}_{\eps}( x )=\eps^{\frac{2}{3}}\tilde{\mathcal{Q}_\eps}(\frac{x}{\eps^{\frac{2}{3}}}).
\end{equation}
and $\tilde \QQQ_\eps$ is defined in the section $u=0$.
Then we proceed as Proposition \ref{prop:exterior} and will approximate the map $\tilde \QQQ_\eps$ by the corresponding
map $\tilde{\mathcal{Q}_0}$ related to the system \eqref{Ricatti}, that we recall here:
\begin{equation}
%\left.
\begin{array}{rcl}
\dot \eta &=&1\\
\dot u &=&2\eta-\frac{\varphi''(1)}{4}u^2.
\end{array}
\end{equation}
Observe that this system has a fold point at $(\eta,u)=(0,0)$, therefore, $\tilde\QQQ_0(0)=\tilde\QQQ_0'(0)=0$.
Nevertheless, for points $\eta\in [-M,-\overline  M]$, we have that $\tilde \QQQ_0'(\eta)\ne 0$ therefore, like in Proposition \ref{prop:exterior}
\begin{equation}\label{derivadesQ}
\quad \tilde{\mathcal{Q}_\eps}(\eta)=\tilde{\mathcal{Q}_0}(\eta)+\OO(\eps^{\frac{1}{3}}), \ \eta\in [-M,-\overline M]
\quad\tilde{\mathcal{Q}_\eps}'(\eta)=
\tilde{\mathcal{Q}_0}'(\eta)+\OO(\eps^{\frac{1}{3}}),\quad\tilde{\mathcal{Q}_\eps}''(\eta)=\tilde{\mathcal{Q}_0}''(\eta)+\OO(\eps^{\frac{1}{3}})
\end{equation}
and therefore
\begin{equation}\label{derivadesQQ}
\quad \mathcal{Q}_\eps(x)=\eps^{\frac{2}{3}}\tilde{\mathcal{Q}_0}(\frac{x}{\eps^{\frac{2}{3}}})+\OO(\eps),\quad \mathcal{Q}'_\eps(x)=\tilde{\mathcal{Q}'_0}(\frac{x}{\eps^{\frac{2}{3}}})+\OO(\eps),
\quad \mathcal{Q}''_\eps(x)=\frac{1}{\eps^{\frac{2}{3}}}\tilde{\mathcal{Q}''_0}(\frac{x}{\eps^{\frac{2}{3}}})+\OO(\eps).
\end{equation}
Next step is to  prove that  $\tilde{\mathcal{Q}}''_0(\eta)<0$, for
$ \eta\in[-M,-\overline{M}]$.
To this end, the  scaling
\[
x=-(\frac{\varphi''(1)}{2})^\frac13 \eta, \quad
y=(\frac{(\varphi''(1))^2}{32})^\frac13 u
\]
transforms system \eqref{Ricatti} into the Ricatti equation:
\begin{equation}\label{RRicatti}
\begin{array}{rcl}
\dot x &=1\\
\dot y &=x+y^2.
\end{array}
\end{equation}\\

\begin{figure}
\begin{center}
\includegraphics[width=10cm,height=4cm]{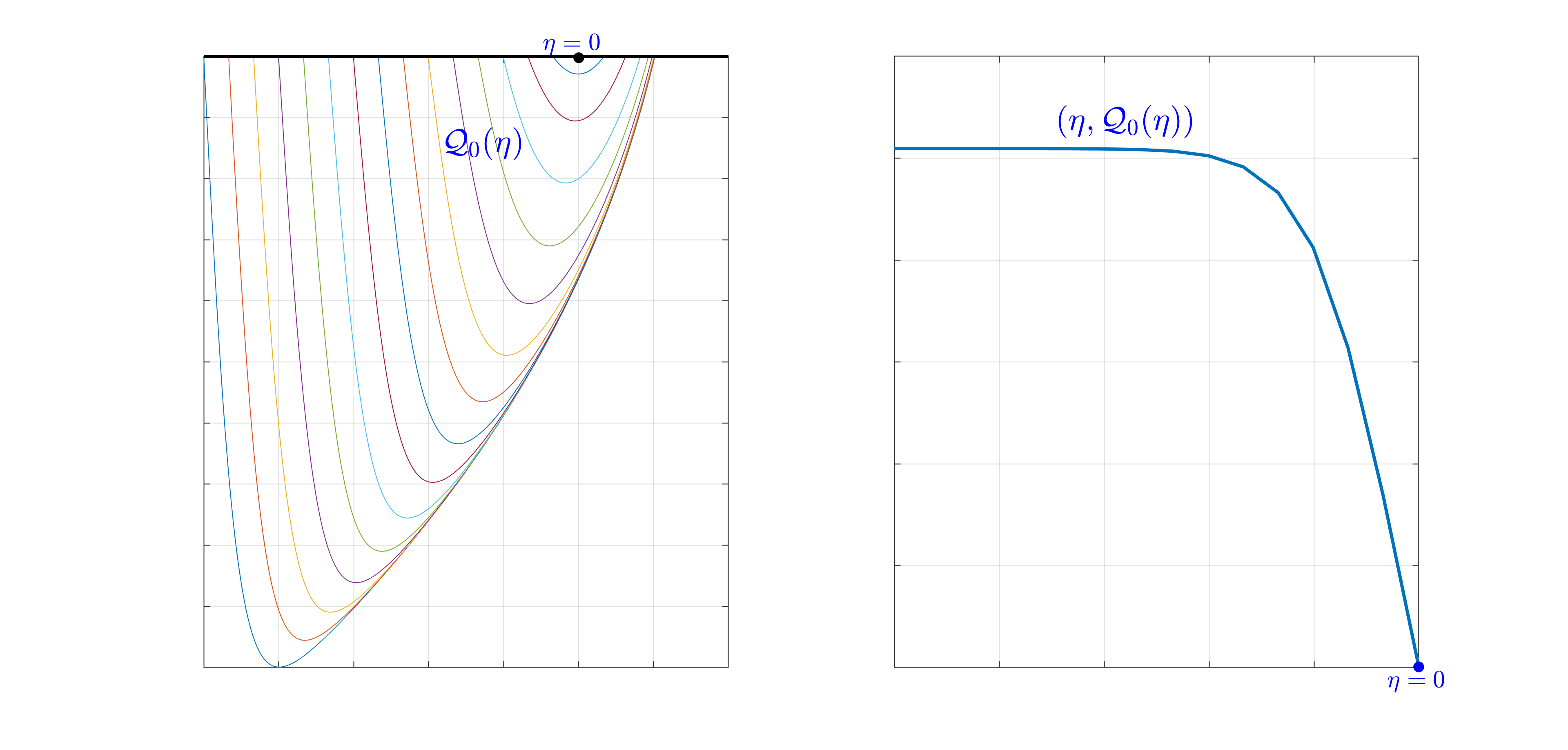}
\caption{First return map $\mathcal{Q}_0 $.}\label{fig:firstreturn}
\end{center}

\end{figure}

Therefore, we will study the map $\QQQ_0$ associated to this system (see Figure \ref{fig:firstreturn}).
In particular, as $\varphi''(1)<0$ the sign of $\QQQ_0''(x)$ will be the same as $\tilde \QQQ''_0(\eta)$.

If we call  $y(t,x_0)$ be the solution of \eqref{RRicatti} which begins at $(x_0,0), x_0<0$, and $t(x_0)>0$ the time of the first cut to ${x>0}$, that is, $y(t(x_0),x_0)=0$, the first return map  is given by $\QQQ_0(x_0)=x_0+t(x_0)$.
Then, on the one hand we have, using that $\QQQ_0$ is decreasing:
\begin{equation}\label{eq:derivadaprimera}
{\mathcal{Q}}'_0(x_0)=1+t'(x_0)<0
\end{equation}
and, one the other hand:
${\mathcal{Q}}''_0(x_0)=t''(x_0)$.

To prove that $t''(x_0)<0$  we proceed in several steps.
\begin{itemize}
\item
First, in Lemma \ref{lem:retornlocal}, we will use Taylor expansions to compute  $t(x_0) $ and its derivatives for points near $x_0=0$. This will allow us to check that $t''(x_0)<0$ for small values of $x_0$.
\item
To see that $t''$ is negative in a finite interval  of the form $[-M,0)$ we need to use the second order variational equations of system \eqref{RRicatti} and relate the sign of $t''$ with several quantities obtained through the study of the solutions of these variational equations. This is done in lemmas \ref{lem:variacionals},\ref{vigual1}, \ref{lem:intu},\ref{lem:derivadav} and, finally, in Proposition \ref{prop:signef}.
\end{itemize}
\begin{lemma}\label{lem:retornlocal}
Near $x_0=0$ the function $t(x_0)$ has the expansion
\begin{equation}\label{prop:retornlocal}
\begin{array}{lcr}
t(x_0)=-2x_0-\frac{4}{15}x_0^4+O(x_0^5), \\
t''(x_0)=-\frac{16}{5}x_0^2  +\OO(x_0^3)<0.
\end{array}
\ x_0\simeq 0
\end{equation}
\end{lemma}
\begin{proof}
To see \eqref{prop:retornlocal} we take
 $y(t,x_0)$ such that $y(0,x_0)=0$
 and compute:
\begin{equation}
\begin{array}{lcr}
y'=x+y^2\quad\Rightarrow y'(0)=x_0\\
                   y''=1+2yy'\quad \Rightarrow y''(0)=1\\
                   y'''=2(y')^2+yy''\quad\Rightarrow y'''(0)=2x_0^2\\
                   y^{(iv)}=6y'y''+2yy'''\quad\Rightarrow y^{(iv)}(0)=6x_0\\
                    y^{(v)}=6(y'')^2+8y'y'''+2yy^{(iv)}\quad\Rightarrow y^{(v)}(0)=6+16x_0^3\quad...
 \end{array}
\end{equation}
On the other hand, we can expand $y(t,x_0)$ near $t=0$ and we obtain
\begin{equation}\label{tximplicita}
\begin{array}{lcr}
 y(t,x_0)
 &=& 0+x_0t+\frac{1}{2}t^2+\frac{1}{3}x_0^2t^3+\frac{1}{4}x_0t^4+\frac{6+16x_0^3}{120}t^5+ O(t^6)\\
        & = &  t(x_0+\frac{1}{2}t+\frac{1}{3}x_0^2t^2+\frac{1}{4}x_0t^3+\frac{6+16x_0^3}{120}t^4+ \OO(t^5));
                     \end{array}
\end{equation}
 By definition of $t(x_0)>0$, and we have $y(t(x_0);x_0,0)=0$, therefore:
 $t(x_0)$ is the implicit solution of:
\begin{equation}\label{eq:implicita}
x_0+\frac{1}{2}t(x_0)+\frac{1}{3}x_0^2t(x_0)^2+\frac{1}{4}x_0t(x_0)^3+\frac{6+16x_0^3}{120}t(x_0)^4+ \OO(t(x_0)^5)=0.
\end{equation}
But we seek the behavior of $t(x_0)$ near $x_0=0$.
Clearly $t(0)=0$ and if we expand it in powers of $x_0$ we will have
\[
t(x_0)=
%t_0+t_1x_0+t_2x_0^2+t_3x_0^3+t_4x_0^4+...=
t_1x_0+t_2x_0^2+t_3x_0^3+t_4x_0^4+...
%\end{array}\end{equation}
\]
has to solve \eqref{eq:implicita}.
That is:
\[
%\begin{equation}
\begin{array}{lcr}
  0=x_0+\frac{1}{2}(t_1x_0+t_2x_0^2+t_3x_0^3+t_4x_0^4+)+
\frac{1}{3}x_0^2(t_1x_0+t_2x_0^2+t_3x_0^3+t_4x_0^4)^2+\\
 \frac{1}{4}x_0(t_1x_0+t_2x_0^2+t_3x_0^3+t_4x_0^4)^3+
 \frac{6+16x_0^3}{120}(t_1x_0+t_2x_0^2+t_3x_0^3+t_4x_0^4)^4+\OO(x_0^5)\\
\end{array}
%\end{equation}
\]
and equating the coefficients of the successive powers of $x_0$ we arrive at the result.
%, and $\tilde{\mathcal{Q}}_0(\eta)''<0$ near $0$.
\end{proof}
As a consequence of the previous lemma, for $x_0$ small enough:
\[
\begin{split}
\QQQ_0(x_0)&=x_0+t(x_0) =-x_0-\frac{4}{15}x_0^4+O(x_0^5), %\ \mbox{for} \ x_0 \approx 0
\\
\QQQ_0''(x_0)&=t''(x_0)=-\frac{16}{5}x_0^2+O(x_0^3)<0. %\mbox{for} \ x_0 \approx 0 .
\end{split}
\]
Observe that going back to the original variables we obtain the last item of the theorem \eqref{derivadesQQvell}.
\begin{equation}
\quad \mathcal{Q}_\eps(x)=\eps^{\frac{2}{3}}\tilde{\mathcal{Q}_0}(\frac{x}{\eps^{\frac{2}{3}}})+\OO(\eps)=-x(1+\OO(\frac{x}{\eps^{\frac23}}))+\OO(\eps)
\end{equation}

Next step is to extend the previous result about the sign of $t(x_0)''$ for any $x_0\in [-M,0]$.
\\
To this end,  we will see that  the sign of the second derivative of the first return map, $\QQQ_0''(x_0)=t''(x_0)$, will be determined by the sign of the function $u^2+x^2v-2xu$, the functions $u$, $v$ are solutions of  the second order variational equations associated to system  \eqref{RRicatti}:
\begin{equation}\label{variacionalsRRicatti}
%\left.
\begin{array}{rcl}
\dot x &=1\\
\dot y &=x+y^2\\
\dot u &=1+2yu\\
\dot v &=2u^2+2yv,
\end{array}
%\right\} \quad r=\sqrt {x^2+(y-\mu-1)^2}
\end{equation}\\
with initial condition $ (x_0,0,0,0) $ and evaluated at $t=t(x_0)$.
Actually, we have
\begin{equation}\label{theo:variacionals}
\begin{array}{lcr}
x(t,x_0)=t+x_0; \quad u(t,x_0)\equiv\frac{\partial y}{\partial x_0}(t,x_0);\quad v(t,x_0)\equiv\frac{\partial^2 y}{\partial x_0^2}(t,x_0)\\
%t'(x_0)=-\frac{\frac{\partial y}{\partial x}(t(x_0);x_0,0)}{t(x_0)+x_0};\quad t''(x_0)=-\frac{(t'(x_0))^2+2t'(x_0)+\frac{\partial^2 y}{\partial x^2}(t(x_0);x_0,0)}{t(x_0)+x_0}\equiv \mathcal{Q}''(x_0)
 \end{array}
 \end{equation}
Next lemma gives $t'(x_0)$ and $t''(x_0)$ in terms of these functions.
and we will see that:
\begin{lemma}\label{lem:variacionals}
\begin{equation}\label{theo:variacionals}
\begin{array}{rcl}
t'(x_0)&=&-\frac{\frac{\partial y}{\partial x_0}(t(x_0),x_0)}{t(x_0)+x_0}=
-\frac{u(t(x_0),x_0)}{x(t(x_0),x_0)}
;\\
 t''(x_0)&=&-\frac{(t'(x_0))^2+2t'(x_0)+\frac{\partial^2 y}{\partial x_0^2}(t(x_0),x_0)}{t(x_0)+x_0}
 =-\frac{(t'(x_0))^2+2t'(x_0)+v(t(x_0),x_0)}{x(t(x_0),x_0)}\equiv \QQQ_0''(x_0)
  \end{array}
  \end{equation}
\end{lemma}
\begin{proof}
To see  formulas \eqref{theo:variacionals}, we apply the Implicit Function theorem to the equation
\begin{equation}\label{eq:y0}
\begin{array}{rcl}
y(t(x_0),x_0)=0,
\end{array}
\end{equation}
%
%from $ y(t(x_0);x_0,0)=0$,
where $(x_0,0)$ with $x_0<0$ is the initial point.
Then, differentiating respect to $x_0$ equation \eqref{eq:y0} we get, denoting $'=\frac{\partial}{\partial x_0}$ and $\dot{}=\frac{\partial}{\partial t}$:
\begin{equation*}
\begin{array}{rcl}
0 &=&
%\frac{\partial}{\partial x}(y(t(x);x,0))=
\dot{y}(t(x_0),x_0)t'(x_0)+\frac{\partial y}{\partial x_0}(t(x_0),x_0)\\
&=&\left(t(x_0)+x_0+y^2(t(x_0),x_0)\right)t'(x_0)+
\frac{\partial y}{\partial x_0}(t(x_0),x_0)=(t(x_0)+x_0)t'(x_0)
+\frac{\partial y}{\partial x_0}(t(x_0),x_0)
\end{array}
\end{equation*}
therefore
\begin{equation}\label{primeraderivadaretorn}
%\Longrightarrow
t'(x_0)=-\frac{\frac{\partial y}{\partial x_0}(t(x_0),x_0)}{t(x_0)+x_0}.
\end{equation}
Differentiating another time, and using that $y(t(x_0),x_0)=0$, we have
\begin{equation}
\begin{array}{rcl}
0&=&
%\frac{\partial}{\partial x}((t(x)+x)t'(x)+\frac{\partial y}{\partial x}(t(x);x,0))=
(t'(x_0)+1)t'(x_0)+(t(x_0)+x_0)t''(x_0)
+(\dot{ \frac{\partial y}{\partial x_0}})(t(x_0),x_0)t'(x_0)
+\frac{\partial^2 y}{\partial x_0^2}(t(x_0),x_0)\\
&=&
(t'(x_0)+1)t'(x_0)+(t(x_0)+x_0)t''(x_0)+\dot u (t(x_0),x_0) t'(x_0)
+\frac{\partial^2y}{\partial x_0^2}(t(x_0),x_0)\\
&=&
(t'(x_0)+1)t'(x_0)+(t(x_0)+x_0)t''(x_0)+u(t(x_0),x_0))t'(x_0)
+\frac{\partial^2y}{\partial x_0^2}(t(x_0),x_0)
\end{array}
\end{equation}
and therefore:
%, using System \eqref{variacionalsRRicatti} and $y(t(x_0);x,0)=0$:
%
\begin{equation}\label{segonaderivadaretorn}
%\Longrightarrow
t''(x_0)=-\frac{(t'(x_0))^2+2t'(x_0)+\frac{\partial^2y}{\partial x^2}(t(x_0),x_0)}{t(x_0)+x_0}
%t''(x)=-\frac{(t'(x))^2+2t'(x)+\frac{\partial^2y}{\partial x^2}(t(x);x,0)}{t(x)+x}
\end{equation}
%In particular $t''(x_0)=-\frac{(t'(x_0))^2+2t'(x_0)+\frac{\partial^2y}{\partial x^2}(t(x_0);x_0,0)}{t(x_0)+x_0}$.\\
%
%
\end{proof}
In view of Lemmma \ref{lem:variacionals},  as $t(x_0)+x_0=\QQQ_0(x_0)>0$, $\QQQ_0''(x_0)<0$ if
\[
(t'(x_0))^2+2t'(x_0)+\frac{\partial^2 y}{\partial x_0^2}(t(x_0),x_0)>0.
\]
And by (\ref{primeraderivadaretorn}), this condition will be equivalent to
\begin{equation}
\begin{split}
&(\frac{\partial{y}}{\partial{x_0}}(t(x_0),x_0))^2-2(t(x_0)+x_0)
\frac{\partial{y}}{\partial{x_0}}(t(x_0),x_0)+
(t(x_0)+x_0)^2\frac{\partial^2{y}}{\partial x_0^2}(t(x_0),x_0)=\\
%(u(t(x_0),x_0))^2-2(t(x_0)+x_0)u(t(x_0),x_0)+(t(x_0)+x_0)^2v(t(x_0),x_0)=\\
&(u(t(x_0),x_0))^2-2x(t(x_0),x_0)u(t(x_0),x_0)+(x(t(x_0),x_0))^2v(t(x_0),x_0)>0,
\end{split}
\end{equation}
where $(x(t,x_0),y(t,x_0),u(t,x_0),v(t,x_0))$ are the solutions of system \eqref{variacionalsRRicatti} with initial condition $(x_0,0,0,0)$, and $t(x_0)>0$ is such that
$y(t(x_0),x_0)=0$.

We need to prove that
\[
f(x_0)= (u(t(x_0),x_0))^2-2x(t(x_0),x_0)u(t(x_0),x_0)+(x(t(x_0),x_0))^2v(t(x_0),x_0)>0,\  \forall x_0<0
\]
To this end, we need some technical lemmas:
\begin{lemma}\label{vigual1}
Consider the solutions of system \eqref{variacionalsRRicatti} and $t(x_0)$ the time such that $y(t(x_0),x_0)=0$. Then one has that:
\begin{equation}
u(t(x_0),x_0)-x(t(x_0),x_0) >0, \quad \forall x_0<0
\end{equation}
\end{lemma}

\begin{proof}
Calling $w(t)=u(t,x_0)-\dot y(t,x_0)=u(t,x_0)-x(t,x_0)-y^2(t,x_0)$ we have that $w(0)= u(0,x_0)-\dot y(0,x_0)=-x_0>0$
 and differentiating
 \[
 \begin{split}
 \dot w&= \dot u(t,x_0)-\ddot y(t,x_0)= 1+2y(t,x_0)u(t,x_0)-\frac{\partial }{\partial t}(x(t,x_0)+y^2(t,x_0)) \\
& =1+2y(t,x_0)u(t,x_0) -1-2y(t,x_0)\dot y(t,x_0)=2y(t,x_0)w
 \end{split}
 \]
 therefore
 \begin{eqnarray}
 w(t) &=&-x_0 e^{\int_0^t 2y(s,x_0)ds}\ \Rightarrow \nonumber \\
u(t, x_0)&=& \dot y(t,x_0)-x_0 e^{\int_0^t 2y(s,x_0)ds}
 =x(t,x_0)+ y^2(t,x_0)-x_0 e^{\int_0^t 2y(s,x_0)ds} \label{eq:wu}
 \end{eqnarray}
 evaluating at $t=t(x_0)$ we obtain, using that $x_0<0$:
 \[
 u(t(x_0), x_0)=x_0+ t(x_0)-x_0 e^{\int_0^{t(x_0)} 2y(s,x_0)ds}
 \Rightarrow \ u(t(x_0), x_0)-x(t(x_0),x_0)=-x_0 e^{\int_0^{t(x_0)} 2y(s,x_0)ds} >0
 \]
\end{proof}

\begin{lemma}\label{lem:intu}
Assume that for some $x_0 <0$ the function $v(t(x_0),x_0) <1$. Then we have:
\[
 1+2 x_0 \int _0^t u(s,x_0 )ds >0, \quad \forall 0\le t\le t(x_0)
\]
\end{lemma}
\begin{proof}
Let's compute $v(t,x_0)$ using the expression for $u(t,x_0)$ obtained in \eqref{eq:wu}:
\begin{eqnarray}
v(t,x_0) &=& \frac{\partial u}{\partial x_0}(t,x_0)=
\frac{\partial \dot y}{\partial x_0}(t,x_0) -  e^{\int_0^t 2y(s,x_0)ds}[ 1+2 x_0 \int _0^t u(s,x_0 )ds]\nonumber
\\
&=& \dot u(t,x_0) -  e^{\int_0^t 2y(s,x_0)ds}[ 1+2 x_0 \int _0^t u(s,x_0 )ds]\label{eq:bonav}\\
&=&
1+2y(t,x_0)u(t,x_0) -  e^{\int_0^t 2y(s,x_0)ds}[ 1+2 x_0 \int _0^t u(s,x_0 )ds]\nonumber
\end{eqnarray}

Evaluating at $t=t(x_0)$ we have:
\begin{equation}\label{eq:goodv}
v(t(x_0),x_0) =
1 -  e^{\int_0^{t(x_0)} 2y(s,x_0)ds}[ 1+2 x_0 \int _0^{t(x_0}u(s,x_0 )ds]
\end{equation}
and therefore
\[
v(t(x_0),x_0) -1 = -  e^{\int_0^{t(x_0)} 2y(s,x_0)ds}[ 1+2 x_0 \int _0^{t(x_0)} u(s,x_0 )ds]
\]

The last equality gives:
\[
v(t(x_0),x_0)< 1 \iff   1+2 x_0 \int _0^{t(x_0)} u(s,x_0 )ds >0
\]
Moreover,  as $x_0<0$, for any $0\le t\le t(x_0)$ we have:
\[
 1+2 x_0 \int _0^t u(s,x_0 )ds >  1+2 x_0 \int _0^{t(x_0)} u(s,x_0 )ds >0
\]
\end{proof}
\begin{lemma}\label{lem:derivadav}
Assume that for some $x_0 <0$ the function $v(t(x_0),x_0) <1$. Then we have:
\[
\frac{d v}{d x_0} (t(x_0),x_0)<0
\]
\end{lemma}
\begin{proof}
We differenciate the expression \eqref{eq:goodv}:
\[
 \begin{split}
&\frac{d v}{d x_0} (t(x_0),x_0)=  -  e^{\int_0^{t(x_0)} 2y(s,x_0)ds}  \left( 2y(t(x_0),x_0) t'(x_0)+\int_0^{t(x_0)} 2u(s,x_0)ds \right)\left[ 1+2 x_0 \int _0^{t(x_0}u(s,x_0 )ds\right]\\
&- e^{\int_0^{t(x_0)}2y(s,x_0)ds}  \left( 2 \int _0^{t(x_0)}u(s,x_0 )ds +2 x_0 u(t(x_0),x_0) t'(x_0)+
 2 x_0 \int _0^{t(x_0)}v(s,x_0 )ds \right)\\
&=  -  e^{\int_0^{t(x_0)} 2y(s,x_0)ds}
\left\{
\int_0^{t(x_0)} 2u(s,x_0)ds \left[ 1+2 x_0 \int _0^{t(x_0}u(s,x_0 )ds\right]+
2 \int _0^{t(x_0)}u(s,x_0 )ds \right. \\
& \left.+2 x_0 u(t(x_0),x_0) t'(x_0)+
 2 x_0 \int _0^{t(x_0)}v(s,x_0 )ds
 \right\}
% &=  -  e^{\int_0^{t(x_0)} 2y(s,x_0)ds}
%\left\{\int_0^{t(x_0)} 2u(s,x_0)ds \left[ 1+2 x_0 \int _0^{t(x_0}u(s,x_0 )ds\right]+ 2 \int _0^{t(x_0)}u(s,x_0 )ds \right. \\
%& \left.-2 x_0 x(t(x_0),x_0) (t'(x_0))^2+ 2 x_0 \int _0^{t(x_0)}v(s,x_0 )ds  \right\}
\end{split}
\]
Observe that the terms involving the integral of $u$ in  this expression are positive because $u$ is positive.
By lemma \ref{lem:intu} we know that also the term
$1+2 x_0 \int _0^{t(x_0}u(s,x_0 )ds$
is positive.
%The fourth term is positive because $x_0<0$ and $x(t(x_0),x_0)>0$.
Therefore, again using that $x_0<0$  we just need to check that:
\begin{equation}\label{eq:finalsigne}
u(t(x_0),x_0) t'(x_0)+\int _0^{t(x_0)}v(s,x_0 )ds <0
\end{equation}
to finish the proof.

To see \eqref{eq:finalsigne} we use  the expression \eqref{eq:bonav}:
\[
\begin{split}
&u(t(x_0),x_0) t'(x_0)+\int _0^{t(x_0)}v(s,x_0 )ds =\\
&u(t(x_0),x_0) t'(x_0)+
\int _0^{t(x_0)}\left(
\dot u(s,x_0) -  e^{\int_0^s 2y(r,x_0)dr}\left[ 1+2 x_0 \int _0^s u(r,x_0 )dr\right]  \right) ds \\
&=u(t(x_0),x_0) t'(x_0)+ u(t(x_0),x_0) -  \int _0^{t(x_0)}e^{\int_0^s 2y(r,x_0)dr}\left[ 1+2 x_0 \int _0^s u(r,x_0 )dr\right]\  ds \\
\\
&=u(t(x_0),x_0)( t'(x_0)+ 1) -  \int _0^{t(x_0)}e^{\int_0^s 2y(r,x_0)dr}\left[ 1+2 x_0 \int _0^s u(r,x_0 )dr\right]   ds <0
\end{split}
\]
Where the last inequality is a consequence of lemma \ref{lem:intu}, equation \eqref{eq:derivadaprimera}
%fact that $\QQQ_0 '(x_0)=t'(x_0)+1 <0$
and the fact that $u(t(x_0),x_0)\ge 0$.
\end{proof}

Now we are ready to prove that
\begin{proposition}\label{prop:signef}
We have:
\[
f(x_0)= (u(t(x_0),x_0))^2-2x(t(x_0),x_0)u(t(x_0),x_0)+(x(t(x_0),x_0))^2v(t(x_0),x_0)>0,\  \forall x_0<0
\]
\end{proposition}
\begin{proof}
By formula \eqref{prop:retornlocal} we know that for  small enough $x_0 \le 0$ one has that $f(x_0)<0$.
Suppose that somewhere in $\{x<0\}$ the function $f(x_0)$
%expression $u^2+x^2v-2xu$
were positive.
Let be $x_1<0$, the first time  where $f(x_1)=0$.
%this expression is zero.
We would have
\[
%\begin{equation}\begin{split}
f(x_1)= (u(t(x_1),x_1))^2-2x(t(x_1),x_1)u(t(x_1),x_1)+(x(t(x_1),x_1))^2v(t(x_1),x_1)=0
\]
%\end{split}\end{equation}
%Solving this equation in $u$ and in $x$ we obtain:
%\[u(t(x_1))=x(t(x_1))(1\pm\sqrt{1-v(t(x_1)})\ \
%x(t(x_1))=\frac{u(t(x_1))(1\pm\sqrt{1-v(t(x_1)})}{v(t(x_1))}\]
%which gives $ v(t(x_1))=(1\pm\sqrt{1-v(t(x_1))})^2$.
%But this equation has the unique solution $ v(t(x_1))=1$.
Observe that we can write $f(x_1)=0$ as:
\[
%begin{equation}\begin{split}
f(x_1)=(u(t(x_1),x_1))-x(t(x_1),x_1))^2+(x(t(x_1),x_1))^2(v(t(x_1),x_1)-1)=0
\]
which can have a solution if:
\begin{enumerate}
\item
 $v(t(x_1),x_1)=1$ and $u(t(x_1),x_1)=x(t(x_1),x_1)$  or
 \item
  $v(t(x_1),x_1)<1$.
\end{enumerate}
Lemma \ref{vigual1} proves that the first possibility can not hold.
Therefore,  if $f(x_1)=0$ then $v(t(x_1),x_1)<1$.
Let us now compute the derivative of $f$:
\begin{equation}
\begin{split}
f'(x_0)&=\frac{d}{dx_0}\left(u(t(x_0),x_0))^2-2x((t(x_0),x_0))u((t(x_0),x_0))+(x(t(x_0),x_0))^2v(t(x_0),x_0)\right)(x_0)\\
&=2u(t(x_0),x_0)\left[\dot u(t(x_0),x_0))t(x_0)'+v(t(x_0),x_0)\right]-2\left[1+t'(x_0)\right]u(t(x_0),x_0)\\
&-2x (t(x_0),x_0) \left[\dot u(t(x_0),x_0))t(x_0)'+v(t(x_0),x_0))\right]+
2x(t(x_0),x_0) \left[1+t'(x_0)\right]v(t(x_0),x_0))\\ &+x^2(t(x_0),x_0)\frac{d v}{dx_0}(t(x_0),x_0))\\
%&=2ut'-2ut'-2xt'+2xt'+x^22u^2t')(x=x_0)=(x^22u^2t')(x=x_1)<0,
&=
2u(t(x_0),x_0)\left[t(x_0)'+v(t(x_0),x_0\right]-
2\left[1+t'(x_0\right]u(t(x_0),x_0)\\
&-2x (t(x_0),x_0) \left[t(x_0)'+v(t(x_0),x_0))\right]+
2x(t(x_0),x_0) \left[1+t'(x_0)\right]v(t(x_0),x_0))\\ &+x^2(t(x_0),x_0)\frac{d v}{dx_0}(t(x_0),x_0))\\
&=
2u(t(x_0),x_0) v(t(x_0),x_0 )-2u(t(x_0),x_0)-2x (t(x_0),x_0) t(x_0)'+
2x(t(x_0),x_0) t'(x_0)v(t(x_0),x_0))\\ &+x^2(t(x_0),x_0)\frac{d v}{dx_0}(t(x_0),x_0))\\
&=
2[v(t(x_0),x_0)-1] [u(t(x_0),x_0)+x (t(x_0),x_0) t(x_0)']+x^2(t(x_0),x_0)\frac{d v}{dx_0}(t(x_0),x_0))\\
&=
x^2(t(x_0),x_0)\frac{d v}{dx_0}(t(x_0),x_0))
\end{split}
\end{equation}
%
%
%where we have applied $v(t(x_1))=1;t'(x_1)<0; y(t(x_1))=0$ and %system \ref{variacionalsRRicatti}. But that is impossible as %the derivative should be positive.
%
%\end{proof}
We know that $f(x_0)>0$ for small values of $x_0<0$. If at some point $x_1$ we have that $f(x_1)=0$ and $f(x_0)>0$ for  values of $x_1<x_0<0$ then we should have that $f'(x_1)>0$.
But  we have seen that $f(x_1)=0$ only can happen if $v(t(x_1),x_1)<1$ and in this case the previous computation and Lemma \ref{lem:derivadav} gives us that
\[
 f'(x_1)=
x^2(t(x_1),x_1)\frac{d v}{dx_0}(t(x_1),x_1)) <0
\]
which is a contradiction. Therefore, we have seen that $f(x_0)<0$, for any $x_0<0$.
\end{proof}
The last proposition and formula \eqref{theo:variacionals} prove that $\QQQ''(x_0)=t''(x_0) <0$ and therefore  $\tilde \QQQ''(x_0)<0$ for any $-L\le x_0<0$.
This implies, going back to variables $(\eta,u)$ that $\tilde \QQQ_0''(\eta_0)<0$  for any $\eta \in [-M,0)$, and, by \eqref{derivadesQQ} $\QQQ_\eps''(x)<0$, for $x\in [-M\eps^\frac23, -\overline M\eps^\frac23]$.

\subsection{The values of the bifurcation: proof of the last two items of Theorem \ref{thm:main}}\label{sec:last2}
\begin{proof}
In the scope of Theorem \ref{thm:main} and range of $\mu's$ and $\delta's$ in (\ref{rangmu}), we will seek the fixed points of the map $\f_{\mu,\eps}\circ\mathcal{Q}_{\mu,\eps}$ as solutions of the system:

\begin{equation}\label{eq:sysbif}
\begin{array}{rcl}
\mathcal{Q}_{\mu,\eps}(x)&=&(\f_{\mu,\eps})^{-1}(x)
\\
\mathcal{Q}'_{\mu,\eps}(x) &=&((\f_{\mu,\eps})^{-1})'(x)
\\
x\in[\f_{\mu,\eps}(\sqrt{2}\eta_0(0)\eps^{\frac{2}{3}}),\f_{\mu,\eps}(\sigma\rdelta)]&:=&\mathcal{I}^{-1}
  \end{array}
  \end{equation}

  In order to better understanding, through all these section we will denote $\pi_{\delta}$ for $\f_{\mu,\eps}$ and $\mathcal{Q}_{\eps}$ for  $\mathcal{Q}_{\mu,\eps}$.\\

From \ref{eq:formulapiebonathm} it is straightforward to see that
\begin{equation}
\pi^{-1}_{\delta}(x)=\frac{1}{\sqrt{\pi'(0)}}\sqrt{x^2+\delta(\pi'(0)-1)}+\OO(\delta)
\end{equation}
To treat system \ref{eq:sysbif} is better to scale its equations by $\eta=\frac{x}{\eps^{\frac{2}{3}}}$, as we did in \ref{Qescalada}:
  \[
\mathcal{Q}_{\eps}( x )=\eps^{\frac{2}{3}}\tilde{\mathcal{Q}_\eps}(\frac{x}{\eps^{\frac{2}{3}}})=\eps^{\frac{2}{3}}\tilde{\mathcal{Q}_\eps}(\eta)
\]

and we have formulas \ref{derivadesQ}
\[ \quad \tilde{\mathcal{Q}_\eps}(\eta)=\tilde{\mathcal{Q}_0}(\eta)+\OO(\eps^{\frac{1}{3}})
\quad\tilde{\mathcal{Q}_\eps}'(\eta)=
\tilde{\mathcal{Q}_0}'(\eta)+\OO(\eps^{\frac{1}{3}}) \quad\tilde{\mathcal{Q}_\eps}''(\eta)=\tilde{\mathcal{Q}_0}''(\eta)+\OO(\eps^{\frac{1}{3}}), \ \eta\in [-M,-\overline M]
\]

\begin{figure}
\begin{center}
\includegraphics[width=13cm,height=6.75cm]{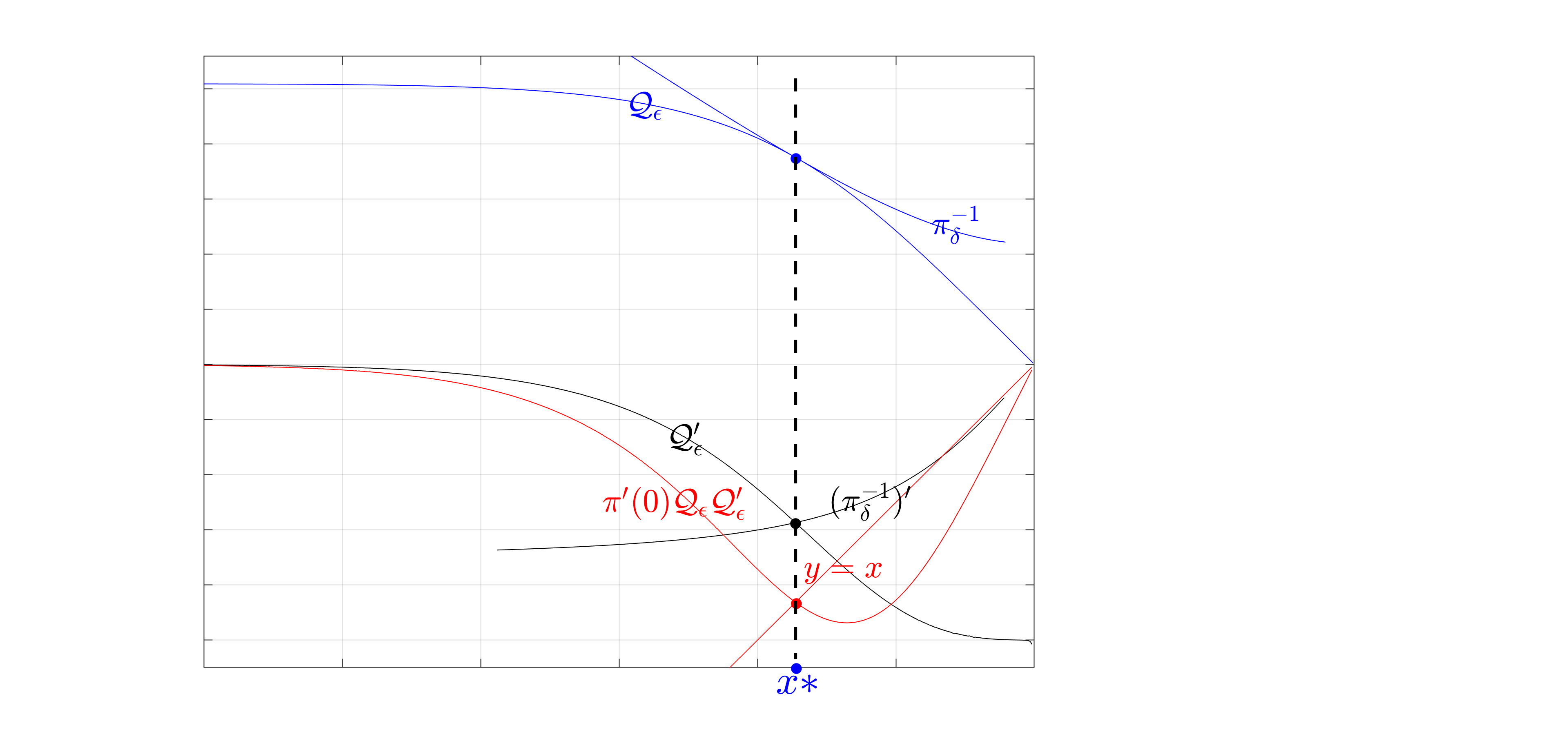}
\caption{Bifurcation value $x_0^{*}$ as solution of system \ref{eq:sysbifsca_0} as well of the equation $\pi'(0)\tilde{\mathcal{Q}_0}(\eta)\tilde{\mathcal{Q}'_0}(\eta)=\eta$}\label{fig:bifvalue}
\end{center}

\end{figure}

To derive similar expressions for $(\pi_{\delta})^{-1}(x)$, we must also scale the $\delta$ parameter by
\begin{equation}\label{deltaescalada}
\tilde{\delta}=\frac{\delta}{\eps^{\frac{4}{3}}}
\end{equation}
and we will have
\begin{equation}
\pi^{-1}_{\delta}(x)=\eps^{\frac{2}{3}}\tilde{\pi}^{-1}_{\frac{{\delta}}{\eps^{\frac{4}{3}}}}(\frac{x}{\eps^{\frac{2}{3}}})=
\eps^{\frac{2}{3}}\tilde{\pi}^{-1}_{\tilde{\delta}}(\eta)
\end{equation}

where $\mathcal{I}^{-1}$ has transformed to
\begin{equation}
\begin{split}
\tilde{\mathcal{I}}^{-1}:=&[\tilde{\pi}_{\tilde{\delta}}(\sqrt{2}\eta_0(0)),\tilde{\pi}_{\tilde{\delta}}(\sigma\sqrt{\tilde{\delta}})]\\
=&[-\sqrt{\tilde{\delta}(1-\pi'(0))+\pi'(0)2\eta^2_0(0)}+\OO(\eps^{\frac{2}{3}}),-\sqrt{1-\pi'(0)(1-\sigma^2)}\sqrt{\tilde{\delta}}+\OO(\eps^{\frac{2}{3}})]
\end{split}
\end{equation}
And now we have
\begin{equation}
\tilde{\pi}^{-1}_{\tilde{\delta}}(\eta)=\frac{1}{\sqrt{\pi'(0)}}\sqrt{\eta^2-(1-\pi'(0))\tilde{\delta}}+\OO(\eps^{\frac{2}{3}})\quad,
K^2_2 \le \tilde{\delta} \le \eta^2_0(0) + K_1\eps^{\frac{1}{3}},\quad \eta\in\tilde{\mathcal{I}}^{-1}
\end{equation}

In despite of its dependence on $\tilde{\delta}$ we also denote by $ \tilde{\pi}^{-1}_{0}(\eta)=\frac{1}{\sqrt{\pi'(0)}}\sqrt{\eta^2-(1-\pi'(0))\tilde{\delta}}$ and we have the formulas

\begin{equation}
 \tilde{\pi}^{-1}_{\tilde{\delta}}(\eta)=\tilde{\pi}^{-1}_{0}(\eta)+\OO(\eps^{\frac{2}{3}})
 \quad(\tilde{\pi}^{-1}_{\tilde{\delta}})'(\eta)=(\tilde{\pi}^{-1}_{0})'(\eta)+\OO(\eps^{\frac{2}{3}})
 \quad(\tilde{\pi}^{-1}_{\tilde{\delta}})''(\eta)=(\tilde{\pi}^{-1}_{0})''(\eta)+\OO(\eps^{\frac{2}{3}}), \eta\in \tilde{\mathcal{I}}^{-1},
\end{equation}

%\begin{equation}
% \mathcal{Q}_{\mu,\eps},(\f_{\mu,\eps})^{-1}:\mathcal{I}^{-1}\subset  \SSS^-_\eps\to \mathcal{I}\subset \SSS_\eps^+
%\end{equation}
then system \ref{eq:sysbif} reads:

\begin{equation}\label{eq:sysbifsca}
\begin{array}{rcl}
\tilde{\mathcal{Q}_0}(\eta)+\OO(\eps^{\frac{1}{3}})=\tilde{\pi}^{-1}_{0}(\eta)+\OO(\eps^{\frac{2}{3}})=
\frac{1}{\sqrt{\pi'(0)}}\sqrt{\eta^2+\tilde{\delta}(\pi'(0)-1)}+\OO(\eps^{\frac{2}{3}})
\\
\tilde{\mathcal{Q}_0}'(\eta)+\OO(\eps^{\frac{1}{3}})=(\tilde{\pi}^{-1}_{0})'(\eta)+\OO(\eps^{\frac{2}{3}})=
\frac{1}{\sqrt{\pi'(0)}}\frac{\eta}{\sqrt{\eta^2+\tilde{\delta}(\pi'(0)-1)}}+\OO(\eps^{\frac{2}{3}})
\end{array}
\end{equation}
Then we can treat this system with implicit function theorem. So we depart from the system

\begin{equation}\label{eq:sysbifsca_0}
\begin{split}
\tilde{\mathcal{Q}_0}(\eta)=\frac{1}{\sqrt{\pi'(0)}}\sqrt{\eta^2+\tilde{\delta}(\pi'(0)-1)}\\
\tilde{\mathcal{Q}'_0}(\eta)=\frac{1}{\sqrt{\pi'(0)}}\frac{\eta}{\sqrt{\eta^2+\tilde{\delta}(\pi'(0)-1)}}
\end{split}
\end{equation}

As $\tilde{\mathcal{Q}_0}$ is concave and $ \frac{1}{\sqrt{\pi'(0)}}\sqrt{\eta^2+\tilde{\delta}(\pi'(0)-1)}$ is convex, in all $\eta<0$, and $\sqrt{\pi'(0)}>1$, system
\ref{eq:sysbifsca_0} has a unique solution $(x^*_0,\delta^*_0)$. Also, from \ref{eq:sysbifsca_0}, $x^*_0$ is the solution of $\pi'(0)\tilde{\mathcal{Q}_0}(\eta)\tilde{\mathcal{Q}'_0}(\eta)=\eta$ (see Figure \ref{fig:bifvalue}). And as $\tilde{\mathcal{Q}''_0}<0$ and $(\frac{1}{\sqrt{\pi'(0)}}\sqrt{\eta^2+\tilde{\delta}(\pi'(0)-1)})''>0$, and $ \frac{\partial}{\partial \tilde{\delta}}(\frac{1}{\sqrt{\pi'(0)}}\sqrt{\eta^2+\tilde{\delta}(\pi'(0)-1)})\neq 0 $, we can apply the implicit function theorem and obtain solutions of system \ref{eq:sysbifsca}for $\eps$ small. Then going back to the original variables we obtain the last two items of Theorem \ref{thm:main}
\end{proof}

\subsection{Proof of Propositions \ref{prop:tildepi} and  \ref{prop:histereticmap}}\label{sec:provapropos}
First we prove Proposition \ref{prop:tildepi}.

Observe that $\Gamma_\al$ is tangent to $\SSS_\al$ at $(0,\al)$.
Therefore, using the map $D$ studied in Lemma \ref{lem:partfinalPe} around this point, we have that, as:
\[
\pi_\al (\al+D(x))= \al -\pi_\al'(\al)\frac{a}{2}x^2+...
= \al -\pi'(0)\frac{a}{2}x^2+...
\]
where the dots indicate terms of higher order in $x$, and therefore:
\[
\tilde \pi (x)= D^{-1}(\pi_\al (\al+D(x))-\al)=
x\sqrt{\pi'(0)}+ ...
\]
which gives
\[
(\tilde \pi)'(0)=\sqrt {\pi'(0)}
\]
\qed

The rest of this section is devoted to prove the Proposition \ref{prop:histereticmap}.
Therefore, from now on in this section, we consider $\mu$ and $\al$ related by \eqref{eq:mu1} and \eqref{eq:sigma1}.

In this case the Poincar\'{e} map $P_\mu$:
\[
P_\mu: [A'_\mu, 0]\times \{y=\al\}\subset \SSS_\al^- \to
\SSS_\al^-
\]
will be constructed as a combination of two maps: an exterior  return map $\f$, and the "interior" hysteretic map $P^h$.
More concretely, consider the maps:
\begin{equation}\label{eq:pitilde}
\begin{split}
{\f} :[E_\mu,0]\times \{y=\al\}  & \to [A'_\mu,0]\times\{y=\alpha\}\\
(x,\al) & \mapsto ({\f} (x),\al)
\end{split}
\end{equation}
defined following the flow of $\X=X_\mu$ until its first cut with $\SSS_\al^-$, and
\begin{equation}\label{eq:phisteretic}
\begin{split}
P^h:[A_\mu',E_\mu]\times\{y=\alpha\}  & \to [A_\mu,0]\times\{y=\alpha\}\\
(x,\al) & \mapsto (P^h(x),\al)
\end{split}
\end{equation}
defined by the hysteretic process determined by the fields $X^{+}=X_\mu$ and $X^{-}=(0,1)$ and computed in the next Lemma.
\begin{lemma}\label{lem:histereticmap}
Take $\mu=\mu_1(\al)$ as  \eqref{eq:mu1} with $\sigma_1$ satisfying \eqref{eq:sigma1}.
Then, the map
\[
\begin{split}
P^h:[A_\mu',E_\mu]\times\{y=\alpha\}  & \to [A_\mu,0]\times\{y=\alpha\}\\
(x,\al) & \mapsto (P^h(x),\al)
\end{split}
\]
is given by
\[
P^h(x)=-\sqrt{-2\al +x^2}+\OO(\al), \ \]
Moreover:
\begin{eqnarray}
A'_\mu &=& -\sqrt{\alpha[\sigma_1((\pi_0)'(0)-1)+(\pi_0)'(0)+1]}+\OO(\alpha)\label{eq:Amuprima}\\
E_\mu &=& -\sqrt{2\al }+\OO(\al), \ P^h(E_\mu) =0, \ \lim_{x\to E_\mu^-}P_\mu (x)=0 \label{eq:Emu}\\
A_\mu & = & P^h (A'_\mu )= -\sqrt{-2\al +(A'_\mu)^2}+\OO(\al)
= -\sqrt{\al(\sigma_1+1)((\pi_0)'(0)+1)}+\OO(\al) \label{eq:Amu}
\end{eqnarray}
\end{lemma}
\begin{proof}
Observe that the map $P^h (x_1)=x_2$ if the flow by $\X_\mu$ through $(x_1,\al)$ intersects $y=-\al$ at $(x_2,-\al)$ (recall that the vector field $\Y=(0,1)^T$).
We will compute the point $x_2$ in two steps.
\begin{itemize}
 \item
First we compute the point $(0,y_1)$ where the orbit through $(x_1,\al)$ intersects $x=0$.
\item
Second we will compute the point $(x_2,-\al)$ where the backward orbit through $(0,y_1)$ intersects $y=-\al$.
\end{itemize}
To compute these points we will focus on the tangency point $(0,-\al)$.
Through the change $\bar y=y+\al$ this tangency becomes $(0,0)$.

The point $(0,\bar y_1)$, with $\bar y_1=y_1+\al$, will be given by $\bar y_1=g(x_1)$ where the map $g$ is given in Proposition \ref{prop:exterior} with $\de=2\al$, and therefore:
\[
\bar y_1= g(x_1)=2\al-x_1^2 +\OO(\al^\frac32)
\]
Analogously, the point $(x_2,0)$, will be given by $D(x_2)=\bar y_1$, where $D$ is the map given in  Lemma \ref{lem:partfinalPe}, and therefore:
\[
\bar y_1= D(x_2)=-(\frac{a}{2}+\OO(\al))x_2^2+\OO(x_2^3)
\]
note that in our case (see \eqref{def:Xg}) $a=2$ and therefore
equalizing the two formulas
\[
2\al-x_1^2 +\OO(\al^\frac32)=(1+\OO(\al)) x_2^2+\OO(x_2^3)
\]
which gives
\[
 x_2=P^h(x_1)=-\sqrt{-2\al +x_1^2}+\OO(\al)
\]
Finally, as the point $E_\mu$ satisfies that $P^h(E_\mu)=0$, we have that
\[
E_\mu =-\sqrt{2\al}+\OO(\al).
\]
Observe that using the Poincar\'{e} map $\pi_\mu$ and the map $g^{-1}$ we can compute the point $A_\mu'$:
\[
\begin{split}
\pi_\mu (-\al) &=\pi_\mu (\mu)+(\pi_\mu)'(\mu)(-\al-\mu)+
\OO((-\al-\mu)^2)\\
&=\mu +(\pi_0)'(0)(-\al-\mu)+\OO((-\al-\mu)^2)=
-\al[\sigma_1((\pi_0)'(0)-1)+(\pi_0)'(0)]+\OO(\al^2)<-3\al\\
A'_\mu &= g^{-1}(\pi_\mu(-\al)+\al)=-\sqrt{\al -\pi_\mu(-\al)}+\OO(\al)\\
&=-\sqrt{\alpha[\sigma_1((\pi_0)'(0)-1)+(\pi_0)'(0)+1]}+\OO(\alpha)
\end{split}
\]
Finally, the point $(A_\mu,\al)$ is given by:
\[
\begin{split}
A_\mu& = P^h (A'_\mu )= -\sqrt{-2\al +(A'_\mu)^2}+\OO(\al)\\
&=-\sqrt{-2\al +\alpha\left[\sigma_1((\pi_0)'(0)-1)+(\pi_0)'(0)+1\right]}+\OO(\al)\\
&=
 -\sqrt{\al(\sigma_1+1)((\pi_0)'(0)+1)}+\OO(\al)
 \end{split}
\]
\end{proof}

To compute the image of the points in $[E_\mu,0]$, fist we need to compute the map $\tilde \pi$ in the following lemma:
\begin{lemma}\label{lem:exterior}
Take $\mu=\mu_1(\al)$ as  \eqref{eq:mu1} with $\sigma_1$ satisfying \eqref{eq:sigma1}.
Then,
the map
\[
\tilde \pi: [E_\mu,0]\times \{y=\al\}
\to [A'_\mu, B_\mu]\times \{y=\al\}\subset \SSS^\al_-
\]
where $E_\mu $ and $A'_\mu$ are given in \eqref{eq:Emu} and \eqref{eq:Amuprima} respectively and $B_\mu=\tilde \pi(0)$, is given by
\[
\begin{split}
\tilde \pi(x)&=-\sqrt{(\al-\mu)(1-(\pi')(0))+(\pi')(0)x^2(1+\OO(\sqrt{\al})}+\OO(\al)\\
&=-\sqrt{\al(\sigma_1-1)( (\pi')(0)-1)+(\pi')(0)x^2(1+\OO(\sqrt{\al})}+\OO(\al)
\end{split}
 \]
and $B_\mu$, satisfies $A_\mu < B_\mu<E_\mu$.
\end{lemma}
\begin{proof}
Observe that $\tilde \pi=D^{-1}\circ \pi_\mu \circ D $ where $D(x)$ is the map associated to the fold at $(\al,0)$ and is given through the formulas given in  Lemma \ref{lem:partfinalPe} after the change $\bar y=y-\al$:
\[
\begin{split}
D(x) &= \al-(1+\OO(\al)) x^2+\OO(x^3)\\
D^{-1}(y) &=-\sqrt{\al-y}+\OO(\al\sqrt{\al-y}, \al-y)
\end{split}
 \]
 Therefore, as $x\in [E_\mu,0]$ satisfy $x=\OO(\sqrt{\al})$, we have:
 \[
  D(x)=\al-x^2(1+\OO(\sqrt{\al}))=\OO(\al)
 \]
therefore, using the Taylor expansion of $\pi_\mu$ around $y=\mu$:
\[
\begin{split}
 \pi_\mu(D(x)) &=\mu +(\pi')(0)(D(x)-\mu)+\OO((D(x)-\mu)^2) \\
&=
\mu +(\pi')(0)\left(\al-x^2(1+\OO(\sqrt{\al}))-\mu\right)+\OO(\al^2)=\OO(\al)
\end{split}
\]
Finally:
\[
\begin{split}
\tilde \pi(x) &=-\sqrt{\al-\left[\mu +(\pi')(0)\left(\al-x^2(1+\OO(\sqrt{\al}))-\mu\right)\right]}+\OO(\al)\\
&=-\sqrt{(\al-\mu)(1- (\pi')(0))+(\pi')(0)x^2(1+\OO(\sqrt{\al})}+\OO(\al)\\
&=-\sqrt{\al(\sigma_1-1)( (\pi')(0)-1)+(\pi')(0)x^2(1+\OO(\sqrt{\al})}+\OO(\al)
\end{split}
\]
which gives, using that $ B_\mu= \tilde \pi(0)$:
\begin{equation}\label{eq:Bmu}
B_\mu
= -\sqrt{\al(\sigma_1-1)( (\pi'_0)(0)-1)}+\OO(\al)
\end{equation}
using the expression of $A_\mu$ given in \eqref{eq:Amu} one easyly  gets that $A_\mu <B_\mu$ and
using the definition of $\sigma_1$ in \eqref{eq:mu1} one gets $B_\mu<E_\mu$ were $E_\mu$ is given in \eqref{eq:Emu}.
\end{proof}
With Lemmas \ref{lem:histereticmap} and \ref{lem:exterior} we can prove Proposition \ref{prop:histereticmap}.
\begin{proof}
 The formulas of the Poincar\'{e} map in the different intervals are a direct consequence of the formulas for $P^h$ and $\tilde \pi$ given in lemmas \ref{lem:histereticmap} and \ref{lem:exterior}.

As $P_\mu(A_\mu)=P^h(A_\mu)$ clearly $A_\mu <P_\mu(A_\mu)$.
The second inequality is just a calculation using the formulas for $P_\mu(A_\mu)$ and $P_\mu(0)$ and using that $\sigma_1>0$ and $(\pi)'(0)-1>0$.

The inequality $P_\mu(0)<E_\mu$ is satisfied provided:
\[
\sigma_1((\pi)'(0)-1)>3+(\pi)'(0)
\]
but the condition for $\sigma_1$ in \eqref{eq:sigma2} and the fact that $(\pi)'(0)>1$ implies this inequality because:
\[
\sigma_1((\pi)'(0)-1)>2(\pi)'(0)+2>(\pi)'(0)+3
\]
Clearly the derivative in positive everywhere except at $x=0$.
\end{proof}

\bigskip

\textbf{Acknowledgments}

The authors have been partly supported by the Spanish MINECO-FEDER Grant
PGC2018-098676-B-100 (AEI/FEDER/UE).  T. M. S. is supported by the Catalan Institution for 
Research and Advanced Studies via an ICREA
Academia Prize 2019.

%
%\bibliography{nonsmooth}
\bibliographystyle{plain}

\end{document}